\documentclass{amsproc}

\theoremstyle{plain} 

\newtheorem{theorem}{Theorem}[section]
\newtheorem{lemma}[theorem]{Lemma}
\newtheorem{corollary}[theorem]{Corollary}

\newtheorem{definition}[theorem]{Definition}
\newtheorem{proposition}[theorem]{Proposition}

\newtheorem{remark}[theorem]{Remark}

\usepackage{amsmath}
\usepackage{amsthm}
\usepackage{amscd}

\numberwithin{equation}{section}

\newcommand{\cc}{{\mathbb C}}
\newcommand{\pp}{{\mathbb P}}
\newcommand{\rr}{{\mathbb R}}

\newcommand{\zz}{{\mathbb Z}}
\newcommand{\nn}{{\mathbb N}}

\newcommand{\ggg}{{\mathbb G}}
\newcommand{\aaa}{{\mathbb A}}
\newcommand{\Gm}{\ggg _m}

\newcommand{\Hh}{{\mathcal H}}
\newcommand{\Pp}{{\mathcal P}}
\newcommand{\Qq}{{\mathcal Q}}

\newcommand{\Oo}{{\mathcal O}}

\newcommand{\Mm}{{\mathcal M}}
\newcommand{\Nn}{{\mathcal N}}

\newcommand{\stackquot}{{/\;\!\!\! /}}

\newcommand{\res}{{\rm res}}
\newcommand{\hdot}{{\bullet}}

\begin{document}

\author[F. Loray]{Frank Loray}
\address{IRMAR, Campus de Beaulieu\\
Universit\'e de Rennes I\\
35042 Rennes Cedex, France}
\email{frank.loray@univ-rennes1.fr}

\author[M.-H. Saito]{Masa-Hiko Saito}
\address{Department of Mathematics, Faculty of Science\\
Kobe University\\
Kobe, Rokko, 657-8501, Japan}
\email{mhsaito@math.kobe-u.ac.jp}

\author[C. Simpson]{Carlos Simpson}
\address{CNRS, Laboratoire J. A. Dieudonn\'e, UMR 6621
\\ Universit\'e de Nice-Sophia Antipolis\\
06108 Nice, Cedex 2, France}
\email{carlos@math.unice.fr}

\thanks{Supported in part by JSPS Grant-in-Aid for Scientific Research (S)19104002
and ANR grant BLAN 08-1-309225 (SEDIGA)}

\title[Foliations on the moduli space of connections]{Foliations on the moduli space of
rank two connections on the projective line minus four points}



\begin{abstract}
We look at natural foliations 
on the Painlev\'e VI moduli space of regular connections of rank $2$
on $\pp ^1 -\{ t_1,t_2,t_3,t_4\}$. These foliations are fibrations,
and are interpreted in terms of the nonabelian Hodge filtration, giving a 
proof of the nonabelian Hodge foliation conjecture in this case. Two basic 
kinds of fibrations arise: from apparent singularities, and from  
quasiparabolic bundles. We show that these are transverse. 
Okamoto's additional symmetry, which may be seen as Katz's
middle convolution, exchanges the quasiparabolic and 
apparent-singularity foliations.  
\end{abstract}

\maketitle

\section{Introduction} \label{sec-introduction}

The Painlev\'e VI equation is the isomonodromic deformation equation for
systems of differential equations of
rank $2$ on $\pp ^1$ with four logarithmic singularities over $D:= \{ t_1,t_2,t_3,t_4\}$. 
Such a system 
of differential equations is encoded in a vector bundle with logarithmic
connection $(E,\nabla )$, where $E$ is a vector bundle on $X=\pp ^1$ and
$\nabla : E\rightarrow E\otimes \Omega ^1_X(\log D)$ is a first order algebraic
differential operator satisfying the Leibniz rule of a connection.  At a singular
point $t_i$ the residue of $\nabla$ is a linear endomorphism of $E_{t_i}$.
The ``space of initial conditions for Painlev\'e VI''  is the moduli space of
$(E,\nabla )$ such that the residues $\res (\nabla , t_i)$ lie in fixed conjugacy
classes.  The conjugacy class information is denoted ${\bf r}$, which for us will just
mean fixing two distinct eigenvalues $r^{\pm}_i$ at each point. 
The isomonodromic evolution equation concerns what happens when the $t_i$ move.
However, in this paper we consider only the moduli space so the $t_i$ are fixed.

The associated moduli stack is denoted by $\Mm ^d({\bf r})$. For generic choices of
${\bf r}$, all connections are irreducible and the moduli stack is a $\Gm$-gerb
over the moduli space $M^d({\bf r})$. Here $d$ denotes the degree of the bundle
$E$, related to ${\bf r}$ by the Fuchs relation \eqref{fuchsrel}. We usually
assume that $d$ is odd, essentially equivalent to $d=1$, because any
bundle of degree $1$ having an irreducible connection must be of the form 
$B=\Oo \oplus \Oo (1)$. This facilitates the consideration of the parameter space
for quasiparabolic structures. 

The object of this paper is to study several natural
fibrations on the moduli space. The second author, with Inaba and Iwasaki,
have described the structure of $M^d({\bf r})$ as obtained by several blow-ups of a
ruled surface over $\pp ^1$ in \cite{InabaIwasakiSaito,InabaIwasakiSaito2}. The function to $\pp ^1$ may be viewed as
given by the position of an apparent singularity, considered also by Szabo \cite{Szabo}
and Aidan \cite{Aidan}. The first author has considered this fibration too but also looked
at the function from
$M^d({\bf r})$ to the space of quasiparabolic bundles \cite{Loray}, 
which as it turns out is again
$\pp ^1$ or  more precisely a non-separated scheme which had been introduced by Arinkin
\cite{Arinkin}. 
The third author has defined a
decomposition of $M^d({\bf r})$ obtained by looking at the limit of $(E,u\nabla )$ as
$u\rightarrow 0$ into the moduli space of semistable parabolic Higgs bundles \cite{idsm}. 

We compare these pictures by examining precisely the condition of stability depending
on parabolic weight parameters. 
A choice of one of the two residues $r^-_i$
is made at each point, and the eigenspace provides a $1$-dimensional subspace 
$P_i\subset E_{t_i}$. The collection $(E,P_{\hdot})$ is a quasiparabolic bundle \cite{Seshadri}. Given that $E\cong B=\Oo \oplus \Oo (1)$, we can write down a parameter
space for all quasiparabolic structures on $B$. The moduli stack for such quasiparabolic
bundles is the stack quotient by $A=Aut(B)$. 

Specifying two parabolic weights $\alpha ^{\pm}_i$ at each 
point\footnote{Here the smaller weight $\alpha _i^-$ is associated to the
subspace $P_i$, which 
may be different from the convention used in some other papers.}
transforms the
quasi\-para\-bolic structures into parabolic ones for which there is a notion of stability. 
There is a collection of $8$ inequalities 
concerning the parabolic weights appearing in Proposition \ref{unstablezones}: (a), (b) and $6$ of type (c),
see also \eqref{staba} \eqref{stabb} \eqref{stabc}. Depending on these inequalities, generically
the underlying parabolic bundle 
will either be stable, or unstable. The space of
parabolic weights is therefore divided up into a stable zone, and $8$ unstable zones.

The different unstable zones are permuted by the operation of performing two elementary
transformations. Doing two at a time 
keeps the condition that the underlying bundle has odd degree.
Up to these permutations, we can assume that we are in the (a)-unstable
zone $\epsilon _1+\epsilon _2+\epsilon_3 + \epsilon _4 < 1/2$ where $\epsilon _i=(\alpha ^+_i-\alpha ^-_i)/2$. 
In this case, the subbundle $\Oo (1)\subset E=\Oo \oplus \Oo (1)$
is destabilizing. It determines an apparent singularity, which is the unique point
at which the subbundle osculates to the direction of the connection. The position
of this apparent singularity gives the map to $\pp ^1$. We point out in
Theorem \ref{unstabletheorem} that, in this
unstable zone, this map is the same as the map taking $(E,\nabla )$ to the limiting
$\alpha$-stable Higgs bundle. This furnishes the comparison between the
Higgs limit decomposition, and the fibration of \cite{InabaIwasakiSaito,InabaIwasakiSaito2}.

This comparison allows us to prove the foliation conjecture of \cite{idsm}
in this case. The Higgs limit decomposition is, from the definition, just a decomposition
of the moduli space into disjoint locally closed subvarieties, which are Lagrangian
for the natural symplectic structure. The foliation conjecture posits that this decomposition
should be a foliation in the case when the moduli space is smooth. For the unstable zone,
the decomposition is just the collection of connected components of the fibers of the
smooth morphism of \cite{InabaIwasakiSaito,InabaIwasakiSaito2} to $\pp ^1$, so it is a foliation. 

We next turn our attention to the stable zone. The quasiparabolic bundles which support an irreducible connection with given residues are exactly the simple ones, and
the quotient of the set of simple quasiparabolic structures by the automorphism
group is the non-separated scheme $\Pp$ which is like $\pp ^1$ but has 
two copies of each $t_i$.
This is the same as the space of leaves in the fibration corresponding to the unstable zone.
It has also appeared in Arinkin's work \cite{Arinkin} 
on the geometric Langlands program.

In the stable zone, the limit $\lim _{u\rightarrow 0}(E,u\nabla , P_{\hdot})$ in the moduli space of $\alpha$-stable parabolic Higgs bundles is just the 
underlying parabolic bundle $(E,P_{\hdot})$, except at one from each pair of 
points lying over $t_i$. Thus, 
Theorem \ref{stabletheorem} says that in the stable zone, the Higgs  limit 
decomposition is just the decomposition into fibers of the projection 
$M^d({\bf r})\rightarrow \Pp$ considered in \cite{Loray},
sending $(E,\nabla , P_{\hdot})$ to $(E,P_{\hdot})$. 
As before, this interpretation 
allows us to prove the foliation conjecture of \cite{idsm} in this case.

Putting these together, we obtain a proof of the foliation conjecture
for the moduli space of parabolic logarithmic connections of rank $2$ on
$\pp ^1-\{ t_1,t_2,t_3,t_4\}$ with any generic residues and any generic
parabolic weights. The genericity condition is non-resonance plus
a natural condition which has been
introduced by Kostov, ruling out the possibility of 
reducible connections. The combination of these two conditions will be called
``nonspeciality''. 

In Section \ref{sec-dm} we point out that this discussion gives the same results for
the case of local systems on a root stack. 
These
correspond to parabolic logarithmic connections on $\pp ^1$ whose
residues and weights are the same rational numbers. In the root stack interpretation,
the Higgs limit decomposition may
be tied back to the same thing on a compact curve, a cyclic covering of $\pp ^1$ branched over $t_1,t_2,t_3,t_4$. 

In Section \ref{sec-transverse} we show that the two different kinds of fibrations,
obtained from apparent singularities and from the quasiparabolic structure, are
strongly transverse: generic fibers intersect once. A similar picture has been described by Arinkin and Lysenko \cite{ArinkinLysenko2}
when we switch to trace-free connections (and $\deg(E)=0$). 

In Section \ref{sec-okamoto} we recall the additional Okamoto symmetry, and
the fact pointed out by the first author in \cite{Loray} that it interchanges the two
different types of
fibrations considered above. The geometrical picture was also investigated in
\cite{ArinkinLysenko2}.
Then in Section \ref{sec-mc}, we propose a possible explanation
by interpreting Okamoto's additional symmetry as Katz middle convolution.
This interpretation is now well known, apparently first pointed out by
Dettweiler and Reiter \cite{DettweilerReiterPainleve}, see also Boalch \cite{Boalch}
\cite{BoalchQuiver}, and Crawley-Boevey \cite{CrawleyBoevey}.

We calculate,
concentrating on the case of finite order monodromy, that a middle convolution with
suitably chosen rank $1$ local system interchanges
the stable and unstable zones.  Assuming a compatibility of
higher direct images which is not yet proven, the middle convolution will preserve the
Higgs limit decomposition and this property would imply that it permutes the two different
kinds of foliations. 

As a part of the 
numerous ongoing investigations of the rich structure of these moduli spaces,
the present discussion points out the role of the different regions in
the space of parabolic weights. Nevertheless,  
a number of further questions remain open in this direction, 
such as what happens along the hyperplanes of special values of residues
and/or parabolic weights. We hope to address these in the future. 

Each of us would like to thank the numerous
colleagues with whom we have discussed these questions. 
The second author would like to thank other authors for their hospitality
during his stays in Nice and Rennes.


\section{Moduli stacks of parabolic logarithmic $\lambda$-connections}

Let $X:= \pp^1$, with a divisor consisting of four distinct
points $D:= \{ t_1,t_2,t_3,t_4\}$, and put $U:= X-D$. Let $\Mm ^d\rightarrow \aaa ^1$ denote the moduli
stack \cite{InabaIwasakiSaito,InabaIwasakiSaito2}
of logarithmic $\lambda$-connections of rank two and degree $d$ with quasiparabolic
structure on $(X,D)$. For a scheme 
$S$, an object of $\Mm (S)$ is a quadruple $(\lambda , E, \nabla ,P_{\hdot})$ where
$\lambda : S\rightarrow \aaa ^1$, $E$ is a rank $2$ vector bundle on $X\times S$
of degree $d$ on the fibers $X\times \{ s\}$,
$P_{\hdot}= (P_1,P_2,P_3,P_4)$ is a collection of rank $1$ subbundles
$$
P_i \subset E|_{\{ t_i\} \times S},
$$
and 
$$
\nabla: E\rightarrow E\otimes _{\Oo _{X\times S}}\Omega ^1_{X\times S/S}(\log D\times S),
$$ is a logarithmic $\lambda$-connection on $X\times S/S$ preserving $P_i$.
This means that $\nabla$ is
a map of sheaves satisfying $\nabla (ae)=a\nabla (e)+\lambda d(a)e$, inducing a residue 
endomorphism 
$$
\res (\nabla , t_i):E_{\{ t_i\} \times S} \rightarrow E_{\{ t_i\} \times S}
$$
which is required to preserve $P_i$. 
The
groupoid $\Mm ^d(S)$ has these objects, and morphisms are isomorphisms of bundles
with $\lambda$-connection preserving the quasiparabolic structure.

For $\lambda \in \aaa ^1$ let $\Mm ^d_{\lambda}$ denote the fiber of 
$\Mm ^d\rightarrow \aaa ^1$ over $\lambda$. For $\lambda =1$ it is the
moduli stack of logarithmic connections, and the fibers are
all the same for $\lambda \neq 0$.  For $\lambda = 0$ it is the moduli stack
of Higgs bundles. In both cases, quasiparabolic structures are included. 

The value of $\lambda$ is determined by $\nabla$, so it may be left out
of the notation, writing if necessary $\lambda = \lambda (\nabla )$. 

Given a point $( E, \nabla ,P_{\hdot})\in \Mm ^d(S)$, we get two residue
eigenvalues:
\begin{itemize}
\item $\res ^-_i(E,\nabla )$ is the scalar by which $\res (\nabla , t_i)$
acts on $P_i$, and 
\item $\res ^+_i(E,\nabla )$ is the scalar by which $\res (\nabla ,t_i)$ acts
on $E_{t_i}/P_i$. 
\end{itemize}

These satisfy the Fuchs relation
\begin{equation}
\label{fuchsrel}
\sum _{i=1}^4 (\res ^+_i(E,\nabla ) + \res ^-_i(E,\nabla )) + \lambda \deg (E)= 0.
\end{equation}

Let $\Nn ^d\rightarrow \aaa ^1$ be the bundle of possible residual data satisfying the 
Fuchs relation for $\deg (E)=d$, so 
$$
\Nn ^d = \{ (\lambda , r_1^+ , r_1^-,\ldots , r^+_4, r_4^-)|r_1^+ + \cdots + r^-_4 +\lambda d =0\} .
$$
The residues give a map 
$$
\Psi : \Mm ^d\rightarrow \Nn ^d 
$$
relative to $\aaa ^1$. 
If ${\bf r}(\lambda ): \aaa ^1\rightarrow \Nn ^d$ is a section denoted
$\lambda \mapsto (\lambda , r^+_1(\lambda ),\ldots , r^-_4(\lambda ))$, let
$\Mm ^d({\bf r}(\lambda ))$ be the pullback of this section in $\Mm ^d$. 
It is the moduli stack of $(E,\nabla , P_{\hdot})$ such that
the eigenvalue of the residue of $\nabla$ acting on $E_{t_i}/P_i$ (resp. $P_i$) is
$r_i^+(\lambda (\nabla ))$ (resp. $r_i^-(\lambda (\nabla ))$) for $i=1,2,3,4$. 

Note that in \cite[\S 2.2]{InabaIwasakiSaito} the notation is slightly different:
the parameter we call $\lambda$ here 
is replaced by $\phi$ but which has a somewhat more general
meaning, and the residues are denoted there by $\lambda _i$ which correspond to our $r^-_i$.
In \cite{InabaIwasakiSaito} it is assumed that $r^-_i+r^+_i\in \zz$ but that 
normalization doesn't make for any loss of generality.

Suppose $r_i^+\neq r_i^-$ for $i=1,\ldots , 4$. Then $\Mm _1^d({\bf r})$
may also be viewed as the moduli stack of logarithmic connections $(E,\nabla )$
with $\deg (E)=d$ and such that the eigenvalues of $\res (\nabla , t_i)$
are $r_i^{\pm}$, but without specifying $P_{\hdot}$. 
The eigenvalue condition is a closed condition, just saying that
$$\begin{matrix}
{\rm Tr}(\res (\nabla , t_i)) &=& r^+_i+r^-_i,\\
{\rm det}(\res (\nabla , t_i))&=& r^+_ir^-_i.
\end{matrix}
$$
Because of the hypothesis that the eigenvalues are distinct, the rank one subspace
$P_i\subset E_{t_i}$ is uniquely determined as the $r^-_i$-eigenspace of 
$\res (\nabla , t_i)$. 

Let $\Mm _1^d({\bf r})^{\rm irr}\subset \Mm _1^d({\bf r})$ be the open substack
parametrizing irreducible connections. It is a $\Gm$-gerb 
over its 
coarse moduli space
$$
\Mm _1^d({\bf r})^{\rm irr}\rightarrow M _1^d({\bf r})^{\rm irr}.
$$
The group $\Gm$ acts on $\Mm^d $ by 
$$
u:(\lambda , E,\nabla ,P_{\hdot})\mapsto (u\lambda , E, u\nabla ,P_{\hdot}).
$$
It is compatible with the standard action on the $\lambda$-line $\aaa ^1$. 
The action on the residues is 
$$
\res _i^{\pm}(E,u\nabla ) = u\res _i^{\pm}(E,\nabla ).
$$
Therefore if ${\bf r}(\lambda )= \lambda {\bf r}$ is a section
such that $r_i^{\pm}(\lambda )=\lambda r_i^{\pm}$ then the action restricts
to an action on $\Mm^d (\lambda {\bf r})$.

Over the open set $\lambda \neq 0$ the Artin stack $\Mm ^d$ is of finite type, but
if $\lambda = 0$ is included then it is only locally of finite type, since the
collection of Higgs bundles of degree $d$ with no semistability condition is unbounded. 
 
Introducing parabolic weights allows
us to consider a semistability condition \cite{InabaIwasakiSaito,InabaIwasakiSaito2}, 
but is also motivated by the growth rates
of harmonic metrics \cite{hbnc}. A vector of parabolic weights denoted $\alpha$ is
a collection of real numbers
$$
\alpha = (\alpha ^-_1,\alpha ^+_1, \alpha ^-_2,\alpha ^+_2, \alpha ^-_3,\alpha ^+_3,
\alpha ^-_4,\alpha ^+_4)
$$
with 
$$
\alpha ^-_i\leq \alpha ^+ _i\leq \alpha ^-_i+1 .
$$
Notice that we don't require that these lie in any particular interval, in fact it
will be convenient to choose different intervals for different points $t_i$ sometimes.

This phenomenon, which goes back to Manin's comments figuring in \cite{Deligne}, 
is related to Mochizuki's notation  \cite{Mochizuki} ${}_cE$ for a parabolic 
structure based at a real number $c$.
A given parabolic sheaf $E_{\hdot}$
in a neighborhood of $t_i$ according to the definitions of \cite{hbnc,MaruyamaYokogawa}, 
will yield a weighted parabolic bundle $({}_{c_i}E,P_i, \alpha _i^{\pm})$
in the present (and original \cite{Seshadri}) sense,
for each choice of $c_i \in \rr$. The parabolic
weights $\alpha _i^{\pm}$ are the 
weights of $E_{\hdot}$ which are contained in the
interval $(c-1,c]$. In the other direction,
a given $(E,P_i,\alpha _i^{\pm})$ as we are considering
here, will come  
from a unique parabolic sheaf $E_{\hdot}$ by the construction ${}_{c_i}E$
using any choice of cutoff number $c_i$ 
between $\alpha _i^+$ and $\alpha ^-_i+1$. 
Since the choice of $c_i$ doesn't 
have any effect for most of our considerations, we leave it out of the notation. 

We use the convention here that smaller weights are associated to
subsheaves or subspaces in the filtration. This is the convention which
was used for example in \cite{Mochizuki} and \cite{IyerSimpson}, but is opposite in sign to some older conventions. So, here the weight 
$\alpha _i^-$ is associated to the subspace $P_i$ and
$\alpha _i^+$ is associated to $E_{t_i}/P_i$. 

If $\alpha$ is a choice of weights, define
$$
\deg ^{\rm par}(E,\nabla , P_{\hdot}):= \deg (E) - \sum _{i=1}^4 (\alpha ^+_i+\alpha ^-_i).
$$
If $F\subset E$ is a rank one subbundle preserved by $\nabla$, let $\sigma (i,F)$
be either $-$, if $F_{t_i}\subset P_i$, or $+$ otherwise. Then put 
$$
\deg ^{\rm par}(F):= \deg (F) - \sum _{i=1}^4 \alpha _i^{\sigma (i,F)}.
$$
Say that $(E,\nabla , P_{\hdot})$ is $\alpha$-semistable if, for any rank one
subbundle preserved by $\nabla$ we have
$$
\deg ^{\rm par}(F)\leq \frac{\deg ^{\rm par}(E,\nabla , P_{\hdot})}{2}; 
$$
say that it is $\alpha$-stable if the strict inequality $<$ always holds.
These stability and semistability conditions are open on $\Mm$, and let
$$
\Mm ^{d,\alpha}\subset \Mm ^d, \;\;\; \Mm ^{d,\alpha}({\bf r}(\lambda ))\subset \Mm ^d
({\bf r}(\lambda ))
$$
be the open substacks of $\alpha$-semistable points. As usual denote by a subscript the fiber
over $\lambda \in \aaa ^1$. 

By geometric invariant theory \cite{InabaIwasakiSaito}
there is a universal categorical coarse moduli space
$$
\Mm ^{d,\alpha} \rightarrow M^{d,\alpha}
$$
where $M^{d,\alpha}$ is a quasiprojective variety. 
This  induces on the closed substack a universal categorical quotient
$$
\Mm ^{d,\alpha}({\bf r}(\lambda )) \rightarrow M^{d,\alpha} ({\bf r}(\lambda ))
$$
where $M^{d,\alpha}({\bf r}(\lambda ))$ is also the closed subscheme 
of $M^{d,\alpha}$ defined as the inverse image of the same section ${\bf r}$
under the morphism
$$
M^{d,\alpha}\rightarrow \Nn ^d
$$
which exists by the categorical quotient property.

These moduli spaces are constructed by Inaba, Iwasaki and the second author \cite{InabaIwasakiSaito,InabaIwasakiSaito2}, see
also Nitsure \cite{Nitsure} for plain logarithmic connections which
can be viewed as the case $\alpha ^+_i=\alpha ^-_i$, 
Maruyama-Yokogawa \cite{MaruyamaYokogawa} for parabolic bundles, 
Konno \cite{Konno}, Boden-Yokogawa \cite{BodenYokogawa}, Nakajima \cite{Nakajima},
Schmitt \cite{Schmitt}
and others for
parabolic Higgs bundles, 
and the papers of Arinkin and Lysenko \cite{Arinkin,ArinkinLysenko1,ArinkinLysenko2} as well as
following papers such as Oblezin \cite{Oblezin},
which treat explicitly the rank two case we are considering here. 

The space of initial conditions of Painlev\'e VI was first introduced in
\cite{Okamoto1} by blowing up of rational surfaces along accessible
singularities of Painlev\'e VI equations. More geometric or deformation
theoretic descriptions of Okamoto spaces of initial conditions are given
by Sakai \cite{Sakai} and by Saito-Takebe-Terajima
\cite{SaitoTakebeTerajima}.  We note that one can identify Okamoto spaces
of initial conditions or their natural compactifications, 
Okamoto-Painlev\'e pairs, in \cite{SaitoTakebeTerajima} with the moduli spaces
of $\alpha$-stable parabolic connections (see
\cite[Theorem 4.1]{InabaIwasakiSaito2}).

The global family of rank $2$ stable parabolic connections
over the space of local exponents
constructed
in \cite{InabaIwasakiSaito} really depends on the choice of stability
condition from the choice of parabolic weights.  However if the
local exponents are Kostov-generic, all connections are
irreducible, so stability does not depends on weights.
Even in this case,
if the local exponents are resonant,
then the fiber of $M^d \rightarrow \Nn ^d$ over that point ${\bf r}$
is independent of the parabolic weights, 
but the total family of connections
are not biregular isomorphic near the nighborhood of
the fiber, rather a flop phenomenon occurs.

The elementary transformation at the point $t_i$,
may be defined as follows, see \cite[\S 3]{InabaIwasakiSaito}. 
\label{etpage}
Set $\widetilde{E}:= \ker (E\rightarrow E_{t_i}/P_i)$, let $\widetilde{\nabla }$ be the
induced $\lambda$-connection, and put $\widetilde{P}_j=P_j$ for $j\neq i$ whereas
$\widetilde{P}_i:= (E_{t_i}/P_i)(-t_i)$ in the exact sequence
$$
0\rightarrow \widetilde{P}_i \rightarrow \widetilde{E}_{t_i}\rightarrow P_i\rightarrow 0.
$$
Then 
$$
\varepsilon _i(E,\nabla , P_{\hdot}):= (\widetilde{E},\widetilde{\nabla },\widetilde{P}_{\hdot}).
$$
Note that $\deg (\widetilde{E})=\deg (E)-1$ so
$$
\varepsilon _i : \Mm ^d\rightarrow \Mm ^{d-1}.
$$
In terms of the notations of \cite{InabaIwasakiSaito} we have
$\widetilde{E}= {\rm Elm}^+_{t_i}(E)\otimes \Oo (-t_i)$.

These transformations are some of the ``B\"acklund transformations''
in the classical theory of Painlev\'e VI and Garnier
equations \cite{InabaIwasakiSaito}, and for more general systems they are called
``Gabber transformations'' by Esnault and Viehweg \cite{EsnaultViehweg},
see also Machu \cite{Machu1}. 

Suppose $r^{\pm}_i(E)$ are the residues of $\nabla$ at $t_i$.
A section of $\widetilde{E}$ projecting into $\widetilde{P}_i$ is of the form $ze$
for $e$ a section of $E$ projecting to something nonzero modulo $P_i$, 
and $z$ a coordinate at $t_i$. We can assume that $\nabla (e)= r^+_i(E)e\cdot d\log z$,
in which case
$$
\nabla (ze) = z\nabla (e) + \lambda e\cdot dz = (r^+_i(E)+\lambda )(ze)\cdot d\log z .
$$
On the other hand a section projecting to 
$\widetilde{E}_{t_i}/\widetilde{P}_i$ is just a section of $E$ projecting to $P_i$. 
Thus the new residues are 
$$
r^+_i(\widetilde{E})= r^-_i(E), \;\;\; r^-_i(\widetilde{E})= r^+(E)+\lambda .
$$
This transformation
defines a function $\varepsilon _i: \Nn ^d \rightarrow \Nn ^{d-1}$,
such that 
$$
\Psi (\varepsilon _i(E,\nabla , P_{\hdot})) = 
\varepsilon _i\Psi (E,\nabla , P_{\hdot}).
$$
The natural transformations $\varepsilon _i$ on $\Mm ^d$ and $\Nn ^d$ are
invertible, because there are natural transformations going in the other direction.

The following well-known fact helps by giving a normal form for the bundles. 

\begin{lemma}
\label{deg1}
Suppose $(E,\nabla , P_{\hdot})\in (\Mm _1^1)^{\rm irr}$ is an irreducible logarithmic
connection on a bundle of degree $1$.  
Then $E\cong \Oo _{\pp ^1}\oplus \Oo _{\pp ^1}(1)$.
\end{lemma}
\begin{proof}
Recall that $\Omega ^1_{\pp ^1}=\Oo _{\pp ^1}(-2)$, so $\Omega ^1_{\pp ^1}(\log D)=\Oo _{\pp ^1}(2)$ since $D$ has $4$ points. 
The bundle $E$ has degree $1$ and rank $2$, with a logarithmic connection
$$
\nabla : E\rightarrow E\otimes _{\Oo _{\pp ^1}}\Omega ^1_{\pp ^1}(\log D)\cong E(2). 
$$
Suppose $L\subset E$ is a line subbundle of $E$ with $\deg (L)\geq 2$. Then 
$\deg (E/L)\leq -1$, so the  $\Oo _X$-linear
map $L\rightarrow (E/L)(2)$ induced by $\nabla$
must be zero. This says that $\nabla$ preserves $L$, but that contradicts the
hypothesis of irreducibility.
\end{proof}

If ${\bf r}$ is a generic collection of residues then any element of $\Mm ^1_1({\bf r})$
is irreducible (see Lemma \ref{kgeneric} below),  
so the previous lemma then applies everywhere. 

Suppose $(E,\nabla , P_{\hdot})\in (\Mm _1^1)^{\rm irr}$ is an irreducible
connection, and a collection of weights $\alpha$
is specified. Then we obtain a parabolic vector bundle $(E,P_{\hdot},\alpha )$.
The underlying bundle $E=\Oo \oplus \Oo (1)$ is fixed, by Lemma \ref{deg1}.
We would like to know whether the parabolic bundle is semistable, and if
not, what is its destabilizing subbundle.


\section{Parametrization of parabolic structures}
\label{sec-ParabolicStructures}

Motivated by the previous lemma,
we now investigate the moduli stack of quasiparabolic structures on the
bundle $B:=\Oo \oplus  \Oo (1)$. Let $x$ denote the usual coordinate on $X=\pp ^1$. 

Let $\Qq $ denote the space of quasiparabolic structures on $B$ over the
collection of four points $t_1,t_2,t_3,t_4$.
Assume that $t_i\neq \infty$, let $e$ be the unit section of $\Oo$, and let $f\in \Oo (1)$
be the unit section vanishing at $\infty$. Thus $e(t_i),f(t_i)$ form a basis
for $B_{t_i}$. With respect to this basis, a parabolic structure at $t_i$ consists
of a line $P_i\subset  \cc ^2$, corresponding hence to a point in $\pp ^1$. Therefore
$$
\Qq = \pp ^1\times \pp ^1\times \pp ^1\times \pp ^1.
$$
Use coordinates $u_1,u_2,u_3,u_4$ which are allowed to take the value $\infty$.
A point $(P_1,P_2,P_3,P_4)$ is given by coordinates $(u_1,u_2,u_3,u_4)$ where
$$
P_i = \langle e(t_i)+ u_if(t_i)\rangle ,
$$
the case $u_i=\infty$ corresponding to $P_i=\langle f(t_i)\rangle$.

Let $A:= {\rm Aut}(B)$. It acts on $\Qq$.
A general element of $A$ may be written as a quadruple
$(a,b,c,s)$ with $a,s\in \cc ^{\ast}$ and $b,c\in \cc$, acting by
$$
e\mapsto s(a e + (b+cx)f), \;\;\; f \mapsto sf.
$$
The elements $(1,0,0,s)$ provide a central $\Gm \hookrightarrow A$ corresponding to scalar 
multiplication acting trivially on $\Qq$. So $A$ acts through the quotient
which has parameters $(a,b,c)$. We have
$$
(a,b,c)(e(t_i)+u_if(t_i))= ae(t_i) + (b+ct_i+u_i) f(t_i)
$$
so in terms of the coordinates this says that $(a,b,c)$ acts by
$$
(u_1,u_2,u_3,u_4)\mapsto \left(
\frac{b+ct_1+u_1}{a}, \frac{b+ct_2+u_2}{a},
\frac{b+ct_3+u_3}{a},\frac{b+ct_4+u_4}{a} \right) .
$$
In other words, $(1,b,c)$ act by translation by $b(1,1,1,1)+c(t_1,t_2,t_3,t_4)$
and $(a,0,0)$ acts by scalar multiplication by $a^{-1}$. 

These actions fix any values of the coordinates $u_i=\infty$. This corresponds to the
fact that $\Oo (1)$ is the destabilizing subbundle of $B$ so it is fixed by the
automorphism group, and the conditions $u_i=\infty \Leftrightarrow P_i\in \Oo (1)_{t_i}$ are preserved
by the action of $A$. 

The open subset $\cc ^4\subset \Qq$ corresponding to finite values of $u_i$ 
is preserved by the action of $A$. There, the quotient stack has the form 
$$
\cc ^2/\cc ^{\ast},
$$
indeed $\cc ^4$ modulo the translation action of the $(1,b,c)$ is $\cc ^2$,
on which the elements $(a,0,0)$ act by scalar multiplication. We can make this
more invariant in the following way. The open set $\cc ^4$ may be written as 
$$
\cc ^4 = \bigoplus _{i=1}^4 \Oo (1)_{t_i}.
$$
Consider the exact sequence
$$
0\rightarrow \Oo (-3)\rightarrow \Oo (1)\rightarrow \bigoplus _{i=1}^4 \Oo (1)_{t_i}\rightarrow 0,
$$
which on the level of cohomology gives
$$
0\rightarrow H^0(\Oo (1))\rightarrow \bigoplus _{i=1}^4 \Oo (1)_{t_i}\rightarrow
H^1(\Oo (-3))\rightarrow 0.
$$
The image of $H^0(\Oo (1))$ is the $\cc ^2$ along which the translations $(1,b,c)$
take place. Therefore, the quotient $\cc ^2$ is naturally identified with 
$H^1(\Oo (-3))\cong H^0(\Oo (1))^{\ast}$ so we can write
$$
\Qq /A \supset \cc ^4 /A \cong H^0(\Oo (1))^{\ast} /\cc ^{\ast}.
$$
Similar considerations hold for the strata such as $(u_1,u_2,u_3,\infty )$
and permutations, $(u_1,u_2,\infty  ,\infty )$ and permutations, and so on.

The moduli space may be given a finer stratification, according to 
how subbundles of the form $\Oo \hookrightarrow B$ and $\Oo (-1)\hookrightarrow B$
meet the $P_i$. These conditions come into play for the stability conditions at
various values of the weight parameters $\alpha$. 

Quasi-parabolic structures may also be interpreted in terms of projective
geometry. Let $\pp (B)\rightarrow \pp ^1$ be the $\pp ^1$-bundle of lines in 
the fibers of $B$. Think of the base $\pp ^1$ as the space of lines $\ell \subset V$
in a $2$-dimensional vector space $V$. The bundle $B$ associates to $\ell$
the space $B_{\ell}=\cc \oplus (V/\ell )$, and a line $L\subset B_{\ell}$
is a $2$-dimensional subspace $\tilde{L}\subset \cc \oplus V$ such that $\ell \subset
\tilde{L}$. Hence, $\pp (B)$ may be seen as the variety of flags 
$$
0\subset \ell \subset \tilde{L}\subset \cc \oplus V
$$
such that $\ell \subset V$, or equivalently 
$$
0\subset \tilde{L}^{\perp} \subset \ell ^{\perp} \subset \cc \oplus V^{\ast}
$$
such that $\cc \subset \ell ^{\perp}$. In this way, $\tilde{L}^{\perp}$ may be viewed
as a point in $\pp (\cc \oplus V^{\ast})=\pp ^2$, and $\ell ^{\perp}$ is a line
contining $\tilde{L}^{\perp}$ and the origin. The origin here means 
$\cc \subset \cc \oplus V^{\ast}$. This describes $\pp (B)$ as the blow-up 
$\tilde{\pp}^2$ of $\pp ^2$ at the origin. 

The space of lines through the origin is our original $\pp ^1$,
and the map $\pp (B)\rightarrow \pp ^1$ is the projection centered at the origin. 
If $T$ is a line through the origin corresponding to a point $t\in \pp ^1$ then
the fiber $\pp (B)_t$ is just the line $T$ itself. 

The four points
$t_1,t_2,t_3,t_4$ correspond to four fixed lines passing through the origin which will
be denoted $T_1,T_2,T_3,T_4$. The above discussion can be summed up as follows. 

\begin{lemma}
\label{projinterp}
A quasiparabolic structure on the bundle $B$ is  the specification 
of a quadruple of points $(U_1,U_2,U_3,U_4)$
in $\pp ^2$ such that $U_i\in T_i$. Thus a more invariant expression for
the parameter space is
$$
\Qq = T_1\times T_2\times T_3 \times T_4.
$$
The coordinates $u_i$ are obtained by trivializations of $T_i$, with $u_i=\infty$
corresponding to the origin $0\in T_i$. 
\end{lemma}

The automorphism group of $B$ acts as the subgroup of automorphisms of $\cc \oplus V$
which fix the origin in $\pp ^2$, and which act trivially on the space of lines passing through
the origin. It has the matrix representation 
$$
A = \left\{ \left( \begin{array}{ccc} 1 & b & c \\ 0 & a & 0 \\ 0 & 0 & a \end{array}
\right) \right\} .
$$

Next consider the addition of a logarithmic connection to a parabolic
structure parametrized as above. 
We say that a collection of residules ${\bf r}= (r_1^{\pm},\ldots , r_4^{\pm})\in \Nn ^d_1$
is Kostov-generic if, for any $\sigma _1,\sigma _2,\sigma _3,\sigma _4\in \{ +,-\}$
\begin{equation}
\label{kgeneq}
r_1^{\sigma _1}+r_2^{\sigma _2}+r_3^{\sigma _3}+r_4^{\sigma _4}\not \in  \zz .
\end{equation}
Say that ${\bf r}$ is non-resonant if 
\begin{equation}
\label{nonreseq}
r^+_i-r^-_i \not \in \zz .
\end{equation}
Say that ${\bf r}$ is nonspecial if it is Kostov-generic and nonresonant,
and special otherwise. These conditions are introduced in \cite[Definition 2.4]{InabaIwasakiSaito}, with the terminology ``generic'' meaning nonspecial. 

The special ${\bf r}$ form a collection of hyperplanes in $\Nn ^d_1$
which are the reflection hyperplanes for the affine $D_4$ Weyl group,
this group of operations acts on the moduli space by the Okamoto symmetries.
These include elementary transformations, plus an
additional symmetry to be discussed at the end of the paper. 

The following property is well-known. 

\begin{lemma}
\label{kgeneric}
Suppose ${\bf r}= (r_1^{\pm},\ldots , r_4^{\pm})\in \Nn ^d_1$ is Kostov-generic. Then for
any $(E,\nabla , P_{\hdot})\in \Mm ^d_1({\bf r})$, the bundle with connection $(E,\nabla )$
is irreducible. 
\end{lemma}
\begin{proof}
See \cite[Lemma 2.1]{InabaIwasakiSaito}.
If $F\subset E$ is a subbundle with compatible connection $\nabla _F$ then
the residues of $F$ are $r_i^{\sigma_i}$ for some $\sigma _1,\sigma _2,\sigma _3,\sigma _4\in \{ +,-\}$. The Fuchs relation for $F$ says 
$$
r_1^{\sigma _1}+r_2^{\sigma _2}+r_3^{\sigma _3}+r_4^{\sigma _4} = -\deg (F)\in \zz ,
$$
contradicting \eqref{kgeneq}.
\end{proof}

A quasi-parabolic bundle is simple if it has no non-scalar endomorphisms preserving
the parabolic subspaces $P_i$. This is an open condition;
denote by $\Qq ^{\rm simple}\subset \Qq$ the subset of simple 
quasi-parabolic bundles. 

The analogue of Weil's criterion in our case is:

\begin{lemma}
\label{simple}
Suppose ${\bf r}\in \Nn ^1_1$ is a nonspecial collection of residues,
and suppose $(E,P_{\hdot})$ is a quasi-parabolic bundle with $\deg (E)=1$. 
Then the following conditions are equivalent:
\newline
---there exists a connection
$\nabla$ on $E$, compatible with the $P_i$ and inducing the given residues 
$r ^-_i$ on $P_i$ and $r ^+_i$ on $E_{t_i}/P_i$;
\newline
---$(E,P_{\hdot})$ is an indecomposable quasi-parabolic bundle;
\newline
---$(E,P_{\hdot})$ is a simple quasi-parabolic bundle;
\newline
---$E\cong B$, there is at most one point with $u_i=\infty$, and the 
$P_i$ for $u_i\neq \infty$ 
are not all contained in a single $\Oo \subset B$;
\newline
---$E\cong B$ and among the points in projective space
$U_i\in T_i\subset \pp ^2$ corresponding 
to the quasiparabolic structure, there are three non-colinear points distinct from 
the origin.
\end{lemma}
\begin{proof}
By the nonspeciality condition, if $(E,P_{\hdot})$ has a connection with given 
residues $r_i^{\pm}$, then it must be indecomposable as a quasiparabolic bundle. 
Indeed, if $E=E_1\oplus E_2$ were
a decomposition into line bundles compatible with $P_{\hdot}$, then writing
$\nabla$ as a matrix the diagonal terms would be connections $\nabla _1,\nabla _2$
on $E_1,E_2$. Compatibility with $P_{\hdot}$ means that the residue would be
either upper or lower triangular at each $t_i$, so the residues of $\nabla _1,\nabla _2$
would be taken from among the residues $r_i^{\pm}$ of $\nabla$. This contradicts the Kostov-genericity
condition for ${\bf r}$. 

So in this case, the Weil criterion
\cite{Weil,Atiyah,Biswas,CrawleyBoevey}
says that a connection exists if and only if $(E,P_{\hdot})$ is indecomposable.
For convenience here is the argument.
Consider the subsheaf $End(E,P_{\hdot})\subset End(E)$
of endomorphisms respecting the parabolic structure. At each $t_i$ we have a map
to a skyscraper sheaf
$$
End(E,P_{\hdot})\rightarrow \cc^2
$$
expressing the action of an endomorphism on $P_i$ and $E_{t_i}/P_i$. 
Let $End^{\rm st}(E,P_{\hdot})$ be the subsheaf which is the kernel of these maps at
each $t_i$. It is the subsheaf of endomorphisms which map $P_i$ to $0$ and $E_{t_i}$ to $P_i$.
The obstruction to the existence of a logarithmic connection having given residues, is $\beta \in H^1(End^{\rm st}(E,P_{\hdot})\otimes \Omega ^1_X(\log D))$. There is a
trace map $End ^{\rm st}(E,P_{\hdot})\rightarrow \Oo _X(-D)$ hence
$$
H^1(End^{\rm st}(E,P_{\hdot})\otimes \Omega ^1_X(\log D))\rightarrow
H^1(X,\Omega ^1_X)\cong \cc .
$$
The trace of the obstruction is zero if the Fuchs relation holds. The Serre dual
of $H^1(End^{\rm st}(E,P_{\hdot})\otimes \Omega ^1_X(\log D))$ is $H^0(End(E,P_{\hdot}))$
which is the space of endomorphisms of the quasiparabolic structure $(E,P_{\hdot})$.

If $(E,P_{\hdot})$ is indecomposable, then any endomorphism has the form $c+\varphi$
where $c\in \cc$ is a scalar constant and $\varphi$ is nilpotent. The pairing
of $c$ with $\beta$ is $cTr(\beta )=0$. On the other hand, $\varphi$ preserves a
filtration and acts by $0$ on the graded pieces. The initial connections on an open
cover, used to define the obstruction, can be chosen compatibly with the filtration,
so $\beta$ comes from a class with coefficients in the endomorphisms respecting the
filtration \cite{Atiyah}. 
As $\varphi$ acts trivially on the graded pieces, $Tr (\varphi \beta )=0$.
This shows that $\beta$ paired with any endomorphism is zero, which by Serre
duality implies that $\beta = 0$. So there exists a connection 
with the given residues. 

If $(E,P_{\hdot})$ is simple then it is indecomposable. 

In the present case the
converse is true too. Suppose $(E,P_{\hdot})$ is indecomposable. If $E\cong \Oo (m)\oplus \Oo (1-m)$ with $m\geq 2$ then one can choose the copy of $\Oo (1-m)$ to pass
through any $P_i$ not contained in $\Oo (m)_{t_i}$, which decomposes the parabolic bundle.
Thus $E\cong B=\Oo \oplus \Oo (1)$. Furthermore, if two or more of the $P_i$ are
equal to $\Oo (1)_{t_i}$ then we can choose the $\Oo \subset B$ to pass through the $\leq 2$ remaining other $P_i$ again giving a decomposition. This shows that there is at
most one $u_i=\infty$, and at least three $u_i\in \cc$. 
Similarly if the $P_i$ with $u_i\neq \infty$ are all contained in a $\Oo \subset B$
then this decomposes the quasiparabolic bundle. 

In the projective space
interpretation of Lemma \ref{projinterp}, 
the quasiparabolic structure on $E=B$ corresponds to four points
in $\pp^2$, $U_i\in T_i\subset \pp ^2$. 
The previous paragraph says that 
the points $U_i$ are not all colinear, which implies that no two can
be at the origin, and if one of them is at the origin
then the remaining three are not all colinear. 
Suppose $a\in A$, viewed as an 
automorphism of $\pp ^2$ preserving the origin and the
$U_i$. There exists a subset of three $U_i$ which are distinct
from the origin and not colinear, and these
together with the origin form a frame for $\pp ^2$. As $a$ preserves the frame,
it acts trivially on $\pp ^2$ so it is a scalar element of $A$. This shows that
$(E,P_{\hdot})$ is simple. This discussion also shows the equivalence with the last
two conditions. 
\end{proof}

Given a parabolic structure consisting of $U_i\in T_i$, there is a conic
$C$ passing through the origin and through the $U_1,U_2,U_3,U_4$. 
Assuming indecomposability, the
conic is unique. 
Conversely,
given a conic passing through the origin, it cuts each line $T_i$ in another point.
So, the open set $\Qq ^{\rm simple}$ is isomorphic to an 
appropriate open set of the set of conics passing through the origin.

To the conic $C$ we can associate its tangent line at the origin (note that 
since the $T_i$ are distinct, $C$ cannot be two lines crossing at the origin). 
This gives a map $Q:\Qq ^{\rm simple}\to\pp ^1$.
For a generic value $Q\not\in\{t_1,t_2,t_3,t_4\}$, the conic has to be smooth
(otherwise all $U_i$ would lie on the same line which is excluded)
We note that for each of the four special values $Q=t_i$, two kinds of situations occur: either the conic 
is smooth, or the conic splits into the union of $T_i$ and another line $T'$ (not passing through the origin).
These cases respectively correspond to either $U_i$ lying at the origin ($u_i=\infty$), or the three other 
$U_j$ being aligned, all contained in $T'$ (the corresponding $P_j$ are all three contained in a single 
$\mathcal O\subset B$). This will be emphasized in section \ref{sec-stablezone}, especially by Lemma \ref{ParabolicModuliStack},
where a moduli stack description of $\Qq ^{\rm simple}\stackquot A$ underlines a non separated phenomenon
over each of the $4$ values $Q=t_i$.
\label{conicpage}

There is a tautological 
universal parabolic structure $P_{\hdot}^{\rm univ}$ on the trivial
bundle $E^{\rm univ}={\rm pr}_2^{\ast}(B)$ over $\Qq \times X$.
Let  $\Hh \rightarrow \Qq$ be the parameter variety for logarithmic connections
on $(E^{\rm univ},P^{\rm univ})$ relative to $\Qq$. Thinking of connections as splittings
of a certain exact sequence, one can see that $\Hh$ is a quasiprojective variety.
The group $A$ acts on $\Hh$ over its action on $\Qq$, with the moduli stack as 
quotient, and we get the map to $\Qq \stackquot A$:
$$
\Mm ^1_1 = \Hh \stackquot A \rightarrow \Qq \stackquot A .
$$
As before there is a map $\Hh \rightarrow \Nn ^1_1$ and for a collection of
residues ${\bf r}= (r_1^{\pm},\ldots , r_4^{\pm})$, 
let $\Hh ({\bf r})\subset \Hh$ denote the inverse image. Thus
$$
\Mm ^1_1({\bf r})= \Hh ({\bf r})\stackquot A.
$$

\begin{corollary}
\label{Hfib}
In the situation of the previous lemma, the space of connections on a given simple
quasiparabolic bundle $(E,P_{\hdot})$ with the specified nonspecial 
residues, has dimension $1$.
In fact $\Hh ({\bf r})\rightarrow \Qq$ is a smooth fibration over
$\Qq ^{\rm simple}$ whose fibers are affine lines $\aaa ^1$. 
\end{corollary}
\begin{proof}
The space of connections is the space of splittings of the appropriate sequence,
in particular it is a principal homogeneous space on a vector space. 
Since $(E,P_{\hdot})$ is simple the dimensions of all the groups involved
are constant as a function of $(u_1,\ldots , u_4)\in \Qq ^{\rm simple}$. 
Semicontinuity theory implies that $\Hh ({\bf r})$ is a smooth fibration. The fiber dimension is $1$ by dimension count,
hence the fibers are $\aaa^1$. 
\end{proof}

\begin{proposition}
\label{kgensmooth}
Fix a nonspecial
collection of residues ${\bf r}\in \Nn ^1_1$. 
Then $\Hh ({\bf r})$ is smooth. 
The quotient $\Mm ^1_1({\bf r})=\Hh ({\bf r})\stackquot A$
is a $\Gm$-gerb over its coarse moduli space 
$M^1_1({\bf r})$ which is a smooth separated quasiprojective variety
and is in fact a fine moduli space. 
The inverse image of a point $e\in \Qq ^{\rm simple}/ A$ under the map 
$$
\Mm ^1_1({\bf r})\rightarrow \Qq ^{\rm simple}/ A 
$$
is a closed substack, a $\Gm$-gerb over a closed subvariety of $M^1_1({\bf r})$. 
\end{proposition}
\begin{proof}
By Corollary \ref{Hfib}, $\Hh ({\bf r})$ is a fibration over 
$\Qq ^{\rm simple}$, so it is smooth. By Lemma \ref{kgeneric}, 
any point of $\Hh ({\bf r})$ represents an irreducible connection.
It follows that the automorphism group of the connection, which is also the
stabilizer in $A$ of the action, is $\Gm$. The coarse moduli space
exists, and is a fine moduli space, by GIT because for an appropriate choice of
parabolic weights all points are stable. See \cite{InabaIwasakiSaito} for example.
Since the stabilizer  group is always $\Gm$, the moduli stack is a $\Gm$-gerb
over the fine moduli space. If $e\in \Qq ^{\rm simple}/ A$, then the
$A$-orbit of $e$ is closed in $\Qq ^{\rm simple}$ as may be seen directly.
Thus, its inverse image is a closed $A$-invariant subset of $\Hh ({\bf r})$
so it corresponds to a closed substack, lying over a closed subvariety of
the fine moduli space. 
\end{proof}

A collection of weights $\alpha = (\alpha ^{\pm}_1,
\alpha ^{\pm}_2,\alpha ^{\pm}_3,\alpha ^{\pm}_4)$ for a parabolic
structure on a bundle of degree $d$ is called 
nonspecial
if $\alpha ^-_i<\alpha ^+_i<\alpha ^-_i+1$, which is analogous to
nonresonance, and if it satisfies the
Kostov-genericity condition that
for any $\sigma _1,\sigma _2,\sigma _3,\sigma _4\in \{ +,-\}$,
$$
\sum _{i=1}^4\alpha _i^{\sigma _i} + \frac{d-\sum _{i=1}^4(\alpha _i^++\alpha _i^-)}{2}
\not \in \zz .
$$

\begin{lemma}
If $\alpha$ is a nonspecial collection of weights for degree $d$,
then any $(E,\nabla , P_{\hdot})\in \Mm ^d_{\lambda}$ which is $\alpha$-semistable,
is in fact $\alpha$-stable. 
\end{lemma}
\begin{proof}
From the Kostov-genericity condition,
there can be no rank $1$ subsystem with an exact equality between slopes. 
\end{proof}


\section{The Higgs limit construction}
\label{sec-HiggsLimit}

Choose nonspecial collections of residues ${\bf r}\in \Nn _1^d$ 
and consider the family of moduli stacks 
$$
\Mm ^{d}(\lambda {\bf r}) \rightarrow \aaa ^1.
$$
The group $\Gm$ acts over its standard action on $\aaa ^1$.

Given a point $(E,\nabla , P_{\hdot})$ in the fiber over $\lambda = 1$,
we would like to take the limit of $(E,u\nabla , P_{\hdot})$ as $u\rightarrow 0$.
The limit will be a vector bundle with $0$-connection, which is to say a Higgs bundle,
i.e. a point in the moduli stack $\Mm ^d_0$. At $\lambda =0$ the residues
go to $0$ since, in order to obtain an action of $\Gm$ we had to take the family
of residues ${\bf r}(\lambda )=\lambda {\bf r}$. Thus, the limit should be
a point in $\Mm ^d_0(0)$.

Unfortunately, the moduli stack is highly unseparated over $\lambda =0$, because
the existence of an $\Oo_X$-linear Higgs field doesn't impose as strong a 
condition as the existence of a connection. 

Therefore, there are many different ways to obtain a limit. It is instructive to consider
some of the possibilities. These basically come from considering families
of gauge transformations depending on $u$. The first and easiest way is to
take the trivial gauge transformations, which is to say we consider the $u$-connections
$u\nabla$ on the fixed quasiparabolic bundle $(E,P_{\hdot})$. As $u\rightarrow 0$
these approach the zero Higgs field $\theta = 0$, so in this case the limit is just
the quasiparabolic bundle
$(E,P_{\hdot})$ considered as a quasiparabolic Higgs bundle with $\theta = 0$.  

Another way of taking the limit is to rescale with respect to the decomposition $E=\Oo \oplus \Oo (1)$. Write the connection as a matrix
$$
\nabla = \left( \begin{array}{cc} 
\nabla _0 & \theta \\
\zeta & \nabla _1
\end{array}\right) 
$$
where $\nabla _0$ and $\nabla _1$ are logarithmic connections on $\Oo$ and $\Oo (1)$ respectively, and
$\theta : \Oo (1)\rightarrow \Oo \otimes \Omega ^1_X(\log D)$ and $\zeta :\Oo \rightarrow \Oo (1)\otimes \Omega ^1_X(\log D)$ are $\Oo_X$-linear operators. Note however that 
the residues of $\nabla$ are not compatible with the decomposition. Then we can
make a gauge transformation rescaling by $u$ on the first component
$$
g_u = \left( \begin{array}{cc} 
u & 0 \\
0 & 1
\end{array}\right) ,
$$ 
so that
$$
u\nabla \sim g_u^{-1}\circ u\nabla \circ g_u=   \left( \begin{array}{cc} 
u\nabla _0 & \theta \\
u^2\zeta & u\nabla _1
\end{array}\right) .
$$
In this case the limiting Higgs bundle is $\Oo \oplus \Oo (1)$ with
Higgs field 
$$
\nabla _0= \left( \begin{array}{cc} 
0 & \theta  \\
0 & 0
\end{array}\right) , \;\;\; 
\theta : \Oo (1)\rightarrow \Oo \otimes \Omega ^1_X(\log D).
$$ 
The quasiparabolic structure projects in the limit to one which is compatible with
the decomposition. 

Other rescalings are possible corresponding to other meromorphic decompositions of the
bundle $E$. In fact, the limiting process works even when the bundle is only filtered,
with the limiting bundle being the associated-graded.

In order to get a unique limit we should look for a separated stack or at least a
stack having a separated coarse moduli space,
and for that reason impose a semistability condition. Fix a nonspecial collection
of parabolic weights $\alpha = ( \alpha ^{\pm}_1,\ldots ,\alpha ^{\pm}_4)$ 
and consider the moduli family
$$
\Mm ^{d,\alpha}(\lambda {\bf r}) \rightarrow \aaa ^1
$$
of $\alpha$-semistable parabolic logarithmic $\lambda$-connections having the given
residues. Note that semistability and stability are equivalent since $\alpha$ is
chosen to be Kostov-generic. 

\begin{proposition}
\label{higgslimit}
For any $(E,\nabla , P_{\hdot})\in \Mm ^{d,\alpha}_1({\bf r})$,
there exists a unique limit
$$
(F,\theta , Q_{\hdot})=\lim _{u\rightarrow 0}(E,u\nabla , P_{\hdot})
$$
in the moduli stack $\Mm ^{d,\alpha}_0(0)$ of parabolic Higgs bundles with
vanishing residues. 
\end{proposition}
\begin{proof}
See \cite{idsm}. However, the treatment there concerned mostly the case of 
compact base curve $X$. Furthermore, in the present case of rank $2$, the
general iterative procedure of \cite{idsm} is not necessary. So it is perhaps
worthwhile to do the existence proof here. 

If $(E,P_{\hdot})$ is already $\alpha$-stable as a parabolic vector bundle, then 
the limit is just $(F,Q_{\hdot})=(E,P_{\hdot})$ with $\theta = 0$ as in the
first example above. 

If $(E,P_{\hdot})$ is not $\alpha$-stable, hence also not $\alpha$-semistable,
there is a quasiparabolic line subbundle $(L, R_{\hdot})\subset (E,P_{\hdot})$
which is maximally destabilizing. Here $R_{i}$ is either $0$ or $L_{t_i}$,
in the second case $R_i=L_{t_i}=P_i$ is required. 
The parabolic weights are assigned accordingly:
$\alpha _{L,i} = \alpha ^+_i$ if $R_i=0$, $\alpha _{L,i} = \alpha ^-_i$ if $R_i=L_{t_i}$. This determines
the parabolic degree $\deg ^{\rm par}(L,R_{\hdot},\alpha _L)$, and
the destabilizing condition says that 
$$
\deg ^{\rm par}(L,R_{\hdot},\alpha _L) > \frac{\deg ^{\rm par}(E,P_{\hdot},\alpha )}{2}.
$$
The quotient $E/L$ similarly has a parabolic structure $R'_{\hdot}$ and 
weights $\alpha _{E/L}$, and 
$$
\deg ^{\rm par}(E/L,R'_{\hdot},\alpha _{E/L}) < \frac{\deg ^{\rm par}(E,P_{\hdot},\alpha )}{2}.
$$
The connection determines an $\Oo_X$-linear map
$$
\theta : L\rightarrow (E/L)\otimes \Omega ^1_X(\log D),
$$
nonzero because otherwise $(E,\nabla )$ would be reducible contradicting 
Lemma \ref{kgeneric} in view of the
genericity assumption for the residues ${\bf r}$.

As in the second example described above, after an appropriate
gauge rescaling, the limiting Higgs bundle is 
$$
(F,Q_{\hdot})= (L,R_{\hdot})\oplus (E/L,R'_{\hdot}),
$$
with Higgs field $\theta$. As $\theta \neq 0$ the only possible 
$\theta$-invariant subbundle
is $(E/L,R'_{\hdot})$, and this has slope strictly smaller than the slope of $F$. 
So the parabolic Higgs bundle $(F,\theta , Q_{\hdot})$ with weights determined by
$\alpha _L$ and $\alpha _{E/L}$ is stable. 

This shows existence of a limit. For unicity, proceed as in \cite{idsm}.
Given two different limits, they correspond to two different families of 
$u$-connections on $X\times \aaa ^1$ relative to $\aaa ^1$, 
isomorphic outside of $u=0$. 
Semicontinuity
of the space of morphisms between them says that there is a nonzero morphism between
the limits at $u=0$, but since both are $\alpha$-stable this must be an isomorphism. 
Thus the limit is unique. 
\end{proof}

The limiting Higgs bundle has to be fixed by the action of $\Gm$ scaling the Higgs field,
so it is a Higgs bundle corresponding to a variation of Hodge structure \cite{hbls}. The
case $\theta = 0$ corresponds to a unitary representation, whereas $L\oplus (E/L)$
with nonzero Higgs field $\theta$ corresponds to a variation of Hodge structure
with structure group $U(1,1)$ and period map taking values in the unit disc. 
We don't use this information any further here, but it is suggestive of some interesting
questions on the position of real monodromy representations in the overall picture.

The limit process leads to an equivalence relation: two points of $\Mm ^d_1({\bf r})$
are equivalent if their limits are the same. The moduli space is decomposed into
equivalence classes which are locally closed subsets, and the foliation conjecture
of \cite{idsm} states that these should be the leaves of a foliation. In the present
situation we will be able to prove that they are in fact the fibers of a morphism;
which morphism it is will depend on the parabolic weight chamber. 

The first step in this direction is to describe the possibilities
for the limiting Higgs bundle $(F,\theta , Q_{\hdot})$. The two examples of limits 
discussed
above will basically cover all of the possibilities, up to making elementary
transformations. The first task is to investigate more closely the
$\alpha$-stability condition.

Let $\mu _i:= (\alpha ^+_i + \alpha ^-_i)/2$ and $\epsilon _i:= (\alpha ^+_i - \alpha ^-_i)/2$
so
$$
\alpha ^+_i = \mu _i +\epsilon _i,\;\;\; \alpha ^-_i = \mu _i -\epsilon _i,
$$
with $0<\epsilon _i < \frac{1}{2}$. The parabolic semistability condition for 
the parabolic bundle (without connection)  
$(E, P_i)$ is described as follows. Let $\mu _{\rm tot}:= \mu _1+\mu _2+\mu _3+\mu _4$,
although 
in fact the values of $\mu_i$ and $\mu _{\rm tot}$ won't turn out to make a difference.
Let $\epsilon _{\rm tot}:= \epsilon _1+\epsilon _2+\epsilon _3+\epsilon _4$.

Assume that the points $t_i$
are ordered so that $\epsilon _1 \geq \epsilon _2\geq \epsilon _3\geq \epsilon _4$. The conclusion will need to be extended by allowing permutations at the end.

For any sub-line bundle $L\subset E$, let 
$$
\Sigma (L):= \{ i\;  |\; L_{t_i}=P_i\} .
$$
Then 
$$
\deg ^{\rm par}(L) = \deg (L) - \mu _{\rm tot} - \epsilon _{\rm tot} +2\sum _{i\in \Sigma (L)} \epsilon _i.
$$
On the other hand, the parabolic slope of $E$ is $(d-2\mu _{\rm tot})/2$ with $d=\deg (E)$.
Therefore, adding $\mu _{\rm tot}$ to both sides of the equation, 
$L$ contradicts semistability if and only if
$$
\deg (L) - \epsilon _{\rm tot} +2\sum _{i\in \Sigma (L)} \epsilon _i
> d/2.
$$
Respectively, $L$ contradicts stability if $\geq$ holds. The left side
may alternatively be written $\deg (L) +\sum _{i\in \Sigma (L)}\epsilon _i -\sum _{i\not \in \Sigma (L)}\epsilon _i$. Under the hypothesis that the weights are nonspecial,
stability and semistability are equivalent, i.e. equality can never hold. 

Specialize now to the case $E=B=\Oo \oplus \Oo (1)$. The parabolic structure
is given by a point $(u_1,\ldots , u_4)\in \Qq$ as discussed previously,
with $P_i=\langle (1,u_i)\rangle$. The semistability condition says 
$$
\deg (L) +\sum _{i\in \Sigma (L)}\epsilon _i -\sum _{i\not \in \Sigma (L)}\epsilon _i \leq 1/2.
$$
If $\deg (L)\leq -2$ then noting that $\epsilon _{\rm tot}< 2$ we always have
$$
\deg (L) +\sum _{i\in \Sigma (L)}\epsilon _i -\sum _{i\not \in \Sigma (L)}\epsilon _i
< 0 < d/2 = 1/2,
$$
so a line bundle of degree $\leq -2$ never contradicts stability. 

Consider $L=\Oo (-1)$. A map $L\rightarrow B$ is given by a pair $(v,w)$ where
$v=v_0+v_1x$ is a linear function and $w=w_0+w_1x+w_2x^2$ is a quadratic function.
Then $i\in \Sigma (L)$ if and only if $(1,u_i)$ is proportional to $(v(t_i),w(t_i))$,
in other words if 
$$
w_0+w_1t_i+w_2t_i^2 = u_i(v_0+v_1t_i).
$$
When $u_i=\infty$ replace this by $(v_0+ v_1t_i)=0$. 
This system of $4$ homogeneous equations in $5$ unknowns always has a nonzero
solution, so there is always an $\Oo (-1)=L\hookrightarrow B$ such that
$L_{t_i}=P_i$ for all $i=1,2,3,4$. This contradicts semistability if and only
if
$$
\epsilon _1+\epsilon _2+\epsilon _3+\epsilon _4 > 3/2.
$$
If this one doesn't contradict semistability then the other ones, with
less contact between $L$ and  the $P_i$, will not either.
Hence $(E,P_{\hdot})$ satisfies the semistability condition for 
line bundles of degree $-1$ if and only if 
$$
\epsilon _1+\epsilon _2+\epsilon _3+\epsilon _4 \leq 3/2 .
$$

Consider the other extreme, $L=\Oo (1)$. There is a unique morphism $L\rightarrow B$,
and $\Sigma (L)$ is the set of values of $i$ such that $u_i=\infty$. This
line subbundle contradicts semistability if and only if
$$
1/2 + \sum _{u_i=\infty} \epsilon _i > \sum _{u_i\neq \infty}\epsilon _i  .
$$
In particular, 
if $\epsilon _1+\epsilon _2+\epsilon _3+\epsilon _4 <1/2$ then
$L$ always contradicts semistability. 
On the other hand, when
$$
\epsilon _1+\epsilon _2+\epsilon _3+\epsilon _4 \geq 1/2.
$$
then there exist parabolic structures such that $L$ doesn't contradict
stability, 
including at least all of those in $\cc ^4\subset \Qq$.
Note however that some parabolic structures on the boundary  can still be unstable. 

Turn now to the subbundles of degree $0$, $L\hookrightarrow B$.
It may be assumed that $L$ is a saturated subbundle, so the inclusion
map doesn't go into $\Oo (1)\subset B$. In other words, the projection 
$B\rightarrow \Oo$ induces an isomorphism $L\stackrel{\cong}{\rightarrow}\Oo$
and we may use this isomorphism to trivialize $L$. Hence the inclusion
is given by $(1,v)$ where $v=v_0+v_1x$ is a polynomial of degree $1$.
For a parabolic structure $P_{\hdot}$ with coordinates $(u_1,u_2,u_3,u_4)$
the condition $L_{t_i}=P_i$ becomes just $v(t_i)=u_i$, i.e.
$$
v_0+v_1t_i = u_i.
$$
For any two indices $i\neq j\in \{ 1,2,3,4\}$ such that $u_i\neq \infty$ and
$u_j\neq \infty$, there is a unique solution $(v_0,v_1)$
to the pair of equations $v(t_i)=u_i$ and $v(t_j)=u_j$. 
In other words, for any pair of indices $i\neq j$ we can choose $L$ such that 
$i,j\in \Sigma (L)$. If the $u_i$ are general then $\Sigma (L)=\{ i,j\}$
has two elements. On the other hand, for some special values of $u_{\cdot}$,
the set $\Sigma (L)$ can have three or four elements. We consider those cases
later on. In the general case, the biggest degree of a subbundle is obtained
by choosing $i,j=1,2$ when the points are ordered according to decreasing values
of $\epsilon $. So, for a 
general parabolic structure $L$ will not contradict semistability, if 
$$
\epsilon _1 + \epsilon _2 -\epsilon _3 -\epsilon _4 \leq 1/2,
$$
whereas all parabolic structures will be unstable if 
$$
\epsilon _1 + \epsilon _2 -\epsilon _3 -\epsilon _4 > 1/2.
$$
Notice that to prove this last statement, 
supposing $\epsilon _1 + \epsilon _2 -\epsilon _3 -\epsilon _4 > 1/2$,
we also need to treat the cases 
where some $u_i=\infty$. If for example $u_2=\infty$, then
$$
1/2 + \epsilon _2 -\epsilon _1-\epsilon _3 -\epsilon _4 
= (1/2 + \epsilon _1 + \epsilon _2 -\epsilon _3 -\epsilon _4) - 2\epsilon_1
> \epsilon _1 + \epsilon _2 -\epsilon _3 -\epsilon _4 - 1/2 > 0,
$$
so this shows that the $\Oo (1)\subset B$ contradicts stability by the previous
discussion. The case $u_1=\infty$ is the same. 

\begin{proposition}
\label{unstablezones}
For $\alpha$ a nonspecial assignment of parabolic weights,  
define $\epsilon _i= (\alpha ^+_i-\alpha ^-_i)/2$ as above. Suppose one of
the following three conditions holds:
\newline
(a)\, $\epsilon _1+\epsilon _2+\epsilon _3+\epsilon _4 <1/2$; 
\newline
(b) \, $\epsilon _1+\epsilon _2+\epsilon _3+\epsilon _4 > 3/2$; or
\newline
(c)\, there exists a renumbering 
$\{ 1,2,3,4\} = \{ i,j,k,l\}$ such that
$$
\epsilon _i + \epsilon _j -\epsilon _k -\epsilon _l > 1/2.
$$
Then every  parabolic structure $(B,P_{\hdot})$ on the bundle $B=\Oo \oplus \Oo (1)$
is unstable. 
If, on the contrary, none of these conditions hold, then a general parabolic
structure is stable; however some special parabolic structures might still be unstable. 
\end{proposition}
\begin{proof}
The arguments have been done above.
\end{proof}

If there is a destabilizing subbundle, then it is unique; indeed any other distinct
destabilizing
subbundle would have nonzero projection to the quotient, but this would be a morphism
of parabolic line bundles strictly decreasing the parabolic degree, which is impossible.


\section{The unstable zones}
\label{sec-unstablezone}

An unstable zone is when one of the conditions (a), (b) or (c) holds in the
previous proposition.
In fact (c) contains $6$ distinct conditions so there are really $8$ different
unstable zones. The conditions are mutually exclusive so the different zones
are disjoint. The stable zone is by definition the complement, when
the opposites of (a), (b) and (c) all hold. 

The discussion will be made easier by the fact that pairs of 
elementary transformations
permute the different zones, allowing us to consider a single condition such as (a).
The following lemma explains how the parabolic weights should be changed along
an elementary transformation. 

\begin{lemma}
\label{et}
Suppose $(\widetilde{E},\widetilde{P}_{\hdot})$ 
is a quasiparabolic bundle obtained by a single
elementary transformation $\varepsilon _i$ of $(E,P_{\hdot})$ at the point $t_i$,
see page \pageref{etpage}.  Define parabolic weights at $t_i$ by
$$
\widetilde{\alpha}^+_i := \alpha ^-_i, \;\;\; \widetilde{\alpha}^-_i:= \alpha ^+_i-1,
\;\;\; \mbox{hence  } \widetilde{\epsilon}_i = 1/2 -\epsilon _i,
$$
leaving $\widetilde{\alpha}^{\pm}_j=\alpha ^{\pm}_j$ for $j\neq i$. 
Then $\widetilde{\alpha}$
is nonspecial if and only if $\alpha$ is, and $(\widetilde{E},\widetilde{P}_{\hdot})$
is $\widetilde{\alpha}$-stable if and only if $(E,P_{\hdot})$ was $\alpha$-stable.
\end{lemma}
\begin{proof}
Whereas 
$\deg (\widetilde{E})=\deg (E)-1$, the change of weights gives back
$\deg ^{\rm par}(\widetilde{E},\widetilde{P}_{\hdot},\widetilde{\alpha})
= \deg^{\rm par}(E,P_{\hdot},\alpha )$. 
Saturated line subbundles of $\widetilde{E}$ correspond to those of $E$, and 
this correspondence also preserves parabolic degree, so the stability conditions
are equivalent.
\end{proof} 

In order to preserve an odd degree of $E$,
we can do two different elementary transformations at $t_i$ and $t_j$
(then tensor say with $\Oo (t_i)$ to get back to degree $1$).
This changes $\epsilon _i$ to $1/2-\epsilon _i$ and 
$\epsilon _j$ to $1/2-\epsilon _j$.

\begin{lemma}
\label{etpermute}
The set of three conditions ((a) or (b) or (c)) is left invariant under any such pair of elementary transformations,
and these operations permute the $8$ zones transitively. So, up to such transformations,
the unstable zones are essentially equivalent. 
\end{lemma}
\begin{proof}
Direct calculation.
\end{proof}

Suppose $(E,\nabla , P_{\hdot})$ is a parabolic connection
with weights $\alpha$,  in one of the unstable zones. 
Up to doing a pair of elementary transformations, we may assume then that we
are in zone (a) where the destabilizing subbundle is $\Oo (1)\subset B$. 
The limiting parabolic Higgs bundle is $L\oplus L'$ where
$L$ is given parabolic weights $\alpha ^+_i$ at $t_i$, if $u_i\neq  \infty$,
or $\alpha ^-_i$ at $t_i$ if $u_i=\infty$. The parabolic weights for $L'$ are
complementary. The Higgs field $\theta : L\rightarrow L'\otimes \Omega ^1_X(\log D)$
is the piece coming from the connection operator $\nabla$. Noting that 
$L\cong \Oo (1)$, $L'\cong \Oo $ and $\Omega ^1_X(\log D)\cong\Oo (2)$,
we see that $\theta$ may be viewed as a section of $\Oo (1)$ or a linear function.
Its zero at a point $z\in X$ is interpreted in \cite{InabaIwasakiSaito}
\cite{InabaIwasakiSaito2} \cite{Szabo} \cite{Aidan}
 as an
``apparent singularity'' of the connection, as we shall now explain.  

\begin{definition}
\label{Pdef}
Let $\Pp$ be  the non-separated scheme obtained by glueing together
two copies of $X=\pp ^1$ by the identity map over the open subset $U=\pp ^1-\{ t_1,\ldots , t_4\}$. The copies are labeled $\Pp ^+$ and $\Pp ^-$. 
\end{definition}

Interestingly enough, this scheme also plays the same role for the stable zone.
It appeared in Arinkin's work on the geometric Langlands correspondence \cite{Arinkin}. 

In \cite{InabaIwasakiSaito2}, Inaba, Iwasaki and the
second author define a morphism
$$
\Upsilon:\Mm ^1_1({\bf r})\rightarrow \Pp 
$$
as follows. Any $(E,\nabla , P_{\hdot})$ in $\Mm ^1_1({\bf r})$
has a unique subbundle $L\subset E$ of degree $1$. The quotient $E/L$ has degree $0$. 
The connection induces an operator 
$\varphi : L\rightarrow (E/L)\otimes \Omega ^1_X(\log D)$. It is an $\Oo _X$-linear
map of line bundles. Comparing degrees of the
source and the target, we see that $\varphi$ has exactly one zero. The position 
of the zero defines a point in $\pp ^1$. If located at one of the singular points $t_i$
then we can further ask whether $L_{t_i}\subset P_{i}\subset E_{t_i}$, if so 
then the point goes into $\Pp ^-$, if not it goes into $\Pp ^+$. 

If the zero of $\varphi$ is not located at $t_i$, then the
condition that $\res (\nabla , t_i)$ respect the quasiparabolic $P_i$ implies that
$P_i \neq L_{t_i}$.  

\begin{proposition}
\label{iis-picture}
This pointwise prescription defines a morphism $\Upsilon$,
all fibers of which are trivial $\Gm$-gerbs over $\aaa^1$. The structure of the
moduli space $M^1_1({\bf r})$
is a ruled surface, blown up at two distinct points on each fiber $F_i$ 
over $t_i \in \pp^1$ of Hirzebruch surface $\Sigma_2 \rightarrow \pp^1$ 
with subsequently the strict transform of the section at infinity and the fibers $F_i$, $1 \leq i \leq 4$ removed. The affine fibers of $\Upsilon$ over points of $U$ are
the fibers of the ruled surface, over the doubled-up points they are the two exceptional divisors.
\end{proposition}
\begin{proof}
See Theorem 4.1 of \cite{InabaIwasakiSaito2}. This picture will be described
in further detail in Section \ref{sec-okamoto}. 
\end{proof}

In order to relate this map with the limit map, we investigate what
stable Higgs bundles look like. 

\begin{lemma}
\label{thetaneqzero}
If $(E,\theta )$ is an $\alpha$-stable Higgs bundle in $\Mm ^1_0(0)$
with $\theta = 0$
then $E\cong B$. If $\alpha$ is in the unstable zone then this can't happen, so in the unstable zone we have $\theta \neq 0$ for any
$\alpha$-stable Higgs bundle. 
\end{lemma}
\begin{proof}
If $\theta = 0$ then the stability condition is supposed to hold for any subbundle. 
If $E$ is not of the form $B=\Oo \oplus  \Oo (1)$ then $E$ has a subbundle of
degree $2$. For this subbundle, assuming the worst-case
scenario $L_{t_i}\not\subset P_i$ for any $i$,
the stability condition as discussed above becomes
$$
2 -\epsilon _1 -\epsilon _2 -\epsilon _3 - \epsilon _4 < \frac{1}{2} .
$$
Suppose this holds. It means that $\epsilon _1+\epsilon _2 + \epsilon _3 + \epsilon _4> 3/2$. However, then there is a subbundle of the form $\Oo (-1)=L'\subset E$
such that $L'_{t_i}=P_i$ for all $i$. For this subbundle,
$$
-1 + \epsilon _1+\epsilon _2 + \epsilon _3 + \epsilon _4 > 1/2
$$
contradicting stability. This contradiction shows that $E\cong B$. 
Furthermore $(E,P_{\hdot})$ is an $\alpha$-stable parabolic bundle, so
Proposition \ref{unstablezones} 
shows that $\alpha$ has to be in the stable zone.

If $\alpha$ is in the unstable zone (that is to say, if
one of the inequalities of Proposition \ref{unstablezones}
holds), then the above argument shows that no stable 
Higgs bundle with $\theta = 0$ can exist, showing that $\theta \neq 0$.
\end{proof}

Suppose $\alpha$ is in the unstable zone and $(E,\theta ,P_{\hdot})$ is an $\alpha$-stable 
parabolic Higgs bundle in the fixed point set $\Mm ^1_0(0)^{\Gm}$. By the lemma, 
$\theta \neq 0$. This means that $E$ must be a nontrivial
system of Hodge bundles \cite{hbls}, which in
the rank two case means it is a direct sum of two line bundles 
$$
E=E^0\oplus E^1, \;\;\; \theta : E^0\rightarrow E^1\otimes \Omega ^1_X(\log D)
$$
with $\theta \neq 0$. It follows that $\deg (E^0)\leq \deg (E^1)+2$. 
The quasiparabolic structure is compatible with the $\Gm$-action, so 
either $P_i\subset E^0$ or $P_i\subset E^1$. The only subbundle preserved by $\theta$
is $E^1$. Let $\Sigma (E^1)$ denote the set of indices $i\in \{ 1,2,3,4\}$ such that
$P_i= E^1_{t_i}$. Then the $\alpha$-stability condition says that 
\begin{equation}
\label{shbstab}
\deg (E^1) -\sum _{i=1}^4\epsilon _i + \sum _{i\in \Sigma (E^1)}2\epsilon _i < 1/2.
\end{equation}

\begin{theorem}
\label{unstabletheorem}
Suppose ${\bf r}\in \Nn ^1_1$ and $\alpha$ is an assignment of parabolic weights,
both nonspecial. Suppose that $\alpha$ is in the (a)-unstable zone, i.e.
condition (a) of Proposition \ref{unstablezones} holds. There is a
set-theoretical isomorphism,
constructibly algebraic but not a morphism of stacks, from the points of $\Pp$ to the
fixed point set of $\Gm$ acting on the moduli space of $\alpha$-stable strictly
parabolic Higgs bundles
$$
{\bf V}_{\alpha}:\Pp \stackrel{\cong}{\rightarrow}(\Mm ^{1,\alpha}_0(0))^{\Gm},
$$
such that 
for any $(E,\nabla ,P_{\hdot})\in \Mm ^{d,\alpha}_1({\bf r})$ we have
$$
\lim _{u\rightarrow 0} (E,u\nabla ,P_{\hdot}) = {\bf V}_{\alpha}(\Upsilon (E,\nabla ,P_{\hdot})).
$$
Here the limit is taken in the $\alpha$-stable Hodge moduli stack 
$\Mm ^{d,\alpha}(\lambda {\bf r})\rightarrow \aaa ^1$. 
\end{theorem}
\begin{proof}
In the (a)-unstable zone, $\epsilon _1+\epsilon _2+\epsilon _3+\epsilon _4 <1/2$
implies that 
$$
\left| -\sum _{i=1}^4\epsilon _i + \sum _{i\in \Sigma (E^1)}2\epsilon _i \right| < 1/2,
$$
so if $\deg (E^1)\geq 1$ then the $\alpha$-stability condition \eqref{shbstab}
never holds, while if $\deg (E^1)\leq 0$ then it always holds. Given that $\deg (E^0)+\deg (E^1)=1$ and $\deg (E^0)\leq \deg (E^1)+2$, the only possibility is $\deg (E^0)=1$ and
$\deg (E^1)=0$. In other words, in this case an $\alpha$-stable system of Hodge
bundles must be of the form
$$
\Oo (1)\stackrel{\theta}{\rightarrow} \Oo \otimes \Omega ^1_X(\log D).
$$
Thus $\theta$ is a section of a line bundle of degree $1$, so it has exactly one zero. 

The condition that $\theta$ be strictly compatible with the parabolic structure means
that if $\theta (t_i)\neq 0$ then $P_i = E^1_{t_i}$. However, if $\theta (t_i)=0$ then
$P_i$ can be either $E^1_{t_i}$ or $E^0_{t_i}$. We see that, set theoretically,
the set of possible choices for $(E,\theta  , P_{\hdot})$ is in bijective correspondence
with the points of $\Pp$. This correspondence is the map ${\bf V}_{\alpha}$. 

Given $(E,\nabla , P_{\hdot})\in \Mm ^{1,\alpha}_0({\bf r})$,
the limit $\lim _{u\rightarrow 0}(E,u\nabla , P_{\hdot})$ is obtained by taking $E^0$ to be the $\alpha$-destabilizing subbundle, $E^1=E/E^0$ and using the map which was previously
denoted $\varphi$ as the Higgs field \cite{idsm}. In view of the definition of 
$\Upsilon$ described above Proposition \ref{iis-picture}, this gives exactly the
required compatibility. 
\end{proof}

\begin{corollary}
\label{foliationunstable}
The foliation conjecture of \cite{idsm} holds for rank two parabolic connections
on $\pp ^1-\{t_1,t_2,t_3,t_4\}$ when the residues and parabolic weights are
nonspecial, and the parabolic weights are in one of the unstable zones.  
\end{corollary}
\begin{proof}
By doing elementary transformations we can reduce to supposing that $\alpha$ is in
the (a)-unstable zone. 
The decomposition into subspaces according to the position of $\lim _{u\rightarrow 0}u()$
is equal to the decomposition into fibers of the map $\Upsilon$, by the preceding theorem.
By Proposition \ref{iis-picture} which recopies 
\cite[Theorem 4.1]{InabaIwasakiSaito2}, this decomposition 
is the decomposition into fibers of a smooth morphism, in particular it is a foliation. 
\end{proof}


\section{The stable zone}
\label{sec-stablezone}

The stable zone will mean when none of (a), (b) or (c) hold, which is to say,
with the nonspeciality hypothesis in effect, that
\begin{equation}
\label{staba}
\epsilon _1+\epsilon _2+\epsilon _3+\epsilon _4 > 1/2;
\end{equation}
\begin{equation}
\label{stabb}
\epsilon _1+\epsilon _2+\epsilon _3+\epsilon _4 < 3/2; 
\end{equation}
and for all renumberings 
$\{ 1,2,3,4\} = \{ i,j,k,l\}$ we have
\begin{equation}
\label{stabc}
\epsilon _i + \epsilon _j -\epsilon _k -\epsilon _l < 1/2.
\end{equation}
Again this is invariant under elementary transformations. If $\alpha$ is an
assignment of parabolic weights in the stable zone, then a general parabolic 
structure on $B$ will be stable, however special ones might not be stable. 

The open subset $\Qq ^{\rm simple}\subset \Qq$ of simple 
quasi-parabolic bundles is preserved by the action of the automorphism
group $A$, and 
$$
\Qq ^{\rm simple}\stackquot A
$$
is the moduli stack of simple quasi-parabolic bundles. Recall from Lemma \ref{simple}
and Corollary \ref{Hfib},
the image of $\Hh \rightarrow \Qq$ is $\Qq ^{\rm simple}$. 

\begin{lemma}\label{ParabolicModuliStack}
The moduli stack $\Qq ^{\rm simple}\stackquot A$ is a $\Gm$-gerb over the non-separated scheme $\Pp$ of Definition \ref{Pdef}. This gerb, which is in fact trivial, 
is the same as Arinkin's stack \cite{Arinkin}. 
\end{lemma}
\begin{proof}
Consider the open set $\Qq ^i\subset \Qq$ consisting of $(u_1,u_2,u_3,u_4)$ such that
$u_j\neq \infty$ for $j\neq i$, and the three corresponding points $U_j$ are not
colinear. The four open sets $\Qq ^i$ cover $\Qq ^{\rm simple}$ from the 
discussion of Lemma \ref{simple}. Fix $U_j^0\in T_j-0$ such that no three of
them is colinear. Any point of $\Qq ^i$ can be brought by a unique element of 
$A$ to a point $(U_1,\ldots , U_4)$ such that $U_j=U_j^0$ for $j\neq i$,
then the position of $U_i\in T_i\cong \pp^1$ provides a coordinate for the quotient $\Qq ^i/A$.
This gives 
$$
\Qq ^i/A  \cong \pp ^1
$$
for each $i$. 
Consider next the intersection $\Qq ^{ij}=\Qq ^i\cap \Qq ^j$. 
Let $U_k$ and $U_l$ be the other two points. Up to the action of $A$,
they may be supposed to lie on the framing points $U^0_k$ and $U^0_l$.
Let $H$ be the line passing through $U^0_k$ and $U^0_l$.
Then $\Qq ^{ij}$ consists of the choices of $U_i\in T_i-0-H\cap T_i$ and $U_j\in T_j-0-H\cap T_j$. The group $A$ acts by scaling both of these. 
Thus, $\Qq ^{ij}/A\cong \Gm$. Glueing together the
two charts $\Qq^i/A$ and $\Qq^j/A$ along the intersection $\Qq ^{ij}/A$ is
therefore 
a doubled projective line
$$
(\Qq ^i\cup \Qq ^j )/A\cong \pp ^1 \cup ^{\Gm}\pp ^1.
$$
It may also be seen as the quotient 
$$
\left(
\pp ^1\times \pp ^1 -\{ (0,0),(0,\infty ), (\infty ,0),(\infty , \infty )\} 
\right) /\Gm .
$$
To get a global picture, fix $i=1$. Now $\Qq ^1/A = \pp ^1$,
a projective line which is identified with $T_1$ when the other three points
are at $U_j^0$. When we glue in $\Qq ^2/A\cong \pp^1$ this doubles up
the origin $0\in T_1$ as well as the intersection point $I_{34}$ of the line $\overline{U^0_3U^0_4}$
with $T_1$. Similarly when we glue in $\Qq ^3/A\cong \pp^1$ it doubles up
the origin (in the same way) and the intersection point $I_{24}$, and
when we glue in $\Qq ^4/A\cong \pp^1$ it doubles up the origin and $I_{23}$. 
One can see that the quadruple of points $(0,I_{34}, I_{24}, I_{23})$ is
equivalent to the original $(t_1,t_2,t_3,t_4)$. Thus 
$$
\Qq ^{\rm simple}/A \cong \Pp .
$$
The gerb is the same as Arinkin's: he was also looking at the moduli
stack of quasiparabolic bundles. These $\Gm$-gerbs are in fact trivial,
as may be seen directly over each chart $\pp ^1$ and on the glueing
from the fact that $\Gm$-torsors over $\Gm$ or $\aaa ^1$ are trivial. 
\end{proof}

The construction using conics described on page
\pageref{conicpage} gives a more canonical $A$-invariant morphism
from $\Qq ^{\rm simple}$ to $\pp ^1$. 

Recall that $\Hh ({\bf r})\rightarrow \Qq$ denotes the moduli space of connections on
the quasi-parabolic bundles parametrized by $\Qq$. Keep the hypothesis that
${\bf r}\in \Nn ^1_1$ is nonspecial. From 
Lemma \ref{simple} it follows that the map
may be written as $\Hh ({\bf r})\rightarrow \Qq ^{\rm simple}$
with $1$-dimensional fibers. We obtain a map
$$
\Mm ^1_1({\bf r})= \Hh ({\bf r})\stackquot A \stackrel{\Phi}{\rightarrow} \Pp .
$$
Our main result identifies this map with the quotient by the
relation defined by Higgs limits under the $\Gm$-action. 

\begin{theorem}
\label{stabletheorem}
Suppose ${\bf r}\in \Nn ^1_1$ and $\alpha$ is an assignment of parabolic weights,
both nonspecial. Suppose that $\alpha$ is in the stable zone, i.e.
\eqref{staba}, \eqref{stabb} and \eqref{stabc} hold. There is a
set-theoretical isomorphism,
constructibly algebraic but not a morphism of stacks, from the points of $\Pp$ to the
fixed point set of $\Gm$ acting on the moduli space of $\alpha$-stable strictly
parabolic Higgs bundles
$$
{\bf V}_{\alpha}:\Pp \stackrel{\cong}{\rightarrow}(\Mm ^{1,\alpha}_0)^{\Gm}(0),
$$
such that 
for any $(E,\nabla ,P_{\hdot})\in \Mm ^{d,\alpha}_1({\bf r})$ we have
$$
\lim _{u\rightarrow 0} (E,u\nabla ,P_{\hdot}) = {\bf V}_{\alpha}(\Phi (E,\nabla ,P_{\hdot})).
$$
Here the limit is taken in the $\alpha$-stable Hodge moduli stack 
$\Mm ^{d,\alpha}(\lambda {\bf r})\rightarrow \aaa ^1$. 
\end{theorem}
\begin{proof}
Recall that $\Pp = \Qq ^{\rm simple}/A$ is the space of $A$-orbits in the simple
quasiparabolic structures, so a point of $\Pp$ represents an isomorphism class of
simple quasiparabolic bundle $(E,P_{\hdot})$ and $\Phi$ is just
the map of forgetting the connection. If $(E,P_{\hdot})$ is $\alpha$-stable,
then take $\theta =0$ as Higgs field 
and set ${\bf V}_{\alpha}(E,P_{\hdot}):= (E,0,P_{\hdot})$. 
If $\nabla$ is any connection on $(E,P_{\hdot})$ then this gives the limiting
$\alpha$-stable Higgs bundle of Proposition \ref{higgslimit}
$$
\lim _{u\rightarrow 0}(E,u\nabla , P_{\hdot})= 
(E,0,P_{\hdot}) = {\bf V}_{\alpha}(E,P_{\hdot}).
$$
It remains to define ${\bf V}_{\alpha}$ on the $(E,P_{\hdot})$ which are
$\alpha$-unstable. Suppose $(E,P_{\hdot})$ is $\alpha$-unstable,
and let $L\subset E$ be the destabilizing subbundle. Since $\alpha$
is in the stable zone, condition \eqref{stabb} says that $L$ is never $\Oo (-1)$.
There are two cases: either $L\cong \Oo$ and there are three $P_i=L_{t_i}$;
or $L=\Oo (1)$ and there is one $P_i=L_{t_i}$. The first case
corresponds to three colinear points $U_i$, while the second case corresponds to
some $U_i$ at the origin. 

The residues of the Higgs field are $0$, which means that 
$\res (\theta , t_i) : E_{t_i}\rightarrow P_i$
and $\res (\theta , t_i):P_i\rightarrow 0$. So we have 
$$
\theta : L\rightarrow (E/L)\otimes \Omega ^1_X(\log D),
$$
which is equal to zero at any point where $P_i=L_{t_i}$. If $L\cong \Oo$
then $\theta$ is a section of a line bundle of degree three with three additional
zeros at the three points $t_i$ with $U_i$ colinear; 
if $L\cong \Oo (1)$ then $\theta$ is a section of
a line bundle of degree $1$ with a single additional zero at the point $t_i$ where
$U_i$ is the origin. In both cases, $\theta$ becomes a nonzero section of
the trivial bundle, in other words it is determined uniquely up to scalar automorphisms
of the two component line bundles. This determines
the Higgs bundle
$$
{\bf V}_{\alpha}(E,P_{\hdot}):= (L\oplus (E/L),\theta )
$$
which will be the limit $\lim _{u\rightarrow 0}(E,u\nabla , P_{\hdot})$
by the construction of Proposition \ref{higgslimit}, for any connection 
$\nabla$ on $(E,P_{\hdot})$. 
\end{proof}

We can be more precise about the possibilities occuring in the above proof. 
There are two points of $\Pp$ over each $t_i\in \pp ^1$. These are the cases 
when $U_i=0$, and when the other three points $U_j,U_k,U_l$ are colinear.
The quasiparabolic structure with $U_i=0$ is unstable if and only if 
$$
1+\epsilon _i -\epsilon _j - \epsilon _k - \epsilon _l > 1/2,
$$
in other words 
$$
\epsilon _j + \epsilon _k + \epsilon _l - \epsilon _i < 1/2.
$$
The quasiparabolic structure with $U_j,U_k,U_l$ colinear is unstable if and
only if 
$$
\epsilon _j + \epsilon _k + \epsilon _l - \epsilon _i > 1/2.
$$
In other words, the point $t_i$ corresponds to the hyperplane 
$\epsilon _j+\epsilon _k + \epsilon _l -\epsilon _i = 1/2$ which divides the
stable zone into two regions, and the question of which of the two points 
lying over $t_i$ is unstable depends on which side of this hyperplane we are on. 
 
The resulting $16$ subzones are quite probably related to the subzones which will show
up as images by the Okamoto symmetry of the various different unstable zones in the
last two sections of the paper.  

\begin{corollary}
\label{foliationstable}
The foliation conjecture of \cite{idsm} holds for rank two parabolic connections
on $\pp ^1-\{t_1,t_2,t_3,t_4\}$ when the residues and parabolic weights are
nonspecial, and the parabolic weights are in the stable zone. 
\end{corollary}
\begin{proof}
By Theorem \ref{stabletheorem}, the pieces of the decomposition according to the
Higgs limit are equal to the fibers of the map
$M ^{d,\alpha}_1({\bf r})\rightarrow \Qq ^{\rm simple}/A = \Pp$. Since this
is a smooth map of schemes, even though the target is non-separated, the collection
of fibers forms a foliation. 
\end{proof}


\section{Local systems on root stacks}
\label{sec-dm}

Consider local systems with monodromy of  finite order around the $t_i$. 
Fix $n\in \nn$  and let 
$$
Z:= X[\frac{D}{n}] \stackrel{p}{\rightarrow} X
$$
be the Cadman-Vistoli root stack, which is the universal Deligne-Mumford
stack over which the line bundle $\Oo (D)$ has an $n$-th root;
a good reference is \cite{Borne}. 
It corresponds to the orbifold obtained
by labeling the points $t_i\in X$ with the integer $n$. The fundamental group
$\pi _1(Z,x)$ is also the orbifold fundamental group of $X$, equivalently
it is $\pi _1(U,x)/\langle \gamma _i^n\rangle$ where $\gamma _i$ are the loops
going around $t_i$. 

In this case the DM-stack $Z$ is a quotient stack. Let $C_n$ be the cyclic
group of order $n$ with generator $c$. Choose a homomorphism $g:\pi _1(U,x)\rightarrow
C_n$ such that $g(\gamma _i)$ is a generator. This exists, for example
we can set $g(\gamma _1)=g(\gamma _2)=c$ and $g(\gamma _3)=g(\gamma _4)= c^{-1}$.
Then $g$ induces a Galois covering $Y\stackrel{q}{\rightarrow} X$ with Galois group $C_n$
and full degree $n$ ramification
over the $t_i$, lifting to an etale Galois covering of the stack $\tilde{q}:Y\rightarrow Z$.
This gives 
$$
Z=Y\stackquot C_n.
$$
Let $\tilde{t}_i\in Y$ be the unique point lying over $t_i\in X$. 

\begin{proposition}
\label{categories}
With the above notations, the following categories are equivalent:
\newline
---local systems on $U$ with finite monodromy of order dividing $n$ around the 
$t_i$;
\newline
---local systems on $Z$;
\newline
---$C_n$-equivariant 
local systems on $Y$. 
\end{proposition}

Given a local system $L$ on $Z$ corresponding to $L_U$ on $U$ and to 
a $C_n$-equivariant local system $L_Y$ on $Y$, we can associate its {\em local monodromy} 
at $t_i$.
This is an object in the category of vector spaces with automorphisms. In terms of $L_U$
it is just the fiber $L_{U,x}$ at the basepoint, together with action of $\gamma _i$.
 
Corresponding to the point $t_i$ is a map ${\bf B}(\zz /n)\rightarrow Z$
from the one-point classifying stack of the cyclic group $\zz /n$ into $Z$,
and in terms of $L$ the local monodromy is the same as the restriction 
$L|_{{\bf B}(\zz /n)}$,
considering a local system over ${\bf B}(\zz /n)$ as being the same as a vector space with
an automorphism of order $n$. 
 
In terms of the $C_n$-equivariant local system $L_Y$
on $Y$, 
the local monodromy is the fiber 
$L_{Y,\tilde{t}_i}$ together with its action of the Galois group
$C_n$, 
but this action
is viewed as an automorphism (i.e. an action of the
local orbifold group $\zz /n$)
using the generating element $g(\gamma _i)\in C_n$. This may be different
from the original generator $c$, which is why we conserved two different
notations $C_n$ and $\zz /n$ for these cyclic groups.

Given a local system $L$ on $Z$, its corresponding sheaf of $\Oo_Z$-modules
is denoted $L\otimes \Oo _Z$. Then 
$$
E:= p_{\ast}(L\otimes \Oo _Z)
$$
is a locally free sheaf on $X$, whose rank is the same as ${\rm rk}(L)$. 
If $L$ corresponds to the $C_n$-equivariant 
local system $L_Y$ on the Galois covering $Y$, with underlying vector bundle $L_Y\otimes \Oo _Y$, then the $C_n$-invariant
part of the direct image is 
$$
E=q_{\ast}(L_Y\otimes \Oo _Y)^{C_n}\subset q_{\ast}(L_Y\otimes \Oo _Y),
$$
indeed since $Y$ provides local charts for the stack $Z$ this may be taken as
the definition of $E$. 

The following proposition is well-known. 

\begin{proposition}
\label{logconnection}
The naturally defined connection on $V|_U$
extends to a logarithmic connection 
$$
\nabla : E\rightarrow E\otimes \Omega ^1_X(\log D).
$$
The residue of $\nabla$ at $t_i$ is semisimple and has eigenvalues in $[0,1)\cap \frac{1}{n}\zz$. More precisely, suppose that the local monodromy of $L$ at $t_i$, in the clockwise direction,
has eigenvalues $e^{\theta _i^j\sqrt{-1}}$
with $0\leq \theta _i^j< 2\pi $ counted with multiplicity. 
Then the residue of $\nabla$ at $t_i$
is semisimple
with eigenvalues $r^{j}_i= \theta_i^j/2\pi $.

This construction sets up an equivalence of categories between 
the category of local systems $L$ on $Z$, and the category of vector bundles with 
logarithmic connection $(E,\nabla )$
whose residues are semisimple with eigenvalues in $[0,1)\cap \frac{1}{n}\zz$.
\end{proposition}

In the situation of the proposition, the bundle $E$ also gets a weighted parabolic structure.
It consists of a quasiparabolic structure or filtration $P^{\hdot}_i$ of $E_{t_i}$,
together with weights $\alpha ^{\hdot}_i\in (-1,0]$. In fact, the filtration is
obtained from the decomposition of $E_{t_i}$ into eigenspaces for ${\rm res}(\nabla )$
and the $j$-th graded piece $Gr^j_{P_i}(E_{t_i})$ is just the $r^j_i$-eigenspace,
weighted by $\alpha ^j_i=-r^j_i$. The index $j$ corresponds to the place of $\alpha ^j_i$
in the increasing order on the interval $(-1,0]$. 

In general, the filtration will not be a full flag.
Say that the local monodromy of $L$
is non-resonant if the eigenvalues of the monodromy transformation
are distinct with multiplicity $1$, corresponding to the same non-resonance condition
for the residue of
the corresponding logarithmic connection. Notice that non-resonance implies
$n\geq {\rm rk}(L)$, otherwise the number of possible available eigenvalues would
be too small.
In the non-resonant case, the parabolic
filtration is a full flag.

Say that a collection of local monodromy data at all the $t_i$ is Kostov-generic
if there is no way of specifying a subset consisting of the same number of eigenvalues
at each point, such that the product over all the points is $1$. Say that the
collection of local monodromy data is nonspecial if it is nonresonant and Kostov-generic. 
This corresponds to the same condition for the logarithmic connection and also 
for the parabolic weights. 

There is a different characterization of the parabolic structure, obtained by looking
at $E$ as $q_{\ast}(L_Y\otimes \Oo _Y)^{C_n}$. Let $y$ be a local coordinate on
$Y$ near $\tilde{t}_i$, then $L_Y\otimes \Oo _Y$ is filtered by the subsheaves
$y^kL_Y\otimes \Oo _Y$. This gives a filtration of $E$
by subsheaves $q_{\ast}(y^kL_Y\otimes \Oo _Y)^{C_n}$.
For $k=n$ the subsheaf is equal to $E(-t_i)$, so for $0\leq k < n$ this defines a subspace
$F^k_i\subset E_{t_i}$. The parabolic subspace $P^j_i$ is defined to be $F^{-n\alpha ^j_i}_i$
where the $\alpha ^j_i$ are the $k/n$ such that the filtration jumps. 

In this point of view, any vector bundle on $Z$ or equivalently $C_n$-equivariant
vector bundle on $Y$ leads to a parabolic bundle on $(X,D)$ with weights 
$\alpha ^j_i\in (-1,0]\cap \frac{1}{n}\zz$. Apply this to vector
bundles with $\lambda$-connection.

\begin{proposition}
\label{lamcon}
The above construction provides an equivalence between the categories of:
\newline
---vector
bundles with $\lambda$-connections on $Z$;
\newline
---$C_n$-equivariant vector
bundles with $\lambda$-connections on $Y$;
\newline
---parabolic bundles $(E,P^{\hdot}_{\hdot}, \alpha ^{\hdot}_{\hdot})$
on $(X,D)$ with weights $\alpha ^j_i\in (-1,0]\cap \frac{1}{n}\zz$
and logarithmic $\lambda$-connection 
$$
\nabla : E\rightarrow E\otimes \Omega ^1_X(\log D)
$$
such that ${\rm res}_{t_i}(\nabla )$ respects the filtration $P^{\hdot}_i$
of $E_{t_i}$ and acts by the scalar $r^j_i=-\lambda \alpha ^j_i$ on $Gr ^j_{P_i}(E_{t_i})$.

For $\lambda = 1$ this correspondence coincides with the correspondences of
Propositions \ref{categories} and \ref{logconnection}.
\end{proposition}

Notice that for $\lambda \neq 0$ the parabolic filtration and weights are
determined by $\nabla$. On the other hand, at $\lambda = 0$ the requirement becomes
just that $\nabla$ acts by $0$ on $Gr ^j_{P_i}(E_{t_i})$, in other words it
respects strictly the parabolic filtration as in \cite{GPGM} for example. 
So for $\lambda = 0$ the connection doesn't determine the weights. 

\begin{lemma}
\label{preservestab}
The correspondence of Proposition \ref{lamcon} is compatible with subobjects
and preserves the degree, using the parabolic degree for parabolic logarithmic
$\lambda$-connections. Hence it preserves stability and semistablity, and induces
an isomorphism between moduli stacks. 
\end{lemma}

Suppose now that $(E,\nabla )$ is a logarithmic connection on $(X,D)$ whose
residues are non-resonant and have rational eigenvalues. Let $n$ be a common
denominator for the eigenvalues. By doing elementary
transformations we may assume that the eigenvalues of the residue lie in
$[0,1)\cap \frac{1}{n}\zz$. From the non-resonance condition, the decomposition
of $E_{t_i}$ into eigenspaces of dimension $1$ induces a
full-flag parabolic structure $P^{\hdot}_{i}$ at $t_i$, and the residues
of $\nabla$ determine the weights $\alpha ^j_i= r^j_i$. 
The degree $d=\deg (E)$ is determined by the Fuchs relation.
We get a point in $\Mm ^d_1({\bf r})$.

If we assume that the residues are nonspecial, then the parabolic weights are
also nonspecial, and our point is stable. We can take the limiting parabolic
Higgs bundle 
$$
\lim _{u\rightarrow 0}(E,u\nabla , P^{\hdot}_{\hdot})\in \Mm ^{d,\alpha}_0({\bf r})^{\Gm}.
$$
which will be stable too (hence unique up to translation of the $\Gm$-action,
see \cite{idsm}).

On the other hand, $(E,\nabla )$ has finite order monodromy so it corresponds to 
a $C_n$-equivariant vector bundle with connection $(E_Y,\nabla _Y)$ on $Y$.
The limit 
$$
\lim _{u\rightarrow 0}(E_Y,u\nabla _Y)
$$
is a $C_n$-equivariant $\Gm$-fixed Higgs bundle on $Y$. Similarly
these correspond to a vector bundle with connection $(E_Z,\nabla _Z)$ on the root stack $Z$
and again the limit $\lim _{u\rightarrow 0}(E_Z,u\nabla _Z)$ is a 
$\Gm$-fixed Higgs bundle on $Z$. 

\begin{lemma}
These three limits are the same via the correspondence of 
Proposition \ref{lamcon}. 
\end{lemma}

The parabolic weights which should be used in order to maintain the correspondence
with bundles on the root stack $Z$ or $C_n$-equivariant bundles on $Y$,
are given by the residues of the connection. These are also given by the
local monodromy operators of the local system. 

Going back to the case of local systems of rank $2$, the parabolic weights
determined by the finite order local monodromy will sometimes be in the
unstable zone, and sometimes in the stable zone. This is the motivation for our
consideration of both zones in the previous discussion. From Corollaries 
\ref{foliationunstable} and \ref{foliationstable} we get the
foliation conjecture for most irreducible components of the
moduli of rank $2$ local systems on $Z$. 

\begin{corollary}
\label{orbicor}
The foliation conjecture of \cite{idsm} holds for the moduli of rank $2$
connections on the orbifold $Z$, at least in the connected components which
correspond to nonspecial local monodromy data.
\end{corollary}


\section{Transversality of the fibrations}
\label{sec-transverse}

Here, we compute the two fibrations defined in the (a)-unstable zone
by the map $\Upsilon$ (see Theorem \ref{unstabletheorem}) 
and in the stable zone by the map $\Phi$ (see Theorem \ref{stabletheorem}). 
We then prove, for Kostov-generic
local exponents ${\bf r}$, that the two fibrations are strongly transversal: generic fibers intersect at one point.
In the next section, we will see that the two fibrations are permuted by an Okamoto symmetry of the moduli space.
A similar description is presented at the end of the paper of Arinkin and Lysenko
in \cite{ArinkinLysenko2}.

Let us first recall the classical construction of canonical coordinates $(p,q)$
on the moduli space $\Mm _1^1({\bf r})$. After twisting by
a convenient logarithmic rank one connection (which has no effect on the 
construction of the two fibrations), we may assume that the local exponents 
are :
\begin{equation}\label{exponent}
(r_1^- , r_1^+,\ldots , r^-_4, r_4^+)=\left(\frac{\kappa_1}{2},-\frac{\kappa_1}{2},\ldots,\frac{\kappa_4}{2}-\frac{1}{2},-\frac{\kappa_4}{2}-\frac{1}{2}\right)
\end{equation}
(note that the last two exponents are shifted by $-\frac{1}{2}$ in order to get a degree $1$ bundle).
We also fix singular points $(t_1,t_2,t_3,t_4)=(0,1,t,\infty)$. For convenience,
denote by $\Mm _1^1(\boldsymbol{\kappa})$ the moduli space of such connections
where $\boldsymbol{\kappa}=(\kappa_1,\ldots,\kappa_4)\in\mathbb C^4$
satisfies Kostov-generic conditions :
\begin{itemize}
\item $\kappa_i\not\in\mathbb Z$ for $i=1,\ldots,4$,
\item $\pm\kappa_1+\cdots+\pm\kappa_4\not\in2\mathbb Z+1$ whatever the signs are.
\end{itemize}
A connection $(E,\nabla , P_{\hdot})\in \Mm _1^1(\boldsymbol{\kappa})$
is therefore irreducible (see Lemma \ref{kgeneric}) and defined on the bundle $E=\Oo \oplus  \Oo (1)$.
Such a connection may be described in the trivialization $\langle e,f\rangle$ used in section 3 
by 
$$
\nabla:Y\mapsto dY+\Omega Y
$$
where $Y=\begin{pmatrix}y_1\\ y_2\end{pmatrix}$ represents the section $y_1e+y_2 f$
and $\Omega=Adx$ is a $2\times 2$-matrix of logarithmic $1$-forms. Being logarithmic
at infinity means that $x(x-1)(x-t)A$ has polynomial coefficients of degree 
$\begin{pmatrix}2&1\\3&2\end{pmatrix}$.
The subbundle $\Oo(1)$ generated
by $f=\begin{pmatrix}0\\ 1\end{pmatrix}$ is not $\nabla$-invariant and the 
$(1,2)$-coefficient vanishes at a single point $x=q\in\pp^1$ (possibly $\infty$). 
This is the apparent singular point of the scalar equation with respect to the cyclic vector $\Oo(1)$: $q$ is the image of the map 
$\Upsilon$ of Theorem \ref{unstabletheorem} and we already get the first fibration.
Assume $q\not=\infty$. After gauge transformation of the form $\begin{pmatrix}\alpha^{-1}&0\\0& \alpha\end{pmatrix}$ we may assume:
\begin{equation}\label{eqA12}
A(1,2)=\frac{x-q}{x(x-1)(x-t)}.
\end{equation}
We can still use a gauge transformation of the form $\begin{pmatrix}1&0\\ \beta_1 x+\beta_0& 1\end{pmatrix}$
to further normalize the connection; this has the effect to add $-\frac{(\beta_1 x+\beta_0)(x-q)}{x(x-1)(x-t)}$ 
to the $(2,2)$-coefficient of $A$. In particular, the value $A(2,2)\vert_{x=q}$ is invariant under gauge freedom.
We set
$$
p:=A(2,2)\vert_{x=q}+\frac{\kappa_1}{2q}+\frac{\kappa_2}{2(q-1)}+\frac{\kappa_3}{2(q-t)}.
$$
More abstractly, at the point $q$ where the subbundle $\Oo (1)$ osculates to the
connection, we can compare the connection with a standard one on $\Oo (1)$
depending on $\boldsymbol{\kappa}$,
and $p$ is the difference. 
Using gauge freedom, we can finally assume
\begin{equation}\label{eqA22}
A(2,2)=p\frac{q(q-1)(q-t)}{x(x-1)(x-t)}-\frac{\kappa_1}{2x}-\frac{\kappa_2}{2(x-1)}-\frac{\kappa_3}{2(x-t)}.
\end{equation}
One can easily check that $A(1,1)+A(2,2)=0$ and that the last coefficient
$$
A(2,1)=\frac{c_1}{x}+\frac{c_2}{x-1}+\frac{c_3}{x-t}+c_4
$$
is determined by specifying the eigenvalues at the four poles.
A straightforward computation shows that the residual matrix $A_i$ at $t_i$ as well as
eigenvectors for $r_i^-$ and $r_i^+$ are respectively given by
$$A_1=\begin{pmatrix}-\frac{\tilde p}{t}+\frac{\kappa_1}{2}&-\frac{q}{t}\\ \frac{\tilde p(\tilde p-t\kappa_1)}{tq}&\frac{\tilde p}{t}-\frac{\kappa_1}{2}\end{pmatrix},\ \ \ 
\begin{pmatrix}1\\ -\frac{\tilde p}{q}\end{pmatrix}\ \ \ \text{and}\ \ \ 
\begin{pmatrix}1\\ -\frac{\tilde p-t\kappa_1}{q}\end{pmatrix}$$
$$A_2=\begin{pmatrix}\frac{\tilde p}{t-1}+\frac{\kappa_2}{2}&\frac{q-1}{t-1}\\ -\frac{\tilde p(\tilde p+(t-1)\kappa_2)}{(t-1)(q-1)}&-\frac{\tilde p}{t-1}-\frac{\kappa_2}{2}\end{pmatrix},\ \ \ 
\begin{pmatrix}1\\ -\frac{\tilde p}{q-1}\end{pmatrix}\ \ \ \text{and}\ \ \ 
\begin{pmatrix}1\\ -\frac{\tilde p+(t-1)\kappa_2}{q-1}\end{pmatrix}$$
$$A_3=\begin{pmatrix}-\frac{\tilde p}{t(t-1)}+\frac{\kappa_3}{2}&-\frac{q-t}{t(t-1)}\\ \frac{\tilde p(\tilde p-t(t-1)\kappa_3)}{t(t-1)(q-t)}&\frac{\tilde p}{t(t-1)}-\frac{\kappa_3}{2}\end{pmatrix},\ \ \ 
\begin{pmatrix}1\\ -\frac{\tilde p}{q-t}\end{pmatrix}\ \ \ \text{and}\ \ \ 
\begin{pmatrix}1\\ -\frac{\tilde p-t(t-1)\kappa_3}{q-t}\end{pmatrix}$$
$$A_4=\begin{pmatrix}\kappa_0+\frac{\kappa_4}{2}-\frac{1}{2}&-1\\ \kappa_0(\kappa_0+\kappa_4)&-\kappa_0-\frac{\kappa_4}{2}-\frac{1}{2}\end{pmatrix},\ \ \ 
\begin{pmatrix}1\\ \kappa_0\end{pmatrix}\ \ \ \text{and}\ \ \ 
\begin{pmatrix}1\\ \kappa_0+\kappa_4\end{pmatrix}$$
where $\tilde p$ and $\kappa_0$ are given by
$$2\kappa_0+\kappa_1+\kappa_2+\kappa_3+\kappa_4=1\ \ \ \text{and}\ \ \ \tilde p=q(q-1)(q-t)p.$$
Here, the residual matrix $A_4$ at $x=\infty$ is computed in the basis $\langle e,xf\rangle$.
The matrix connection $A$ is finally given by
$$A=\frac{A_1}{x}+\frac{A_2}{x-1}+\frac{A_3}{x-t}+C\ \ \ \text{with}\ \ \ C=\begin{pmatrix}0&0\\ -\kappa_0(\kappa_0+\kappa_4)&0\end{pmatrix}.$$
All these formulae make sense 
on the Zariski open subset of the moduli space $M _1^1(\boldsymbol{\kappa})$
defined by $(p,q)\in\mathbb C^2$ and $q\not=0,1,t$. This will be enough to compute and compare the two fibrations.
Note that this formula is simpler than the usual one from Jimbo-Miwa where the matrix at infinity is diagonalized.

In order to compute the $Q$ map defined by the parabolic structure, we consider 
the unique subbundle $\varphi:\Oo(-1)\hookrightarrow\Oo\oplus\Oo(1)$
that contains the parabolic directions over all $4$ points. This line bundle is the destabilizing bundle for the (b)-zone (see section
\ref{sec-HiggsLimit}). That line bundle also
provides the conic discussed in section \ref{sec-ParabolicStructures}, and the unique zero
of the first component of $\varphi$ will coincide with the parameter $Q$ of the moduli
space of parabolic bundles (see discussion following Lemma \ref{simple} on page \pageref{conicpage}). If we denote by $\begin{pmatrix}1\\ u_i\end{pmatrix}$
a generator for the parabolic $P_i$, we find that
the line bundle is generated by the section
$$\tiny{ Y=\begin{pmatrix}((t-1)u_1-tu_2+u_3)x+t(u_2-u_3+(t-1)u_4)\\ \\ \\
((t-1)u_1-tu_2+u_3)u_4x^2\hfill\\
\hskip1cm -(u_1u_2-tu_1u_3+(t-1)u_2u_3+(t^2-1)u_1u_4-t^2u_2u_4+u_3u_4)x\hskip1cm\\
\hfill+tu_1(u_2-u_3+(t-1)u_4)\end{pmatrix} }$$
and we get
$$
Q=-\frac{t(u_2-u_3+(t-1)u_4)}{(t-1)u_1-tu_2+u_3}.
$$
After substituting the values of the parabolics computed from the matrix connection above, we finally obtain
\begin{equation}\label{eq:okamoto}
Q=q+\frac{\kappa_0}{p} .
\end{equation}
In Section \ref{sec-okamoto}, 
we can see that this transformation (\ref{eq:okamoto}) is nothing but the extra Okamoto involution $s_0$ (\ref{eq:ok-inv}).  

Clearly, the $q$-fibration and $Q$-fibration are strongly transversal whenever $\kappa_0\not=1$, which is implied by Kostov-genericity condition. More precisely, although we worked out computations so far on a Zariski open subset of the moduli space, a complete description of it will be given in the next section; it will follow that:
the intersection number $F \cdot L $ of general fibers, $F$ of the $q$-fibration and $L$ of the $Q$-fibration, is one. 
Exceptions are:
\begin{itemize}
\item For general $\lambda\in\pp^1$, fibers $q=\lambda$ and $Q=\lambda$ do not intersect; they do at the infinity in the compactification of the moduli space.
\item For $\lambda=t_i$ one of the poles, fibers split as $F=F^+\sqcup F^-$ and $L=L^+\sqcup L^-$
and have a common component.
\end{itemize}
This will be clarified in the next section.

\begin{theorem}
\label{transverse}
For any $\boldsymbol{\kappa}$ satisfying the Kostov-genericity condition,
the two fibrations defined by $\Phi$ and $\Upsilon$ are strongly transversal. 
\end{theorem}

We end the section by an alternate description of the connection, this time with parameters $p$ and $Q$.
The idea is to normalize first the parabolic structure as $(u_1,u_2,u_3,u_4)=(0,1,u,0)$
with $u=t\frac{Q-1}{Q-t}$ (assuming $P_1$, $P_2$ and $P_4$ not on the same $\mathcal O\subset B$)
which fix the gauge freedom
and next use $p$ as a parameter: $\nabla=\nabla_0+p\cdot\Theta$
where $\nabla_0$ is the unique connection with $p=0$ and $\Theta$ a Higgs field. We find
$\nabla_0=d+\left(\frac{A_1}{x}+\frac{A_2}{x-1}+\frac{A_3}{x-t}\right)dx$
with
$$A_1=\kappa_0\frac{Q-t}{t}\begin{pmatrix}0&1\\ 0&0\end{pmatrix}+\frac{\kappa_1}{2}\begin{pmatrix}1&0\\ 0&-1\end{pmatrix}$$
$$A_2=-\kappa_0\frac{Q-t}{t-1}\begin{pmatrix}1&1\\ -1&-1\end{pmatrix}+\frac{\kappa_2}{2}\begin{pmatrix}1&0\\ -2&-1\end{pmatrix}$$
$$A_3=\kappa_0\frac{Q-t}{t(t-1)}\begin{pmatrix}u&1\\ -u^2&-u\end{pmatrix}+\frac{\kappa_3}{2}\begin{pmatrix}1&0\\ -2u&-1\end{pmatrix}$$
and $\Theta=\left(\frac{\Theta_1}{x}+\frac{\Theta_2}{x-1}+\frac{\Theta_3}{x-t}\right)dx$ with
$$\Theta_1=-\frac{Q(Q-t)}{t}\begin{pmatrix}0&1\\ 0&0\end{pmatrix}$$
$$\Theta_2=\frac{(Q-1)(Q-t)}{t-1}\begin{pmatrix}1&1\\ -1&-1\end{pmatrix}$$
$$\Theta_3=-\frac{(Q-t)^2}{t(t-1)}\begin{pmatrix}u&1\\ -u^2&-u\end{pmatrix}$$
As we can check, $\det(A_i)=-\frac{\kappa_i^2}{4}$, $A_4=-A_1-A_2-A_3=\begin{pmatrix}\frac{1-\kappa_4}{2}&0\\ *&\frac{\kappa_4-1}{2}\end{pmatrix}$, $A(1,2)=p(Q-t)\frac{x-q}{x(x-1)(x-t)}$ and $A(2,2)\vert_{x=q}=p$.


\section{Okamoto symmetries}
\label{sec-okamoto}

In the article \cite{Okamoto3}, Okamoto constructs a group of birational
transformations of the moduli space, generated by elementary
transformations, permutation of poles $t_i$,
and a rather mysterious extra involution denoted $s_0$
in what follows. This group is described in many papers.
Here we follow notations of \cite{InabaIwasakiSaito,InabaIwasakiSaito2} but also use 
the presentation of Noumi-Yamada \cite{NoumiYamada}
for relators.

In order to describe Okamoto symmetries more geometrically, 
recall first the geometry of the moduli space $M_1^1(\boldsymbol{\kappa})$ and its 
natural compactification $\overline{M_1^1(\boldsymbol{\kappa})}$ (\cite{InabaIwasakiSaito2}).

In Theorem 4.1 in \cite{InabaIwasakiSaito2} (which we have already mentioned in
Proposition \ref{iis-picture} above), moduli spaces 
$\overline{M_1^1(\boldsymbol{\kappa})}$ of $\boldsymbol{\alpha}$-stable 
parabolic $\phi$-connections were constructed as follows.
We fix the weight $\boldsymbol{\alpha}$ as in \cite[Theorem 4.1]{InabaIwasakiSaito2}, 
but for simplicity we will not specify them for a while. 
Let us consider the Hirzebruch surface of degree 2 which is  
the $\pp^1$-bundle over $\pp^1$ 
$$
\pi:\Sigma_2 = \pp( \Omega^1_{\pp^1}(D) \oplus \Oo_{\pp^1}) = \pp (\Oo_{\pp^1}(2) 
\oplus \Oo_{\pp^1}) \longrightarrow \pp^1.  
$$
Let $C_0$ be the unique section of $\pi:\Sigma_2 \longrightarrow \pp^1$ with 
the self-intersection number $(C_0)^2 = -2$ 
and $F$ the class of  a general fiber of $\pi$. 
Moreover we have another class of a section $C_1$ of $\pi$ 
with the condition $C_1 \cdot C_0 = 0$. We see that $C_1 \sim C_0 + 2 F$ 
where $\sim$ means the linear equivalence of divisors. 

We fixed four distinct points $t_1,t_2, t_3, t_4$ in $\pp^1$ and 
consider the fibers $F_i = \pi^{-1}(t_i)$, $1 \leq i \leq 4$ .  
Since the data ${\bf r}=\{r_i^{\pm}\} $ are given by 
$\boldsymbol{\kappa} = \{ \kappa_i \}$ 
as in (\ref{exponent}) which are nonspecial, we can define two different points 
$b_i^{\pm}$ in each fiber $F_{i}$ as follows.

Let $e$ be the unit section of $\Oo$, and $f$ be the unit
section of $\Oo(2)$ vanishing twice at $\infty$. denote by 
$q$ the projective variable of $\pp^1$; a point of $\Sigma_2$
over $q\not=\infty$ is given by $e(q)+\tilde p f(q)$, thus 
characterized by $\tilde p\in\pp^1$. In the affine chart $(q,\tilde p)$,
we set
$$
\left\{\begin{matrix}
b_1^-&=&(0,0)\hfill\\ b_1^+&=&(0,t\kappa_1)
\end{matrix}\right.\ \ \ 
\left\{\begin{matrix}
b_2^-&=&(1,0)\hfill\\ b_2^+&=&(1,(1-t)\kappa_2)
\end{matrix}\right.\ \ \ 
\left\{\begin{matrix}
b_3^-&=&(t,0)\hfill\\ b_3^+&=&(t,t(t-1)\kappa_3)
\end{matrix}\right.
$$
Now, let $g$ be the unit section of $\Oo(2)$ vanishing twice at $0$:
$e(q)+\tilde p f(q)=e(q)+\tilde p_\infty g(q)$ where $\tilde p_\infty=\frac{\tilde p}{q^2}$ whenever $q\not=0,\infty$.
In coordinates $(q,\tilde p_\infty)$, we set
$$
\left\{\begin{matrix}
b_4^-&=&(\infty,-\kappa_0)\hfill\\ b_4^+&=&(\infty,-\kappa_0-\kappa_4)
\end{matrix}\right.
$$
(see Figure \ref{fig:hirzebruch}).

\begin{figure}[h]
\begin{center}
\unitlength 0.1in
\begin{picture}(40.55,27.87)(6.30,-28.37)
%
\special{pn 20}%
\special{pa 638 374}%
\special{pa 4194 374}%
\special{fp}%
%
\special{pn 20}%
\special{pa 1116 50}%
\special{pa 1116 2480}%
\special{fp}%
%
\special{pn 20}%
\special{pa 1764 58}%
\special{pa 1764 2488}%
\special{fp}%
%
\special{pn 20}%
\special{pa 2744 50}%
\special{pa 2744 2480}%
\special{fp}%
%
\special{pn 20}%
\special{pa 3554 58}%
\special{pa 3554 2488}%
\special{fp}%
%
\special{pn 8}%
\special{pa 638 2812}%
\special{pa 4194 2812}%
\special{fp}%
%
\special{pn 8}%
\special{pa 4680 374}%
\special{pa 4680 2488}%
\special{fp}%
\special{sh 1}%
\special{pa 4680 2488}%
\special{pa 4700 2421}%
\special{pa 4680 2435}%
\special{pa 4660 2421}%
\special{pa 4680 2488}%
\special{fp}%
%
\special{pn 20}%
\special{sh 1}%
\special{ar 1764 2820 10 10 0  6.28318530717959E+0000}%
\special{sh 1}%
\special{ar 1772 2812 10 10 0  6.28318530717959E+0000}%
%
\special{pn 20}%
\special{sh 1}%
\special{ar 2736 2812 10 10 0  6.28318530717959E+0000}%
\special{sh 1}%
\special{ar 2744 2804 10 10 0  6.28318530717959E+0000}%
%
\special{pn 20}%
\special{sh 1}%
\special{ar 3546 2812 10 10 0  6.28318530717959E+0000}%
\special{sh 1}%
\special{ar 3554 2804 10 10 0  6.28318530717959E+0000}%
%
\special{pn 20}%
\special{sh 1}%
\special{ar 1116 2820 10 10 0  6.28318530717959E+0000}%
\special{sh 1}%
\special{ar 1124 2812 10 10 0  6.28318530717959E+0000}%
\put(21.9300,-29.9000){\makebox(0,0)[lb]{q}}%
\put(10.7600,-30.0700){\makebox(0,0)[lb]{$t_1$}}%
\put(45.8300,-28.4500){\makebox(0,0)[lb]{{\Large $\pp^1$}}}%
%
\special{pn 8}%
\special{pa 638 50}%
\special{pa 4194 50}%
\special{pa 4194 2488}%
\special{pa 638 2488}%
\special{pa 638 50}%
\special{fp}%
%
\special{pn 13}%
\special{pa 638 1889}%
\special{pa 672 1882}%
\special{pa 706 1875}%
\special{pa 740 1869}%
\special{pa 773 1862}%
\special{pa 807 1855}%
\special{pa 841 1848}%
\special{pa 874 1841}%
\special{pa 908 1834}%
\special{pa 941 1826}%
\special{pa 975 1819}%
\special{pa 1008 1812}%
\special{pa 1041 1804}%
\special{pa 1074 1796}%
\special{pa 1107 1788}%
\special{pa 1140 1780}%
\special{pa 1172 1772}%
\special{pa 1204 1763}%
\special{pa 1237 1755}%
\special{pa 1268 1746}%
\special{pa 1300 1737}%
\special{pa 1332 1727}%
\special{pa 1363 1717}%
\special{pa 1394 1707}%
\special{pa 1425 1697}%
\special{pa 1455 1687}%
\special{pa 1485 1676}%
\special{pa 1515 1664}%
\special{pa 1545 1653}%
\special{pa 1574 1641}%
\special{pa 1603 1629}%
\special{pa 1632 1616}%
\special{pa 1660 1603}%
\special{pa 1688 1589}%
\special{pa 1715 1575}%
\special{pa 1742 1561}%
\special{pa 1769 1546}%
\special{pa 1795 1531}%
\special{pa 1821 1515}%
\special{pa 1847 1499}%
\special{pa 1872 1483}%
\special{pa 1896 1465}%
\special{pa 1921 1448}%
\special{pa 1944 1430}%
\special{pa 1967 1411}%
\special{pa 1990 1391}%
\special{pa 2012 1372}%
\special{pa 2034 1351}%
\special{pa 2055 1330}%
\special{pa 2076 1309}%
\special{pa 2096 1286}%
\special{pa 2116 1264}%
\special{pa 2135 1241}%
\special{pa 2154 1217}%
\special{pa 2172 1193}%
\special{pa 2190 1169}%
\special{pa 2208 1144}%
\special{pa 2225 1118}%
\special{pa 2242 1092}%
\special{pa 2259 1066}%
\special{pa 2275 1040}%
\special{pa 2291 1012}%
\special{pa 2306 985}%
\special{pa 2321 957}%
\special{pa 2336 929}%
\special{pa 2351 901}%
\special{pa 2365 872}%
\special{pa 2379 843}%
\special{pa 2393 813}%
\special{pa 2406 784}%
\special{pa 2419 754}%
\special{pa 2432 723}%
\special{pa 2445 693}%
\special{pa 2457 662}%
\special{pa 2470 631}%
\special{pa 2482 600}%
\special{pa 2494 569}%
\special{pa 2505 537}%
\special{pa 2517 505}%
\special{pa 2528 473}%
\special{pa 2540 441}%
\special{pa 2551 409}%
\special{pa 2562 377}%
\special{pa 2573 344}%
\special{pa 2584 312}%
\special{pa 2594 279}%
\special{pa 2605 246}%
\special{pa 2616 213}%
\special{pa 2626 181}%
\special{pa 2637 148}%
\special{pa 2647 115}%
\special{pa 2658 82}%
\special{pa 2663 66}%
\special{sp}%
%
\special{pn 13}%
\special{pa 2574 2480}%
\special{pa 2576 2442}%
\special{pa 2578 2403}%
\special{pa 2580 2365}%
\special{pa 2582 2327}%
\special{pa 2584 2289}%
\special{pa 2586 2251}%
\special{pa 2589 2214}%
\special{pa 2592 2177}%
\special{pa 2595 2141}%
\special{pa 2598 2105}%
\special{pa 2602 2070}%
\special{pa 2607 2036}%
\special{pa 2611 2002}%
\special{pa 2617 1969}%
\special{pa 2623 1937}%
\special{pa 2629 1906}%
\special{pa 2636 1876}%
\special{pa 2644 1847}%
\special{pa 2652 1819}%
\special{pa 2661 1793}%
\special{pa 2671 1767}%
\special{pa 2682 1743}%
\special{pa 2693 1720}%
\special{pa 2706 1699}%
\special{pa 2719 1679}%
\special{pa 2733 1660}%
\special{pa 2749 1644}%
\special{pa 2765 1628}%
\special{pa 2782 1615}%
\special{pa 2800 1603}%
\special{pa 2820 1592}%
\special{pa 2840 1583}%
\special{pa 2861 1575}%
\special{pa 2882 1569}%
\special{pa 2905 1564}%
\special{pa 2928 1560}%
\special{pa 2953 1558}%
\special{pa 2978 1556}%
\special{pa 3004 1556}%
\special{pa 3030 1558}%
\special{pa 3057 1560}%
\special{pa 3085 1564}%
\special{pa 3114 1568}%
\special{pa 3143 1574}%
\special{pa 3173 1580}%
\special{pa 3204 1588}%
\special{pa 3235 1597}%
\special{pa 3267 1606}%
\special{pa 3299 1616}%
\special{pa 3331 1628}%
\special{pa 3365 1640}%
\special{pa 3398 1652}%
\special{pa 3433 1666}%
\special{pa 3467 1680}%
\special{pa 3502 1695}%
\special{pa 3538 1710}%
\special{pa 3573 1727}%
\special{pa 3610 1743}%
\special{pa 3646 1760}%
\special{pa 3683 1778}%
\special{pa 3720 1796}%
\special{pa 3757 1815}%
\special{pa 3795 1834}%
\special{pa 3833 1853}%
\special{pa 3871 1873}%
\special{pa 3909 1893}%
\special{pa 3947 1913}%
\special{pa 3986 1933}%
\special{pa 4024 1954}%
\special{pa 4063 1974}%
\special{pa 4102 1995}%
\special{pa 4141 2016}%
\special{pa 4179 2037}%
\special{pa 4194 2045}%
\special{sp}%
\put(16.9100,-29.9800){\makebox(0,0)[lb]{$t_2$}}%
\put(26.8700,-29.9000){\makebox(0,0)[lb]{$t_3$}}%
\put(34.9700,-29.9000){\makebox(0,0)[lb]{$t_4$}}%
\put(6.8700,-13.1400){\makebox(0,0)[lb]{$\tilde{p}$}}%
%
\special{pn 20}%
\special{sh 1}%
\special{ar 1116 690 10 10 0  6.28318530717959E+0000}%
\special{sh 1}%
\special{ar 1116 690 10 10 0  6.28318530717959E+0000}%
%
\special{pn 20}%
\special{sh 1}%
\special{ar 1764 690 10 10 0  6.28318530717959E+0000}%
\special{sh 1}%
\special{ar 1764 690 10 10 0  6.28318530717959E+0000}%
%
\special{pn 20}%
\special{sh 1}%
\special{ar 2744 982 10 10 0  6.28318530717959E+0000}%
\special{sh 1}%
\special{ar 2744 982 10 10 0  6.28318530717959E+0000}%
%
\special{pn 20}%
\special{sh 1}%
\special{ar 3554 625 10 10 0  6.28318530717959E+0000}%
\special{sh 1}%
\special{ar 3554 625 10 10 0  6.28318530717959E+0000}%
\put(9.5400,-6.9800){\makebox(0,0)[lb]{$b^+_1$}}%
\put(13.9100,-3.4200){\makebox(0,0)[lb]{$C_0$}}%
\put(18.2100,-6.8200){\makebox(0,0)[lb]{$b^+_2$}}%
\put(28.1700,-9.5700){\makebox(0,0)[lb]{$b^+_3$}}%
\put(36.4300,-5.7700){\makebox(0,0)[lb]{$b^+_4$}}%
\put(46.0000,-3.0000){\makebox(0,0)[lb]{{\Large $\Sigma_2$}}}%
\put(11.4800,-14.3500){\makebox(0,0)[lb]{$(b')^-_1$}}%
\put(15.3700,-17.5100){\makebox(0,0)[lb]{$b^-_2$}}%
\put(20.4000,-15.5000){\makebox(0,0)[lb]{$C_{234}$}}%
\put(25.0000,-18.0000){\makebox(0,0)[lb]{$b^-_3$}}%
\put(32.7900,-20.7500){\makebox(0,0)[lb]{$b^-_4$}}%
\put(11.8100,-24.4000){\makebox(0,0)[lb]{$F_1$}}%
\put(17.9600,-24.4000){\makebox(0,0)[lb]{$F_2$}}%
\put(28.0100,-24.4000){\makebox(0,0)[lb]{$F_3$}}%
\put(36.2700,-24.3100){\makebox(0,0)[lb]{$F_4$}}%
%
\special{pn 8}%
\special{pa 646 1435}%
\special{pa 678 1439}%
\special{pa 710 1443}%
\special{pa 742 1446}%
\special{pa 774 1450}%
\special{pa 806 1454}%
\special{pa 837 1458}%
\special{pa 869 1461}%
\special{pa 901 1465}%
\special{pa 933 1469}%
\special{pa 965 1472}%
\special{pa 997 1476}%
\special{pa 1029 1479}%
\special{pa 1061 1483}%
\special{pa 1093 1487}%
\special{pa 1124 1490}%
\special{pa 1156 1494}%
\special{pa 1188 1497}%
\special{pa 1220 1500}%
\special{pa 1252 1504}%
\special{pa 1284 1507}%
\special{pa 1316 1510}%
\special{pa 1347 1513}%
\special{pa 1379 1517}%
\special{pa 1411 1520}%
\special{pa 1443 1523}%
\special{pa 1475 1526}%
\special{pa 1506 1528}%
\special{pa 1538 1531}%
\special{pa 1570 1534}%
\special{pa 1602 1537}%
\special{pa 1633 1539}%
\special{pa 1665 1542}%
\special{pa 1697 1544}%
\special{pa 1729 1546}%
\special{pa 1760 1549}%
\special{pa 1792 1551}%
\special{pa 1824 1553}%
\special{pa 1855 1555}%
\special{pa 1887 1557}%
\special{pa 1919 1559}%
\special{pa 1950 1561}%
\special{pa 1982 1563}%
\special{pa 2014 1565}%
\special{pa 2045 1567}%
\special{pa 2077 1569}%
\special{pa 2108 1571}%
\special{pa 2140 1574}%
\special{pa 2172 1576}%
\special{pa 2204 1578}%
\special{pa 2235 1581}%
\special{pa 2267 1584}%
\special{pa 2299 1587}%
\special{pa 2331 1590}%
\special{pa 2362 1593}%
\special{pa 2394 1597}%
\special{pa 2426 1601}%
\special{pa 2458 1605}%
\special{pa 2490 1609}%
\special{pa 2522 1614}%
\special{pa 2554 1619}%
\special{pa 2586 1624}%
\special{pa 2618 1630}%
\special{pa 2650 1636}%
\special{pa 2682 1642}%
\special{pa 2714 1649}%
\special{pa 2747 1656}%
\special{pa 2779 1664}%
\special{pa 2811 1672}%
\special{pa 2844 1680}%
\special{pa 2876 1688}%
\special{pa 2908 1697}%
\special{pa 2941 1705}%
\special{pa 2973 1713}%
\special{pa 3005 1721}%
\special{pa 3037 1729}%
\special{pa 3069 1736}%
\special{pa 3101 1743}%
\special{pa 3133 1749}%
\special{pa 3165 1755}%
\special{pa 3196 1760}%
\special{pa 3227 1765}%
\special{pa 3259 1768}%
\special{pa 3290 1771}%
\special{pa 3320 1772}%
\special{pa 3351 1772}%
\special{pa 3381 1772}%
\special{pa 3411 1769}%
\special{pa 3441 1766}%
\special{pa 3470 1761}%
\special{pa 3499 1755}%
\special{pa 3528 1747}%
\special{pa 3557 1738}%
\special{pa 3585 1728}%
\special{pa 3614 1717}%
\special{pa 3642 1705}%
\special{pa 3669 1692}%
\special{pa 3697 1677}%
\special{pa 3724 1662}%
\special{pa 3751 1646}%
\special{pa 3778 1629}%
\special{pa 3805 1611}%
\special{pa 3832 1592}%
\special{pa 3858 1573}%
\special{pa 3885 1553}%
\special{pa 3911 1532}%
\special{pa 3937 1511}%
\special{pa 3963 1489}%
\special{pa 3989 1467}%
\special{pa 4015 1445}%
\special{pa 4041 1422}%
\special{pa 4067 1398}%
\special{pa 4093 1375}%
\special{pa 4118 1351}%
\special{pa 4144 1328}%
\special{pa 4170 1304}%
\special{pa 4194 1281}%
\special{sp 0.070}%
\put(18.3700,-10.4600){\makebox(0,0)[lb]{$C_1+F$}}%
\put(9.4600,-19.0500){\makebox(0,0)[lb]{$b^-_1$}}%
\put(24.3000,-3.2000){\makebox(0,0)[lb]{$Q$}}%
%
\special{pn 20}%
\special{sh 1}%
\special{ar 1108 1492 10 10 0  6.28318530717959E+0000}%
\special{sh 1}%
\special{ar 1116 1492 10 10 0  6.28318530717959E+0000}%
%
\special{pn 20}%
\special{sh 1}%
\special{ar 2558 358 10 10 0  6.28318530717959E+0000}%
\special{sh 1}%
\special{ar 2558 358 10 10 0  6.28318530717959E+0000}%
\put(30.2000,-15.1000){\makebox(0,0)[lb]{$C_1+F$}}%
%
\special{pn 8}%
\special{pa 646 2480}%
\special{pa 638 860}%
\special{fp}%
\special{sh 1}%
\special{pa 638 860}%
\special{pa 618 927}%
\special{pa 638 913}%
\special{pa 658 927}%
\special{pa 638 860}%
\special{fp}%
%
\special{pn 8}%
\special{pa 630 2812}%
\special{pa 2201 2804}%
\special{fp}%
\special{sh 1}%
\special{pa 2201 2804}%
\special{pa 2134 2784}%
\special{pa 2148 2804}%
\special{pa 2134 2824}%
\special{pa 2201 2804}%
\special{fp}%
%
\special{pn 13}%
\special{pa 2390 60}%
\special{pa 2380 2490}%
\special{fp}%
\put(22.2000,-12.8000){\makebox(0,0)[lb]{$F$}}%
\end{picture}%

\end{center}
\caption{Hirzebruch surface $\Sigma_2$}
\label{fig:hirzebruch}
\end{figure}
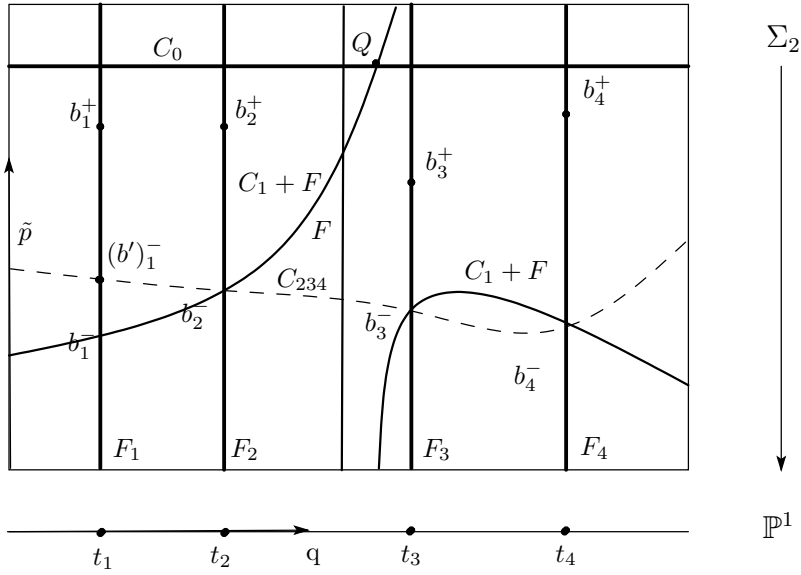

Blowing up these $8$ points $ \{ b_i^{\pm} \}_{1 \leq i \leq 4} $  of $\Sigma_{2}$,  we obtain a morphism 
$$
f_{\boldsymbol{\kappa}}:S_{\boldsymbol{\kappa}} = \tilde{\Sigma}_{2, \boldsymbol{\kappa}} \longrightarrow \Sigma_2
$$ 
where $S_{\boldsymbol{\kappa}} $ is a smooth rational surface. 
We set $E_i^{\pm} = f_{\boldsymbol{\kappa}}^{-1}(b_i^{\pm})$ the exceptional curves of $f$, and we denote by
$F'_i $ the proper transform of $F_i$.  Then 
one can see that the Picard group of $S_{\boldsymbol{\kappa}}$ is 
generated by the classes of $C_0, F, E_{1}^{\pm}, \cdots, E_{4}^{\pm}$, and moreover 
the anti-canonical class $-K_{S_{\boldsymbol{\kappa}}}$ has a unique 
effective member 
\begin{equation}\label{anti}
 Y = 2 C_0 + F'_1+ F'_2+F'_3+F'_4. 
\end{equation}
The pair $(S_{\boldsymbol{\kappa}}, Y)$ is an 
Okamoto-Painlev\'e pair in the sense of 
\cite{SaitoTakebeTerajima} (see also \cite{Sakai}), which means that the rational surface $S_{\boldsymbol{\kappa}}$ has a rational two form $\omega$ (unique up to non-zero constants) whose pole divisor is given by $Y$ with the conditions $Y \cdot C_0 = Y \cdot F'_i = 0$ for $1 \leq i \leq 4$.  Precisely, we have 
$$
\omega=dp\wedge dq\ \ \ \text{with}\ \tilde p=q(q-1)(q-t)p.
$$
Note that the complement $S_{\boldsymbol{\kappa}} \setminus Y$ has a holomorphic symplectic structure induced by $\omega$.  
Then in \cite{InabaIwasakiSaito2}, we have the following isomorphisms
$$
\begin{array}{ccc}
\overline{M_1^1(\boldsymbol{\kappa})} & \simeq  & S_{\boldsymbol{\kappa}}   \\
\cup &   & \cup \\
M_1^1(\boldsymbol{\kappa}) &  \simeq &  S_{\boldsymbol{\kappa}} \setminus Y.  \\
\end{array}
$$
The apparent singularity map 
$\Upsilon:\Mm^1_1(\boldsymbol{\kappa})\rightarrow \Pp $ in Section \ref{sec-unstablezone} 
induces a morphism 
$$
M^1_1 (\boldsymbol{\kappa} ) \rightarrow \Pp \rightarrow \pp^1 
$$
which can be identified with the natural map $ \pi_{1, \boldsymbol{\kappa}}=
\pi \circ f_{\boldsymbol{\kappa}} :M^1_1 (\boldsymbol{\kappa} ) \simeq 
S_{\boldsymbol{\kappa}} \setminus Y \rightarrow \Sigma_2 \rightarrow \pp^1$.  
One can easily see that  $\pi_{1, \boldsymbol{\kappa}}$ 
can be extended to the natural morphism 
$$
\pi_{1, \boldsymbol{\kappa}}:
\overline{M_1^1(\boldsymbol{\kappa})}  \simeq   S_{\boldsymbol{\kappa}} \rightarrow \pp^1. 
$$ Note that this morphism $\pi_{1, \boldsymbol{\kappa}}$ is the morphism 
induced by the linear system $|F| \simeq \pp^1$.  
In the Picard group of $\overline{M_1^1(\boldsymbol{\kappa})}  \simeq   S_{\boldsymbol{\kappa}}$, we have the relations
\begin{equation}\label{eq:reducible}
F \sim F'_i + E_i^{+} + E_i^{-} \quad  \mbox{ for each  $i$,  }  1 \leq i \leq 4
\end{equation}
which correspond to four singular fibers 
$ F'_i + E_i^{+} + E_i^{-}$ of the morphism $\pi_{1, \boldsymbol{\kappa}}$ (see Figure \ref{fig:b-hirzebruch}). 
Moreover  on a certain Zariski open set of $M_1^1(\boldsymbol{\kappa})$, it coincides with the natural 
projection $(p,q) \mapsto q$ as in Section \ref{sec-transverse}.  

On the other hand, we also have the natural morphism 
$ \Phi: \Mm _1^1(\boldsymbol{\kappa}) \rightarrow \Pp $ in 
Section \ref{sec-stablezone}, where 
in this case $ \Pp $ can be identified with the moduli space  $Q^{simple}/A$ of simple quasiparabolic bundles (cf. Lemma \ref{staba}).  
This induces another natural morphism
$\pi_{2, \boldsymbol{\kappa}}:M_1^1(\boldsymbol{\kappa}) \rightarrow \Pp \rightarrow \pp^1 $ 
which can be identified with the map $(q, \tilde{p}) \mapsto Q$ given by  (\ref{eqA12}) 
on a Zariski open set of 
$M_1^1(\boldsymbol{\kappa})$.  
From the construction of $\overline{M_1^1(\boldsymbol{\kappa})}$ as above, we see that 
 $\pi_{2, \boldsymbol{\kappa}}$ can be  extended to a morphism  $\pi_{2, \boldsymbol{\kappa}}:
\overline{M_1^1(\boldsymbol{\kappa})} \rightarrow \pp^1$.  
Let us denote by $L$ the class of general fiber of $\pi_{2, \boldsymbol{\kappa}}:
\overline{M_1^1(\boldsymbol{\kappa})} \rightarrow \pp^1$. 
Then from the calculation in Section \ref{sec-transverse},  we see that
\begin{equation}\label{eq:parafiber}
L \sim C_1 + F - E_{1}^{-} - E_{2}^{-}-E_{3}^{-} - E_{4}^{-}. 
\end{equation}
A general member $L$  of the linear system $|L|$ is isomorphic to 
$\pp^1$ and interesection numbers of related divisors are given as follows:
$$
L \cdot C_0 = 1, \ \  L \cdot E_{i}^{-} = 1, \ \   L \cdot F = 1.  
$$

For each $i$, $1 \leq i \leq 4$, let $\{j, k, l \}$ be the complement of 
$i$ in $\{1, 2, 3, 4\}$.  Then there exists an irreducible curve $C_{jkl}$ in $\Sigma_2$ 
such that $C_{jkl} \sim C_1$ and $C_{jkl}$ is passing through 
3 points $b_{j}^{-}, b_k^{-}, b_l^{-}$.
Let $(E')_i^-$ be the proper transform of $C_{jkl} \in \overline{M_1^1(\boldsymbol{\kappa})}$ by the blowing up $\overline{M_1^1(\boldsymbol{\kappa})} \rightarrow \Sigma_2$.  Then we see that 
$$
(E') _i^{-} \sim C_1 - E_j^{-} - E_{k}^{-} - E_{l}^{-},  
$$
and $((E')_i^{-})^2 = -1$, that is, $(E')_i^{-}$ is a $(-1)$-exceptional curve.  
The intersection point of $C_{jkl}$ with $F_i$ is denote by $(b')_i^{-}$.  Note that the morphism $\pi_{2, \boldsymbol{\kappa}}$ has also exactly 4 singular fibers, which 
correponds to the following linear equivalences of divisors 
\begin{equation}\label{fiberation2}
L \sim  F'_i + (E')_i^{-}+ E_i^{+} \quad \mbox{for each $i$ \quad } 1 \leq i \leq 4. 
\end{equation}

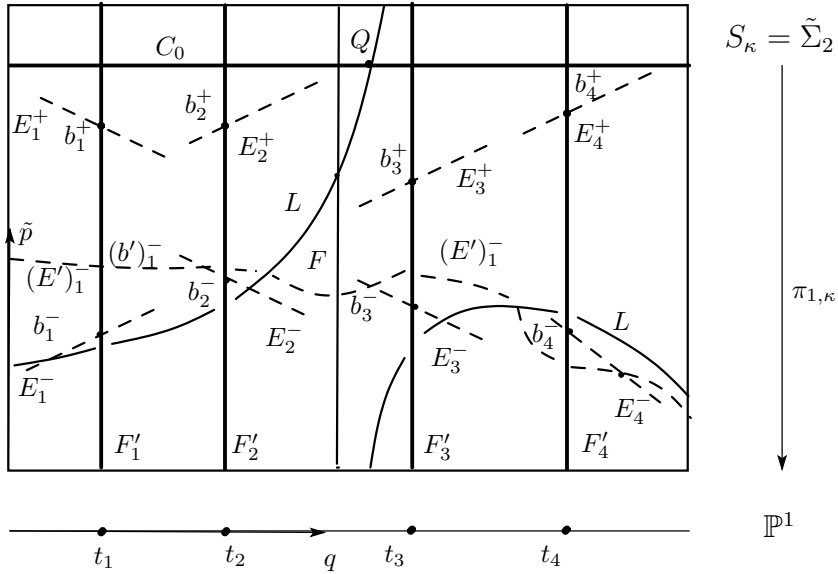
\begin{figure}[b]
\begin{center}
\unitlength 0.1in
\begin{picture}(40.99,41.31)(6.30,-41.51)
%
\special{pn 20}%
\special{pa 638 344}%
\special{pa 4194 344}%
\special{fp}%
%
\special{pn 20}%
\special{pa 1116 20}%
\special{pa 1116 2450}%
\special{fp}%
%
\special{pn 20}%
\special{pa 1764 28}%
\special{pa 1764 2458}%
\special{fp}%
%
\special{pn 20}%
\special{pa 2744 20}%
\special{pa 2744 2450}%
\special{fp}%
%
\special{pn 20}%
\special{pa 3554 28}%
\special{pa 3554 2458}%
\special{fp}%
%
\special{pn 8}%
\special{pa 638 2782}%
\special{pa 4194 2782}%
\special{fp}%
%
\special{pn 8}%
\special{pa 4680 344}%
\special{pa 4680 2458}%
\special{fp}%
\special{sh 1}%
\special{pa 4680 2458}%
\special{pa 4700 2391}%
\special{pa 4680 2405}%
\special{pa 4660 2391}%
\special{pa 4680 2458}%
\special{fp}%
%
\special{pn 20}%
\special{sh 1}%
\special{ar 1764 2790 10 10 0  6.28318530717959E+0000}%
\special{sh 1}%
\special{ar 1772 2782 10 10 0  6.28318530717959E+0000}%
%
\special{pn 20}%
\special{sh 1}%
\special{ar 2736 2782 10 10 0  6.28318530717959E+0000}%
\special{sh 1}%
\special{ar 2744 2774 10 10 0  6.28318530717959E+0000}%
%
\special{pn 20}%
\special{sh 1}%
\special{ar 3546 2782 10 10 0  6.28318530717959E+0000}%
\special{sh 1}%
\special{ar 3554 2774 10 10 0  6.28318530717959E+0000}%
%
\special{pn 20}%
\special{sh 1}%
\special{ar 1116 2790 10 10 0  6.28318530717959E+0000}%
\special{sh 1}%
\special{ar 1124 2782 10 10 0  6.28318530717959E+0000}%
\put(22.8200,-29.9300){\makebox(0,0)[lb]{$q$}}%
\put(10.7600,-29.7700){\makebox(0,0)[lb]{$t_1$}}%
\put(45.8300,-28.1500){\makebox(0,0)[lb]{{\Large $\pp^1$}}}%
%
\special{pn 13}%
\special{pa 630 28}%
\special{pa 4186 28}%
\special{pa 4186 2466}%
\special{pa 630 2466}%
\special{pa 630 28}%
\special{fp}%
\put(17.6400,-29.6800){\makebox(0,0)[lb]{$t_2$}}%
\put(25.9800,-29.6800){\makebox(0,0)[lb]{$t_3$}}%
\put(34.1600,-29.6800){\makebox(0,0)[lb]{$t_4$}}%
\put(6.8700,-12.8400){\makebox(0,0)[lb]{$\tilde{p}$}}%
%
\special{pn 20}%
\special{sh 1}%
\special{ar 1116 660 10 10 0  6.28318530717959E+0000}%
\special{sh 1}%
\special{ar 1116 660 10 10 0  6.28318530717959E+0000}%
%
\special{pn 20}%
\special{sh 1}%
\special{ar 1764 660 10 10 0  6.28318530717959E+0000}%
\special{sh 1}%
\special{ar 1764 660 10 10 0  6.28318530717959E+0000}%
%
\special{pn 20}%
\special{sh 1}%
\special{ar 2744 952 10 10 0  6.28318530717959E+0000}%
\special{sh 1}%
\special{ar 2744 952 10 10 0  6.28318530717959E+0000}%
%
\special{pn 20}%
\special{sh 1}%
\special{ar 3554 595 10 10 0  6.28318530717959E+0000}%
\special{sh 1}%
\special{ar 3554 595 10 10 0  6.28318530717959E+0000}%
\put(9.2200,-7.9000){\makebox(0,0)[lb]{$b^+_1$}}%
\put(14.0000,-3.1200){\makebox(0,0)[lb]{$C_0$}}%
\put(15.5300,-6.2800){\makebox(0,0)[lb]{$b^+_2$}}%
\put(25.7400,-9.1100){\makebox(0,0)[lb]{$b^+_3$}}%
\put(35.9500,-5.2200){\makebox(0,0)[lb]{$b^+_4$}}%
\put(43.8000,-2.7100){\makebox(0,0)[lb]{{\Large $S_{\kappa}=\tilde{\Sigma}_2$}}}%
\put(11.4800,-14.0500){\makebox(0,0)[lb]{$(b')^-_1$}}%
\put(15.7000,-16.0800){\makebox(0,0)[lb]{$b^-_2$}}%
\put(24.2000,-16.8100){\makebox(0,0)[lb]{$b^-_3$}}%
\put(33.7000,-18.3000){\makebox(0,0)[lb]{$b^-_4$}}%
\put(11.8100,-24.1000){\makebox(0,0)[lb]{$F'_1$}}%
\put(17.9600,-24.1000){\makebox(0,0)[lb]{$F'_2$}}%
\put(28.0100,-24.1000){\makebox(0,0)[lb]{$F'_3$}}%
\put(36.2700,-24.0100){\makebox(0,0)[lb]{$F'_4$}}%
%
\special{pn 13}%
\special{pa 662 1915}%
\special{pa 694 1910}%
\special{pa 725 1905}%
\special{pa 757 1900}%
\special{pa 788 1894}%
\special{pa 820 1889}%
\special{pa 851 1883}%
\special{pa 883 1877}%
\special{pa 914 1870}%
\special{pa 945 1863}%
\special{pa 976 1855}%
\special{pa 1007 1847}%
\special{pa 1038 1839}%
\special{pa 1069 1830}%
\special{pa 1084 1826}%
\special{sp}%
%
\special{pn 13}%
\special{pa 1173 1802}%
\special{pa 1205 1796}%
\special{pa 1236 1789}%
\special{pa 1268 1782}%
\special{pa 1299 1776}%
\special{pa 1330 1768}%
\special{pa 1362 1761}%
\special{pa 1393 1753}%
\special{pa 1423 1744}%
\special{pa 1454 1735}%
\special{pa 1484 1725}%
\special{pa 1514 1714}%
\special{pa 1544 1702}%
\special{pa 1573 1690}%
\special{pa 1602 1676}%
\special{pa 1631 1663}%
\special{pa 1660 1648}%
\special{pa 1688 1634}%
\special{pa 1707 1624}%
\special{sp}%
%
\special{pn 13}%
\special{pa 1821 1567}%
\special{pa 1846 1547}%
\special{pa 1872 1527}%
\special{pa 1897 1507}%
\special{pa 1922 1487}%
\special{pa 1947 1466}%
\special{pa 1972 1446}%
\special{pa 1996 1425}%
\special{pa 2019 1403}%
\special{pa 2042 1382}%
\special{pa 2064 1360}%
\special{pa 2086 1337}%
\special{pa 2107 1314}%
\special{pa 2127 1290}%
\special{pa 2147 1266}%
\special{pa 2166 1242}%
\special{pa 2185 1217}%
\special{pa 2203 1191}%
\special{pa 2220 1165}%
\special{pa 2237 1139}%
\special{pa 2253 1113}%
\special{pa 2269 1086}%
\special{pa 2284 1058}%
\special{pa 2299 1031}%
\special{pa 2313 1003}%
\special{pa 2327 974}%
\special{pa 2340 946}%
\special{pa 2353 917}%
\special{pa 2366 887}%
\special{pa 2378 858}%
\special{pa 2390 828}%
\special{pa 2401 798}%
\special{pa 2412 767}%
\special{pa 2422 737}%
\special{pa 2433 706}%
\special{pa 2443 675}%
\special{pa 2452 643}%
\special{pa 2462 612}%
\special{pa 2471 580}%
\special{pa 2480 548}%
\special{pa 2488 516}%
\special{pa 2497 484}%
\special{pa 2505 451}%
\special{pa 2513 419}%
\special{pa 2521 386}%
\special{pa 2528 353}%
\special{pa 2536 321}%
\special{pa 2543 288}%
\special{pa 2551 254}%
\special{pa 2558 221}%
\special{pa 2565 188}%
\special{pa 2572 155}%
\special{pa 2579 122}%
\special{pa 2586 88}%
\special{pa 2592 55}%
\special{pa 2598 28}%
\special{sp}%
%
\special{pn 13}%
\special{pa 2525 2450}%
\special{pa 2530 2418}%
\special{pa 2535 2386}%
\special{pa 2541 2354}%
\special{pa 2546 2322}%
\special{pa 2552 2290}%
\special{pa 2558 2258}%
\special{pa 2565 2227}%
\special{pa 2572 2196}%
\special{pa 2580 2165}%
\special{pa 2589 2134}%
\special{pa 2598 2104}%
\special{pa 2608 2074}%
\special{pa 2619 2044}%
\special{pa 2632 2014}%
\special{pa 2645 1986}%
\special{pa 2659 1957}%
\special{pa 2674 1929}%
\special{pa 2690 1901}%
\special{pa 2704 1876}%
\special{sp}%
%
\special{pn 13}%
\special{pa 2809 1745}%
\special{pa 2835 1726}%
\special{pa 2862 1707}%
\special{pa 2889 1690}%
\special{pa 2916 1673}%
\special{pa 2944 1658}%
\special{pa 2972 1645}%
\special{pa 3001 1634}%
\special{pa 3030 1625}%
\special{pa 3061 1617}%
\special{pa 3091 1612}%
\special{pa 3122 1608}%
\special{pa 3154 1606}%
\special{pa 3186 1604}%
\special{pa 3218 1604}%
\special{pa 3250 1606}%
\special{pa 3283 1608}%
\special{pa 3317 1611}%
\special{pa 3350 1614}%
\special{pa 3383 1619}%
\special{pa 3417 1623}%
\special{pa 3451 1628}%
\special{pa 3485 1633}%
\special{pa 3502 1636}%
\special{sp}%
%
\special{pn 13}%
\special{pa 3611 1664}%
\special{pa 3640 1679}%
\special{pa 3669 1694}%
\special{pa 3697 1709}%
\special{pa 3726 1724}%
\special{pa 3754 1739}%
\special{pa 3782 1755}%
\special{pa 3810 1771}%
\special{pa 3838 1787}%
\special{pa 3866 1804}%
\special{pa 3893 1821}%
\special{pa 3919 1838}%
\special{pa 3946 1856}%
\special{pa 3972 1875}%
\special{pa 3997 1894}%
\special{pa 4022 1913}%
\special{pa 4046 1934}%
\special{pa 4070 1955}%
\special{pa 4093 1977}%
\special{pa 4116 1999}%
\special{pa 4139 2022}%
\special{pa 4161 2045}%
\special{pa 4183 2069}%
\special{pa 4190 2076}%
\special{sp}%
%
\special{pn 13}%
\special{pa 800 514}%
\special{pa 1448 822}%
\special{da 0.070}%
%
\special{pn 13}%
\special{pa 1529 1340}%
\special{pa 2177 1648}%
\special{da 0.070}%
%
\special{pn 13}%
\special{pa 2461 1470}%
\special{pa 3109 1778}%
\special{da 0.070}%
%
\special{pn 13}%
\special{pa 3473 1672}%
\special{pa 4042 2109}%
\special{da 0.070}%
%
\special{pn 13}%
\special{pa 727 1940}%
\special{pa 1391 1624}%
\special{da 0.070}%
%
\special{pn 13}%
\special{pa 1580 750}%
\special{pa 2245 434}%
\special{da 0.070}%
%
\special{pn 13}%
\special{pa 2469 1081}%
\special{pa 3133 765}%
\special{da 0.070}%
%
\special{pn 13}%
\special{pa 3327 700}%
\special{pa 3992 385}%
\special{da 0.070}%
\put(47.2900,-16.0000){\makebox(0,0)[lb]{$\pi_{1,\kappa}$}}%
\put(6.4600,-7.3300){\makebox(0,0)[lb]{$E^+_1$}}%
\put(18.3000,-8.5000){\makebox(0,0)[lb]{$E^+_2$}}%
\put(29.7100,-10.1600){\makebox(0,0)[lb]{$E^+_3$}}%
\put(35.8700,-7.8100){\makebox(0,0)[lb]{$E^+_4$}}%
\put(7.6000,-17.8600){\makebox(0,0)[lb]{$b^-_1$}}%
\put(6.7900,-21.1000){\makebox(0,0)[lb]{$E^-_1$}}%
\put(19.7500,-18.5100){\makebox(0,0)[lb]{$E^-_2$}}%
\put(28.4100,-19.5600){\makebox(0,0)[lb]{$E^-_3$}}%
\put(38.0000,-22.2000){\makebox(0,0)[lb]{$E^-_4$}}%
%
\special{pn 13}%
\special{pa 638 1357}%
\special{pa 670 1360}%
\special{pa 702 1363}%
\special{pa 734 1366}%
\special{pa 766 1369}%
\special{pa 798 1372}%
\special{pa 830 1375}%
\special{pa 862 1378}%
\special{pa 894 1380}%
\special{pa 926 1383}%
\special{pa 957 1386}%
\special{pa 989 1388}%
\special{pa 1021 1390}%
\special{pa 1053 1393}%
\special{pa 1085 1395}%
\special{pa 1117 1397}%
\special{pa 1149 1398}%
\special{pa 1181 1400}%
\special{pa 1213 1402}%
\special{pa 1245 1403}%
\special{pa 1277 1404}%
\special{pa 1309 1405}%
\special{pa 1341 1406}%
\special{pa 1373 1406}%
\special{pa 1405 1407}%
\special{pa 1437 1407}%
\special{pa 1469 1406}%
\special{pa 1501 1406}%
\special{pa 1533 1405}%
\special{pa 1565 1404}%
\special{pa 1597 1403}%
\special{pa 1629 1401}%
\special{pa 1661 1400}%
\special{pa 1693 1398}%
\special{pa 1715 1397}%
\special{sp 0.070}%
%
\special{pn 13}%
\special{pa 2704 1421}%
\special{pa 2673 1436}%
\special{pa 2641 1451}%
\special{pa 2610 1465}%
\special{pa 2578 1479}%
\special{pa 2547 1492}%
\special{pa 2516 1504}%
\special{pa 2485 1516}%
\special{pa 2454 1525}%
\special{pa 2423 1534}%
\special{pa 2392 1540}%
\special{pa 2361 1545}%
\special{pa 2331 1548}%
\special{pa 2300 1549}%
\special{pa 2270 1547}%
\special{pa 2240 1543}%
\special{pa 2210 1536}%
\special{pa 2180 1526}%
\special{pa 2151 1515}%
\special{pa 2122 1502}%
\special{pa 2092 1488}%
\special{pa 2063 1472}%
\special{pa 2034 1456}%
\special{pa 2015 1446}%
\special{sp 0.070}%
%
\special{pn 13}%
\special{pa 2015 1446}%
\special{pa 2015 1446}%
\special{pa 2015 1446}%
\special{pa 2015 1446}%
\special{pa 2015 1446}%
\special{pa 2015 1446}%
\special{sp}%
%
\special{pn 13}%
\special{pa 2015 1446}%
\special{pa 2015 1446}%
\special{pa 2015 1446}%
\special{pa 2015 1446}%
\special{pa 2015 1446}%
\special{pa 2015 1446}%
\special{sp}%
\put(24.2000,-2.8000){\makebox(0,0)[lb]{$Q$}}%
%
\special{pn 20}%
\special{sh 1}%
\special{ar 2517 336 10 10 0  6.28318530717959E+0000}%
\special{sh 1}%
\special{ar 2517 336 10 10 0  6.28318530717959E+0000}%
%
\special{pn 8}%
\special{pa 638 2782}%
\special{pa 2282 2790}%
\special{fp}%
\special{sh 1}%
\special{pa 2282 2790}%
\special{pa 2215 2770}%
\special{pa 2229 2790}%
\special{pa 2215 2810}%
\special{pa 2282 2790}%
\special{fp}%
%
\special{pn 8}%
\special{pa 630 2450}%
\special{pa 638 1203}%
\special{fp}%
\special{sh 1}%
\special{pa 638 1203}%
\special{pa 618 1270}%
\special{pa 638 1256}%
\special{pa 658 1270}%
\special{pa 638 1203}%
\special{fp}%
%
\special{pn 13}%
\special{pa 2793 1446}%
\special{pa 2825 1450}%
\special{pa 2857 1455}%
\special{pa 2889 1459}%
\special{pa 2921 1464}%
\special{pa 2952 1469}%
\special{pa 2984 1475}%
\special{pa 3015 1481}%
\special{pa 3047 1488}%
\special{pa 3077 1496}%
\special{pa 3108 1505}%
\special{pa 3138 1515}%
\special{pa 3168 1526}%
\special{pa 3198 1538}%
\special{pa 3227 1551}%
\special{pa 3246 1559}%
\special{sp 0.070}%
%
\special{pn 13}%
\special{pa 3246 1559}%
\special{pa 3246 1559}%
\special{pa 3246 1559}%
\special{pa 3246 1559}%
\special{pa 3246 1559}%
\special{pa 3246 1559}%
\special{pa 3246 1559}%
\special{pa 3246 1559}%
\special{pa 3246 1559}%
\special{pa 3246 1559}%
\special{pa 3246 1559}%
\special{pa 3246 1559}%
\special{pa 3246 1559}%
\special{pa 3246 1559}%
\special{pa 3246 1559}%
\special{pa 3246 1559}%
\special{sp 0.070}%
%
\special{pn 13}%
\special{pa 3246 1559}%
\special{pa 3246 1559}%
\special{pa 3246 1559}%
\special{pa 3246 1559}%
\special{pa 3246 1559}%
\special{pa 3246 1559}%
\special{pa 3246 1559}%
\special{pa 3246 1559}%
\special{pa 3246 1559}%
\special{pa 3246 1559}%
\special{pa 3246 1559}%
\special{sp 0.070}%
%
\special{pn 13}%
\special{pa 3578 1924}%
\special{pa 3612 1927}%
\special{pa 3646 1929}%
\special{pa 3681 1932}%
\special{pa 3714 1935}%
\special{pa 3748 1939}%
\special{pa 3781 1944}%
\special{pa 3814 1949}%
\special{pa 3846 1955}%
\special{pa 3877 1962}%
\special{pa 3908 1971}%
\special{pa 3938 1980}%
\special{pa 3967 1992}%
\special{pa 3995 2004}%
\special{pa 4022 2019}%
\special{pa 4048 2035}%
\special{pa 4073 2053}%
\special{pa 4097 2074}%
\special{pa 4120 2096}%
\special{pa 4142 2119}%
\special{pa 4163 2143}%
\special{pa 4184 2168}%
\special{pa 4202 2191}%
\special{sp 0.070}%
\put(7.1900,-15.3500){\makebox(0,0)[lb]{$(E')_1^{-}$}}%
%
\special{pn 13}%
\special{pa 3295 1608}%
\special{pa 3303 1641}%
\special{pa 3311 1674}%
\special{pa 3321 1706}%
\special{pa 3332 1737}%
\special{pa 3346 1766}%
\special{pa 3362 1792}%
\special{pa 3381 1815}%
\special{pa 3404 1835}%
\special{pa 3431 1853}%
\special{pa 3459 1869}%
\special{pa 3489 1883}%
\special{pa 3506 1891}%
\special{sp 0.070}%
%
\special{pn 13}%
\special{pa 3506 1891}%
\special{pa 3506 1891}%
\special{pa 3506 1891}%
\special{pa 3506 1891}%
\special{pa 3506 1891}%
\special{pa 3506 1891}%
\special{pa 3506 1891}%
\special{pa 3506 1891}%
\special{pa 3506 1891}%
\special{pa 3506 1891}%
\special{pa 3506 1891}%
\special{pa 3506 1891}%
\special{pa 3506 1891}%
\special{sp 0.070}%
\put(20.7200,-11.0500){\makebox(0,0)[lb]{$L$}}%
\put(37.9000,-17.3000){\makebox(0,0)[lb]{$L$}}%
%
\special{pn 13}%
\special{pa 3506 1891}%
\special{pa 3506 1891}%
\special{pa 3506 1891}%
\special{pa 3506 1891}%
\special{pa 3506 1891}%
\special{pa 3506 1891}%
\special{pa 3506 1891}%
\special{pa 3506 1891}%
\special{pa 3506 1891}%
\special{pa 3506 1891}%
\special{pa 3506 1891}%
\special{pa 3506 1891}%
\special{pa 3506 1891}%
\special{sp 0.070}%
\put(9.5400,-32.3600){\makebox(0,0)[lb]{$-K_{S} \sim Y = 2 C_0+F'_1+F'_2+F'_3+F'_4$}}%
\put(10.0300,-34.7100){\makebox(0,0)[lb]{$S_{\boldsymbol{\kappa}}\simeq \overline{M^1_1({\boldsymbol{\kappa}})}$, \quad $S_{\boldsymbol{\kappa}} \setminus Y \simeq M^1_1({\boldsymbol{\kappa}})$}}%
%
\special{pn 13}%
\special{sh 1}%
\special{ar 1108 1753 10 10 0  6.28318530717959E+0000}%
\special{sh 1}%
\special{ar 1772 1470 10 10 0  6.28318530717959E+0000}%
\special{sh 1}%
\special{ar 2760 1608 10 10 0  6.28318530717959E+0000}%
\special{sh 1}%
\special{ar 3838 1964 10 10 0  6.28318530717959E+0000}%
\special{sh 1}%
\special{ar 3570 1737 10 10 0  6.28318530717959E+0000}%
\special{sh 1}%
\special{ar 3562 1737 10 10 0  6.28318530717959E+0000}%
%
\special{pn 13}%
\special{pa 3506 1891}%
\special{pa 3506 1891}%
\special{pa 3506 1891}%
\special{pa 3506 1891}%
\special{pa 3506 1891}%
\special{pa 3506 1891}%
\special{pa 3506 1891}%
\special{pa 3506 1891}%
\special{pa 3506 1891}%
\special{pa 3506 1891}%
\special{pa 3506 1891}%
\special{pa 3506 1891}%
\special{pa 3506 1891}%
\special{sp 0.070}%
%
\special{pn 13}%
\special{pa 3506 1891}%
\special{pa 3506 1891}%
\special{pa 3506 1891}%
\special{pa 3506 1891}%
\special{pa 3506 1891}%
\special{pa 3506 1891}%
\special{pa 3506 1891}%
\special{pa 3506 1891}%
\special{pa 3506 1891}%
\special{pa 3506 1891}%
\special{pa 3506 1891}%
\special{pa 3506 1891}%
\special{pa 3506 1891}%
\special{sp 0.070}%
%
\special{pn 13}%
\special{pa 3506 1891}%
\special{pa 3506 1891}%
\special{pa 3506 1891}%
\special{pa 3506 1891}%
\special{pa 3506 1891}%
\special{pa 3506 1891}%
\special{pa 3506 1891}%
\special{pa 3506 1891}%
\special{pa 3506 1891}%
\special{pa 3506 1891}%
\special{pa 3506 1891}%
\special{pa 3506 1891}%
\special{pa 3506 1891}%
\special{sp 0.070}%
%
\special{pn 13}%
\special{pa 3506 1891}%
\special{pa 3506 1891}%
\special{pa 3506 1891}%
\special{pa 3506 1891}%
\special{pa 3506 1891}%
\special{pa 3506 1891}%
\special{pa 3506 1891}%
\special{pa 3506 1891}%
\special{pa 3506 1891}%
\special{pa 3506 1891}%
\special{pa 3506 1891}%
\special{pa 3506 1891}%
\special{pa 3506 1891}%
\special{sp 0.070}%
%
\special{pn 13}%
\special{pa 3506 1891}%
\special{pa 3506 1891}%
\special{pa 3506 1891}%
\special{pa 3506 1891}%
\special{pa 3506 1891}%
\special{pa 3506 1891}%
\special{pa 3506 1891}%
\special{pa 3506 1891}%
\special{pa 3506 1891}%
\special{pa 3506 1891}%
\special{pa 3506 1891}%
\special{pa 3506 1891}%
\special{pa 3506 1891}%
\special{sp 0.070}%
%
\special{pn 13}%
\special{pa 3506 1891}%
\special{pa 3506 1891}%
\special{pa 3506 1891}%
\special{pa 3506 1891}%
\special{pa 3506 1891}%
\special{pa 3506 1891}%
\special{pa 3506 1891}%
\special{pa 3506 1891}%
\special{pa 3506 1891}%
\special{pa 3506 1891}%
\special{pa 3506 1891}%
\special{pa 3506 1891}%
\special{pa 3506 1891}%
\special{sp 0.070}%
%
\special{pn 13}%
\special{pa 3506 1891}%
\special{pa 3506 1891}%
\special{pa 3506 1891}%
\special{pa 3506 1891}%
\special{pa 3506 1891}%
\special{pa 3506 1891}%
\special{pa 3506 1891}%
\special{pa 3506 1891}%
\special{pa 3506 1891}%
\special{pa 3506 1891}%
\special{pa 3506 1891}%
\special{pa 3506 1891}%
\special{pa 3506 1891}%
\special{sp 0.070}%
\put(28.8200,-13.9700){\makebox(0,0)[lb]{$(E')_1^{-}$}}%
%
\special{pn 13}%
\special{pa 1821 1413}%
\special{pa 1926 1421}%
\special{da 0.070}%
%
\special{pn 13}%
\special{pa 2360 40}%
\special{pa 2350 2430}%
\special{fp}%
\special{pa 2360 2450}%
\special{pa 2350 2450}%
\special{fp}%
\put(21.9000,-14.1000){\makebox(0,0)[lb]{$F$}}%
\put(9.5000,-39.9000){\makebox(0,0)[lb]{$F \cdot L = 1 $}}%
%
\special{pn 4}%
\special{sh 1}%
\special{ar 2350 920 10 10 0  6.28318530717959E+0000}%
\special{sh 1}%
\special{ar 2350 920 10 10 0  6.28318530717959E+0000}%
\put(9.5000,-37.0000){\makebox(0,0)[lb]{$L \sim C_1 + F - E_1^{-}-E_2^--E_3^--E_4^-$}}%
\end{picture}%

\end{center}
\caption{8 points blowing ups of Hirzebruch surface $\Sigma_2$}
\label{fig:b-hirzebruch}
\end{figure}

We have the following two fibrations
\begin{equation}\label{eq:twofib}
\begin{CD} 
\overline{M_1^1(\boldsymbol{\kappa})} & \quad \stackrel{\pi_{2, \boldsymbol{\kappa}}}{\longrightarrow} & \quad 
\pp^1    \\
 @V \pi_{1, \boldsymbol{\kappa}} VV     \\
\pp^1&    &  \\  
\end{CD}
\end{equation}
where $\pi_{1, \boldsymbol{\kappa}}$, $\pi_{2, \boldsymbol{\kappa}}$ 
are corresponding to the linear systems $|F|$ and $|L|=
|C_1+F - E_{1}^{-} - E_{2}^{-}-E_{3}^{-} - E_{4}^{-}|$ respectively.  
These  give two different Lagrangian fibrations on the moduli space $
M_1^1(\boldsymbol{\kappa})$.

It is interesting to remark that the morphism $\pi_{2, \boldsymbol{\kappa}}$ 
can be identified with the apparent singularity map 
$\pi_{1, \boldsymbol{\kappa'}}$ for different data $\boldsymbol{\kappa}'$.   In fact, contracting the $8$ exceptional curves 
$(E'_i)^{+}, E_i^{-}$, we obain 
the morphism  $\overline{M_1^1(\boldsymbol{\kappa})} \rightarrow \Sigma_2$, and then the points of blowing ups are corresponding to 
$(b'_i)^{+}, b_i^{-}$ on the fiber $F_i$ of the natural fibration 
of $\Sigma_2 \rightarrow \pp^1$.

We summarize the results. 
\begin{proposition}\label{prop:geometry}
The $q$-fibration and $Q$-fibration in Section \ref{sec-transverse} 
can be identified with 
the maps $\pi_{1, \boldsymbol{\kappa}}$, $\pi_{2, \boldsymbol{\kappa}}$ 
in (\ref{eq:twofib}) respectively.  The general fibers of  $\pi_{1, \boldsymbol{\kappa}}$ and  $\pi_{2, \boldsymbol{\kappa}}$ are given by $F$ and $L$ respectively and 
they are strongly transversal, that is, 
$
F \cdot L = 1. 
$
\end{proposition}

Next we vary the parameter $\boldsymbol{\kappa}$ and consider the 
B\"acklund transformations acting on the family of the moduli spaces.  
From \cite{InabaIwasakiSaito,InabaIwasakiSaito2}, after fixing weights
$\boldsymbol{\alpha}$, we get a smooth fibration 
$\boldsymbol{\kappa}:M _1^{1,\boldsymbol{\alpha}}\to\mathbb C^4$ with fiber $M _1^{1,\boldsymbol{\alpha}}(\boldsymbol{\kappa})$. The classical group of B\"{a}cklund transformations is an equivariant
(with respect to $\boldsymbol{\kappa}$-projection) group of birational transformations
(that preserves the isomonodromy flow when we consider $t$ as a variable).
In restriction to fibers $M _1^{1,\boldsymbol{\alpha}}(\boldsymbol{\kappa})$
with Kostov-generic $\boldsymbol{\kappa}$, B\"{a}cklund transformations are biregular.
The restriction of the B\"{a}cklund transformations group to the action on 
the parameter space $\boldsymbol{\kappa}$ is faithfull and its image is an
affine reflection group, an affine Weyl group of type $D_4$.
Let us describe the generators. 

Firstly, one can switch the parabolic structure over $t_i$
to the eigendirection of the other eigenvalue $-\frac{\kappa_i}{2}$. 
By using  coordinates $(q, p)$ for a suitable Zariski open set of the moduli 
spaces, we can describe $4$ generators as follows:
$$
\left\{\begin{matrix}
s_1:(\kappa_1,\kappa_2,\kappa_3,\kappa_4,q,p)\mapsto
(-\kappa_1,\kappa_2,\kappa_3,\kappa_4,q,p-\frac{\kappa_0}{q})\\
s_2:(\kappa_1,\kappa_2,\kappa_3,\kappa_4,q,p)\mapsto
(\kappa_1,-\kappa_2,\kappa_3,\kappa_4,q,p-\frac{\kappa_0}{q-1})\\
s_3:(\kappa_1,\kappa_2,\kappa_3,\kappa_4,q,p)\mapsto
(\kappa_1,\kappa_2,-\kappa_3,\kappa_4,q,p-\frac{\kappa_0}{q-t})\\
s_4:(\kappa_1,\kappa_2,\kappa_3,\kappa_4,q,p)\mapsto
(\kappa_1,\kappa_2,\kappa_3,-\kappa_4,q,p)\\
\end{matrix}\right.
$$
One can next permute the poles of the connection by a fractional linear $x$-transformation
in such a way that the cross-ratio
$t$ is preserved (we skip here the permutations that do not preserve $t$ parameter).
$$
\left\{\begin{matrix}
r_{(12)(34)}:(\kappa_1,\kappa_2,\kappa_3,\kappa_4,q,p)\mapsto
(\kappa_2,\kappa_1,\kappa_4,\kappa_3,t\frac{q-1}{q-t},-(q-t)\frac{(q-t)p+\kappa_0}{t(t-1)})\\
r_{(13)(24)}:(\kappa_1,\kappa_2,\kappa_3,\kappa_4,q,p)\mapsto
(\kappa_3,\kappa_4,\kappa_1,\kappa_2,\frac{q-t}{q-1},(q-1)\frac{(q-1)p+\kappa_0}{t-1})\\
r_{(14)(23)}:(\kappa_1,\kappa_2,\kappa_3,\kappa_4,q,p)\mapsto
(\kappa_4,\kappa_3,\kappa_2,\kappa_1,\frac{t}{q},-q\frac{qp+\kappa_0}{t})\\
\end{matrix}\right.
$$
One can also apply an even number of elementary transformations centered
at parabolics. This has the effect to shift $\boldsymbol{\kappa}$ parameters
by integers. We skip the formula of generators which is much too huge;
we will describe them in another way just below.
By the way, we obtain the group of Schlesinger transformations. So far,
the transformations come from geometric transformations on parabolic connections.

Finally, the larger B\"{a}cklund transformation group is generated by the transformations
$r_{(ij)(kl)}$ and $s_i$ above and the extra Okamoto involution:
\begin{equation}\label{eq:ok-inv}
s_0:(\kappa_1,\kappa_2,\kappa_3,\kappa_4,q,p)\mapsto
(\kappa_1+\kappa_0,\kappa_2+\kappa_0,\kappa_3+\kappa_0,\kappa_4+\kappa_0,q+\frac{\kappa_0}{p},p)
\end{equation}
The geometric nature (even the Galois group) of the connection is not preserved.
This involution exchanges finite and
infinite monodromy, reducible and irreducible monodromy, and
real monodromy groups $SL(2,\rr )$ and $SU(2)$. 
The first author and S. Cantat have described the
action of these on the Betti moduli spaces in Appendix B of \cite{CantatLoray}. 

We have relations
$$
s_i^2=1,\ \ \ s_is_j=s_js_i\ \text{for}\ i,j\not=0\ \ \ \text{and}\ \ \ s_0s_is_0=s_is_0s_i
$$
$$
r_\sigma^2=1 \ \ \ \text{and}\ \ \ r_\sigma s_i=s_jr_\sigma\ \text{for}\ \sigma=(ij)(kl)
$$
The elementary transformations can be derived by combinations like:
$$
r_{(12)(34)}s_3s_4s_0s_1s_2s_0:(\kappa_1,\kappa_2,\kappa_3,\kappa_4)\mapsto
(\kappa_1+1,\kappa_2+1,\kappa_3,\kappa_4)
$$
(we omit the huge formula in $p$ and $q$).

Our main remark of the section is that the Okamoto 
transformation $s_0$ exchanges the two fibrations. Precisely,
recall that the targets of the two maps $\Upsilon$ and $\Phi$
are canonically identified as $\Pp$ (see Sections \ref{sec-unstablezone} and \ref{sec-stablezone}).
After projection $\Pp\to\pp^1$ (identifying pair-wise the non separated points)
we respectively get the two maps $ \pi_{1, \boldsymbol{\kappa}}, \pi_{2, \boldsymbol{\kappa}}$ or $q, Q:M _1^{1,\boldsymbol{\alpha}}\to\pp^1$ computed above (here we consider $\boldsymbol{\kappa}$ as variables). 
Comparing (\ref{eq:okamoto}) with  (\ref{eq:ok-inv}), 
one can then check that the $Q$-map factors as 
\begin{equation}\label{eq:q-Q}
Q=q\circ s_0.
\end{equation}
Since $s_0$ is an involution, we also get
$$
q=Q\circ s_0.
$$
A similar fact was already observed for $\mathrm{SL}(2,\mathbb C)$-connections
on the trivial bundle $\Oo\oplus\Oo$ by Arinkin-Lysenko in \cite{ArinkinLysenko2}
section $8$ and in \cite{Loray}. More precisely, following \cite{ArinkinLysenko2}, 
the two maps $\pi_{i, \boldsymbol{\kappa}}$ above, $i=1,2$, glue together to define a proper morphism
$$ \pi_{1, \boldsymbol{\kappa}}\times \pi_{2, \boldsymbol{\kappa}}\ :\ \overline{M _1^{1,\boldsymbol{\alpha}}}\to\pp^1\times\pp^1$$
which is just the blow-up of $8$ points along the diagonal. More precisely,
the  coordinates 
$(x,y)$ on $\pp^1\times\pp^1$ are given by the two fibrations
$$x=q\ \ \ \text{and}\ \ \ y=q+\frac{\kappa_0}{p}.$$
The usual symplectic form is given by 
$$dp\wedge dq=\kappa_0\frac{dx\wedge dy}{(x-y)^2}.$$
The $8$ points to blow-up are the $4$ ordinary points
$$(x,y)=(0,0),\ \ \ (1,1),\ \ \ (t,t)\ \ \ \text{and}\ \ \ (\infty,\infty)$$
along the diagonal and next the $4$ infinitesimal points over, given by
the respective slopes 
$$\frac{dy}{dx}=1+\frac{\kappa_0}{\kappa_1},\ \ \ 1+\frac{\kappa_0}{\kappa_2},\ \ \ 1+\frac{\kappa_0}{\kappa_3}\ \ \ 
\text{and}\ \ \ \frac{\kappa_4}{\kappa_0+\kappa_4}.$$
Indeed, the blow-ups are exactly those ones needed to desingularize the alternate parabolic fibration
$$Q'=q+\frac{1-\kappa_0}{p-\frac{\kappa_1}{q}-\frac{\kappa_2}{q-1}-\frac{\kappa_3}{q-t}}\ =\ x+\frac{1-\kappa_0}{\frac{\kappa_0}{y-x}-\frac{\kappa_1}{x}-\frac{\kappa_2}{x-1}-\frac{\kappa_3}{x-t}}.$$
 The anti-canonical divisor 
$Y=2 C_0+F_1'+F_2'+F_3'+F_4'$ is therefore defined by
the strict transform $C_0$ of the diagonal and the strict transforms
$F_i'$ of the $F_i$'s.
We note that the fibration given by the dual coordinate $p=\frac{\kappa_0}{y-x}$ (common for both $x$ and $y$ fibrations)
is simply given in this picture by the fibration
$$dp=0\ \ \ \Leftrightarrow\ \ \ dx=dy.$$
Like in Arinkin-Lysenko's picture, the anti-canonical divisor 
$Y=2 C_0+F_1'+F_2'+F_3'+F_4'$ at infinity
is defined by the strict transforms of the diagonal, $C_0$, and the $4$ exceptional divisors $F_i'$ produced by firstly blowing-up the $4$ ordinary points.
Last, but not least, the Okamoto symmetry is given in this picture
by 
$$
s_0:\ \ \ \left\{\begin{matrix} 
\kappa_i&\mapsto&\kappa_i+\kappa_0\ \text{for}\ i=1,2,3,4,\\
\kappa_0&\mapsto&-\kappa_0\hfill\\
(x,y)&\mapsto&(y,x)\hfill
\end{matrix}\right.
$$
This provides an alternate and nice description 
of our moduli space.

As noticed in Section \ref{sec-unstablezone}, there are $8$ unstable zone for the weights 
$\boldsymbol{\alpha}$, one of which giving the $q$-fibration. The other ones
give other cyclic vectors and thus other fibrations. They can be deduce 
from $q$ after applying an even number of elementary transformations
at the parabolics $P_i$. This is also given by B\"{a}cklund transformations.
For instance
$$
r_{(12)(34)}s_0s_1s_2s_0:(\kappa_1,\kappa_2,\kappa_3,\kappa_4,q,p)\mapsto
(1-\kappa_1,1-\kappa_2,\kappa_3,\kappa_4,q',p') .
$$
where 
$$
q'=t(q-1)(q-t)\frac{p^2+\left(\frac{1-\kappa_1-\kappa_2}{q-1}-\frac{\kappa_3}{q-t}\right)p+\frac{\kappa_0(\kappa_0+\kappa_4)}{(q-1)(q-t)}}{((q-t)p+\kappa_0+\kappa_4)((q-t)p+\kappa_0)}
$$
and
$$
p'=-\frac{((q-t)p+\kappa_0+\kappa_4)((q-t)p+\kappa_0)}{t(t-1)p}.
$$
Here, $q'$ gives the parabolic fibration corresponding to the choice $r_1^+,r_2^+,r_3^-,r_4^-$;
this is one of case (c) of the unstable zone.

There are $16$ natural choices for the parabolic
structure, corresponding to a choice of one of the two eigenvalues
at each point. But switching the parabolic structure over $t_i$ is given by the
action of the symmetry $s_i$, $i=1,2,3,4$. So the $16$ parabolic fibrations
are all obtained from the $Q$-fibration after applying an element of the $16$-order group 
generated by the $s_i$, $i=1,2,3,4$. For instance, switching for the other parabolic structure
$P_i'$ defined by the $r_i^+$ eigenspace, we get the fibration defined by 
$$
Q'=Q\circ s_1s_2s_3s_4=q+\frac{1-\kappa_0}{p-\frac{\kappa_1}{q}-\frac{\kappa_2}{q-1}-\frac{\kappa_3}{q-t}}
$$
So the involution $s_1s_2s_3s_4$ exchanges the two parabolic fibrations $Q$ and $Q'$.

Among all affine $\mathbb A^1$-fibrations over $\Pp$ that can be deduced on our moduli space 
by applying B\"{a}cklund transformations (there are infinitely many)
on 
the $q$-fibration, 
the $16$ ones above play a special role:


\begin{proposition}
The $16$ parabolic fibrations above are the unique affine $\mathbb A^1$-fibrations on $M^1_1(\boldsymbol{\kappa})$ that are strongly transversal to the $q$-fibration and that compactify as $\pp^1$-fibrations in the natural compactification $\overline{M^1_1(\boldsymbol{\kappa})}$.  
\end{proposition}

\begin{proof} Recall that the moduli space $M^1_1(\boldsymbol{\kappa})$ can be obtained by removing $Y_{red} = C_0 + F'_1+F'_2+F'_3+F'_4$ 
from the compactification 
$\overline{M^1_1(\boldsymbol{\kappa})}$ (see Figure \ref{fig:b-hirzebruch}). Moreover $\overline{M^1_1(\boldsymbol{\kappa})}$ is obtained by $8$ blowing ups at $\{ b_i^{\pm} \}_{1 \leq i \leq 4}$ of $\Sigma_2 
\rightarrow \pp^1$ and the $q$-fibration $\pi_{1,\boldsymbol{\kappa}}:\overline{M^1_1(\boldsymbol{\kappa})} 
\rightarrow \pp^1$ is obtained by the linear system $|F|$. 
Let $L'$ be the divisor class of a general fiber of a fibration 
strongly transversal to the $q$-fibration. Since $L' \cdot F =1$ and the linear system $|L'|$ is base point free by the assumption, 
one can see that $L'$ can
be written as  
$$
L'= C_1 + n F - \sum_{i=1}^{4} a_i E_i^{+} - \sum_{i=1}^{4} b_i E_i^{-}, \ \  n \geq 0.  
$$
Note that $F \sim F'_i + E_i^{+} + E_i^{-}$ and $F'_i \cdot E_i^{\pm} = 1$.  Since $L'$ is numerically effective and 
$F \cdot F'_i = F \cdot E_i^{\pm} = C_1 \cdot E_i^{\pm}=0$,  $C_1 \cdot F = C_1 \cdot F'_i = 1$, $(E_i^{\pm})^2 = -1$, we see that 
$$
L' \cdot F'_i = 1  - (a_i + b_i) \geq 0, \ \ 
L' \cdot E_i^{+} = a_i \geq 0, \ \ L' \cdot E_i^{-} = b_i \geq 0.  
$$
Hence $0\leq a_i,b_i \leq 1, 0 \leq a_i + b_i \leq 1$, or $(a_i, b_i) 
= (1, 0), (0,1), (0,0)$.   

On the other hand, since the general fiber  of such a fibration 
$ \overline{M^1_1(\boldsymbol{\kappa})} \rightarrow \Pp$ is $\pp^1$, 
if we require the restriction of this fibration 
to $M^1_1(\boldsymbol{\kappa})$ to be an affine $\mathbb A^1$-fibration, 
we see that $L' \cdot Y_{red} = 1$.  
Since $C_1 \cdot C_0=0$, $C_0 \cdot E_i^{\pm} =0$, we see that $ L' \cdot C_0= n F \cdot C_0 = n$.  Then again by  $L' \cdot F'_i = 1  - (a_i + b_i) \geq 0$, 
\begin{equation}
1 = L' \cdot Y_{red} = L' \cdot (C_0 + F'_1+F'_2+F'_3+F'_4) = n 
+ \sum_{i=1}^4(1-(a_i + b_i))= n+4 -\sum_{i=1}^{4}(a_i + b_i)  
\end{equation}
Hence $ \sum_{i=1}^4(a_i + b_i)= n+3$. Note that $ \sum_{i=1}^4(a_i^2 + b_i^2)= \sum_{i=1}^4 (a_i + b_i) = n+3 $.  
Moreover $L'^2 = 0$ implies that
$$
0  = C_1^2 + 2 n C_1 \cdot F - \sum_{i=1}^{4} (a_i^2 + b_i^2) =   
2+ 2n - \sum_{i=1}^{4}  (a_i^2 + b_i^2) = 2n+2 - n-3 =n-1.  
$$
Hence we have $n=1$ and $(a_i, b_i) = (1, 0)$ or $(0, 1)$  for all $i$, $1 \leq i \leq 4$. 
For each choice of $\boldsymbol{\sigma}=(\sigma_1, \sigma_2, \sigma_3, \sigma_4) \in \{+, -\}^{4}$, we 
can consider the divisor class 
$$
L^{\boldsymbol{\sigma}} = C_1+F - E_1^{\sigma_1} - E_2^{\sigma_2}- E_3^{\sigma_3}-E_4^{\sigma_4}.
$$
Then as in (\ref{eq:parafiber}), the linear system 
$|L^{\boldsymbol{\sigma}}|$ defines a morphism 
$
\overline{M^1_1(\boldsymbol{\kappa})} \rightarrow \pp^1
$
which gives an affine $\mathbb A^1$-fibrations on 
$M^1_1(\boldsymbol{\kappa}) = \overline{M^1_1(\boldsymbol{\kappa})} \setminus Y_{red} \rightarrow \pp^1$ 
which is strongly transversal to the $q$-fibration.  
We obtain $16$ different fibrations associated to the 
linear systems $|L^{\boldsymbol{\sigma}}|$ and   
the above consideration shows that the linear systems $|L^{\boldsymbol{\sigma}}|$ are 
the only possible strongly 
transversal fibrations which give affine $\mathbb A^1$-fibrations 
on $M^1_1(\boldsymbol{\kappa})$.  
\end{proof}

Now, for any B\"{a}cklund transformation $s$, one can consider
the fibration defined by $q\circ s$. Since $s$ is biregular in restriction to 
$M^1_1(\boldsymbol{\kappa})$, the resulting $(q\circ s)$-fibration
is again an $\mathbb A^1$-fibration over $\Pp$. We can now prove

\begin{corollary}Among all $\mathbb A^1$-fibration of the form 
$q\circ s$,
only the $16$ ones above are transversal to $q$.
\end{corollary}

\begin{proof}One easily check that the generators $s_i$ and $r_{(ij)(kl)}$ for the B\"{a}cklund transformation group restrict as a biregular transformation of $C_0$ (mind that they are only birational on $\overline{M^1_1(\boldsymbol{\kappa})}$):
the identity for $s_i$ and a Moebius permutation for $r_{(ij)(kl)}$. 
As a consequence, $q\circ s:C_0\to\pp^1$ is $1:1$ and 
the linear system defined by the fibers of $q\circ s$ is base point free, even at infinity. We can apply the proposition above to conclude.
\end{proof}

\begin{remark}There are many $\mathbb A^1$-fibrations 
on $M^1_1(\boldsymbol{\kappa})$ that are transversal to the $q$-fibration. For instance, the $p$-fibration is like this, but its compactification is not base point free: the general fiber 
intersects $Y_{\rm red}$ exactly at $C_0\cap F_4'$.
The previous statement shows in particular that $p\not=q\circ s$ for any B\"{a}cklund transformation $s$.
One can also find examples of $\mathbb A^1$-fibrations transversal to $q$
having arbitrary high intersection
number at the base point at infinity.
However, all examples which are not 
of the form $q\circ s$ seem to have only $2$ special fibers,
not $4$ as happens with the $16$ ones of the statement.
\end{remark}

We end the section by the following

\begin{proposition}
\label{szaboquest}
For general $\boldsymbol{\kappa}$, parabolic and apparent fibrations (whatever the choice of stable 
and unstable zone) are always different.
\end{proposition}

This answer a question raised by S. Szabo.

\begin{proof}
For general $\boldsymbol{\kappa}$, parabolic and apparent fibrations (whatever the choice of stable 
and unstable zone) are always different. Indeed, we explained how these fibrations are related respectively 
to $Q$ and $q$ by Schlesinger transformation (by switching parabolic structure or by applying elementary 
transformations on the connection). We have also shown that $Q$ is related to $q$ by composition by a B\"acklund transformation, namely $s_0$. However, they are not related by a Schlesinger transformation for general $\boldsymbol{\kappa}$. First of all, $s_0$ is not a Schlesinger transformation (the monodromy is not preserved).
Now any B\"acklund transformation exchanging $Q$ and $q$ must be a composite of $s_0$ 
with a B\"acklund transformation commuting with $q$. But we claim that the later one must be element
of the $16$-order group generated by the $s_i$, $i=1,2,3,4$, thus a Schlesinger transformation,
proving the proposition.

Although it might be well-known, the claim can be easily proved as follows. 
Along isomonodromic deformation of a connection, $q(t)$ is solution (as a function 
of the deformation parameter $t$) of the Painlevé VI equation. If a B\"acklund transformation $s$ commutes with $q$, 
then $q(t)$ is a common solution of the Painlev\'e VI equation for two parameters
$\boldsymbol{\kappa}$ and $\boldsymbol{\kappa}'=\boldsymbol{\kappa}\circ s$ (B\"acklund transformation
are symetries of the Painlev\'e equation). Writing down the difference of the two 
Painlev\'e equations, we get the following algebraic relation (not differential anymore)
$$
\left(\frac{\kappa_\infty'^2-\kappa_\infty^2}{2}-\frac{\kappa_0'^2-\kappa_0^2}{2}\frac{t}{q^2}
+\frac{\kappa_1'^2-\kappa_1^2}{2}\frac{t-1}{(q-1)^2}+\frac{\kappa_t^2-\kappa_t'^2}{2}\frac{t(t-1)}{(q-t)^2}\right)=0
$$
For a general connection, $q(t)$ is transcendental and we promtly get $\kappa_i'=\pm\kappa_i$, $i=1,2,3,4$.
Up to the $16$-order group above, we can assume that the action of $s$ on $\boldsymbol{\kappa}$ is trivial.
However, it is well-known that the affine action of the group of B\"acklund transformations on $\boldsymbol{\kappa}$
parameters is faithfull.
\end{proof}

\section{Middle convolution interpretation}
\label{sec-mc}

We point out here a possible explanation for the existence of a symmetry
interchanging the two fibrations, in terms of
Katz's middle convolution operation. When applied to local  systems of
rank $2$ on $\pp ^1$ with $4$ singular points having non-resonant local 
monodromy, the middle convolution operator gives back a new rank $2$ system
with the same $4$ singularities. However, the monodromy transformations are
changed. 

Consider local systems $L$ with
monodromy eigenvalues at $t_i$ denoted by
$f^{\pm}_i$. This notation, intended to coincide with the notation
in the previous parts of the paper, is a first choice of ordering.
Notice however 
that because the elementary transformations interchange
$f^+_i$ and $f^-_i$ this choice isn't a big constraint.

The middle convolution operation depends on a choice
of local system $\beta$ over $\pp ^1\times \pp ^1$ with singularities
on four  horizontal lines, four vertical lines, and the diagonal.
It is given by the $8$ monodromy transformations $(a_1,a_2,a_3,a_4,b_1,b_2,b_3,b_4)$
with $a_i$ corresponding to the horizontal lines and $b_i$ to the
vertical lines. These are subject to the relation
$$
a_1a_2a_3a_4 = b_1b_2b_3b_4
$$
both products being the inverse of the monodromy on the diagonal. 

Now given a rank $2$ local system $L$ whose local monodromy transformations have
eigenvalues denoted $f_i$ and $f'_i$, assume they are nonspecial. 
This change of notation is there because $f_i$ could be either one of
$f_i^+$ or $f_i^-$, in which case $f'_i$ is the other one
(see Remark \ref{badchoice}).
In order to have a convolution
with rank as small as possible, $a_i$ should be the inverse of one
of the eigenvalues; assume that it is 
$$
a_i= (f_i)^{-1}. 
$$

\begin{lemma}
\label{katzmon}
With the above notations,  the middle convolution of $L$ with
$\beta$ is a local system ${\bf mc}_{\beta}(L)$ of rank $2$ with local monodromy eigenvalues
$$
b_i\;\;\;\; \mbox{   and   } \;\;\;\; 
b_if'_i(f_1f_2f_3f_4)/f_i .
$$
\end{lemma}
\begin{proof}
There are by now a large number of possible references
for the middle convolution operation. For the authors' convenience
we follow the notations of \cite{kmca}. 
The local monodromy transformations
fit into the vector denoted $\overset{\rightharpoonup}{g}$
with components
$$
\overset{\rightharpoonup}{g}_i = [f_i]+[f'_i].
$$ 
The multiplicities are 
$m_i(f_i)=m_i(f'_i)=1$. In the notations of \cite{kmca}
we have $a_i=\beta ^{H_i} = f_i^{-1}$,
$b_i = \beta ^{V_i}$, and 
$$
\beta ^T = (a_1a_2a_3a_4)^{-1} = (b_1b_2b_3b_4)^{-1} = f_1f_2f_3f_4.
$$
Then 
$m_i(f_i)=m_i((\beta ^{H_i})^{-1}) =1$. 
The components corresponding to exceptional curves on a blow-up of $\pp ^1\times \pp ^1$
are 
$$
\beta ^{U_i} := \beta ^{H_i}\beta ^{V_i}\beta ^T = b_i(f_1f_2f_3f_4)/f_i .
$$
The number of points is $n=4$ and the initial rank is $r=2$ so the defect is
$$
\delta (\beta , \overset{\rightharpoonup}{g}) = 
(n-2)r -\sum _{i=1}^4 m_i ((\beta ^{H_i})^{-1}) = 0.
$$
This simplifies the formula for the local Katz transformation 
$$
\kappa _i(\beta , \overset{\rightharpoonup}{g})=
[\beta ^{V_i}] + [f'_i\beta ^{U_i}].
$$
In other words, the monodromy eigenvalues are $b_i$ and 
with
$$
\beta ^{U_i}f'_i =  b_if'_i(f_1f_2f_3f_4)/f_i.
$$
\end{proof}

We would like to investigate what this does to the stable and unstable zones,
in the case of finite-order local monodromy where the parabolic weights are
the same as the rational residues of the connection, which are in turn the
angular arguments of the monodromy eigenvalues. 

\begin{proposition}
Fix finite order local monodromy transformations corresponding
to residues ${\bf r}$ and parabolic weights $\alpha$. Assume that they
are nonspecial. Let ${\bf r}'$ and $\alpha '$ denote
the corresponding values after middle convolution discussed in the previous lemma.
The middle convolution operation extends to an operation on the full Hodge moduli
stack of $\alpha$-semistable $\lambda$-connections,
$$
{\bf mc}_{\beta} : \Mm ^{d,\alpha} ({\bf r})\rightarrow 
\Mm ^{d,\alpha '} ({\bf r}').
$$
It preserves the action of $\Gm$, hence preserves the operation $\lim _{u\rightarrow 0}u()$. 
\end{proposition}

We don't do the proof here. One should be able to show that stability is preserved by
saying that direct image preserves harmonic bundles. This is the subject of work in progress
by the third author with R. Donagi and T. Pantev; however it should also be a consequence of
Sabbah's theory of twistor ${\mathcal D}$-modules \cite{Sabbah}. Nonetheless this proposition will be used in the following discussion, meaning that the remainder of the paper is for the moment heuristic. 

The local monodromy eigenvalues can be written 
$$
f^+_i = e^{\sqrt{-1}\theta ^+_i}, \;\;\; f^-_i = e^{\sqrt{-1}\theta ^-_i},
$$
with 
$$
\theta ^{\pm} _i = 2\pi (\mu _i \pm \epsilon _i),
$$
and $0<\epsilon _i < 1/2$. This corresponds to a logarithmic connection whose
residues are $\mu _i \pm \epsilon _i$. 

In our main discussion of foliations on the moduli space, we have 
used the normalization $\deg (E)=1$.
By the Fuchs relation this means 
\begin{equation}
\label{fuchsmu}
\mu _1+\mu _2 +\mu _3 +\mu _4 = -\frac{1}{2}.
\end{equation}
To correspond to a bundle of odd degree, this relation should hold modulo $\zz$.

Now make a choice of which eigenvalue will be used for the  middle convolution
at each point i.e. along each horizontal line of $\pp ^1\times \pp ^1$,
by choosing the identification between $\{ f_i,f'_i\}$ and $\{ f^+_i,f^-_i\}$.
Choose
$f_i:= f^+_i$ for all $i$. 
It
should be stressed that the result will depend on this choice,
see Remark \ref{badchoice} below.

Note that 
$$
f_1f_2f_3f_4 = -e^{2\pi \sqrt{-1}(\epsilon _1+\epsilon _2+\epsilon_3+\epsilon _4)}
$$
because of the relation \eqref{fuchsmu} for the $\mu _i$. 

By Lemma \ref{katzmon} and the relation \eqref{fuchsmu}, 
the eigenvalues of the local monodromy transformation of the
middle convolution ${\bf mc}_{\beta}(L)$ at $t_i$ are 
$$
b_i\mbox{   and   }  
b_i e^{2\pi \sqrt{-1}y_i}
$$
where
$$
y_i = -\frac{1}{2}+ 
\epsilon _1 + \epsilon _2 + \epsilon _3 +\epsilon _4 -2\epsilon _i.
$$
We should choose a labeling of these eigenvalues of
the middle convolution as $c_i^{\pm}$, in such a way
as to correspond to a logarithmic connection of odd degree. 

Let us start off with a local system in one of the unstable zones,
for example 
\begin{equation}
\label{firstzone}
\epsilon _1+\epsilon _2 +\epsilon _3 + \epsilon _4 < 1/2.
\end{equation}
In this case we have 
$$
-1 < y_i < 0.
$$
For this case, set
$$
c^+_i := b_i \mbox{   and   }c^-_i:= b_i e^{2\pi \sqrt{-1}y_i}.
$$
Write $b_i= e^{2\pi \sqrt{-1}z_i}$ and put 
$\mu '_i:= z_i + y_i/2$ and $\epsilon '_i = -y_i/2$. 
Now
$$
c^+_i = e^{2\pi \sqrt{-1}(\mu '_i+\epsilon '_i)},
$$
$$
c^-_i = e^{2\pi \sqrt{-1}(\mu '_i-\epsilon '_i)}.
$$
We have 
$$
\frac{y_1+y_2+y_3+y_4}{2} = -1 + \epsilon _1 + \epsilon _2 +\epsilon _3 + \epsilon _4.
$$
On the other hand,
the equation $b_1b_2b_3b_4 = (f_1f_2f_3f_4)^{-1}$ yields
$$
z_1+z_2+z_3+z_4 =  \frac{1}{2} - (\epsilon _1+\epsilon _2 +\epsilon _3 + \epsilon _4) \mbox{  mod } \zz .
$$
Putting these together we get
$$
\mu '_1+\mu '_2 +\mu '_3 +\mu '_4 
= z_1+z_2+z_3+z_4 + \frac{y_1+y_2+y_3+y_4}{2}
$$
$$ 
=
\left[\frac{1}{2} - (\epsilon _1+\epsilon _2 +\epsilon _3 + \epsilon _4)\right]
+ \left[-1 + \epsilon _1 + \epsilon _2 +\epsilon _3 + \epsilon _4\right] = -\frac{1}{2} 
$$
modulo $\zz$. Therefore the given choice corresponds to a logarithmic connection on
a bundle of odd degree, and indeed modifying some $z_i$ by
an integer allows us to assume the degree is $1$. 
We have 
\begin{eqnarray*}
\epsilon '_1 & = & \frac{1}{4}+ 
\frac{\epsilon _1}{2} - \frac{\epsilon _2}{2} - \frac{\epsilon _3}{2} -\frac{\epsilon _4}{2}
\\
\epsilon '_2 & = & \frac{1}{4}- 
\frac{\epsilon _1}{2} + \frac{\epsilon _2}{2} - \frac{\epsilon _3}{2} -\frac{\epsilon _4}{2}
\\
\epsilon '_3 & = & \frac{1}{4}- 
\frac{\epsilon _1}{2} - \frac{\epsilon _2}{2} + \frac{\epsilon _3}{2} -\frac{\epsilon _4}{2}
\\
\epsilon '_4 & = & \frac{1}{4}- 
\frac{\epsilon _1}{2} - \frac{\epsilon _2}{2} - \frac{\epsilon _3}{2} +\frac{\epsilon _4}{2} .
\end{eqnarray*}
Notice that these are all in the interval $0<\epsilon '_i<1/2$, in view of \eqref{firstzone}.

We can now calculate
$$
\epsilon '_1 + \epsilon '_2 +\epsilon '_3 + \epsilon '_4 
= 1 -  \epsilon _1 -\epsilon _2 -\epsilon _3 -\epsilon _4
$$
so 
$$
1/2 < \epsilon '_1 + \epsilon '_2 +\epsilon '_3 + \epsilon '_4 < 1 <\frac{3}{2}.
$$
Also for example
$$
\epsilon '_1 + \epsilon '_2 -\epsilon '_3 - \epsilon '_4 
=  \epsilon _1 +\epsilon _2 -\epsilon _3 -\epsilon _4
$$
which, in view of the assumption \eqref{firstzone}, gives
$$
-1/2 < \epsilon '_1 + \epsilon '_2 -\epsilon '_3 - \epsilon '_4 < 1/2 .
$$
Similarly for the other conditions of type \eqref{stabc}.

\begin{lemma}
Suppose $L$ is a local system with finite order monodromy, corresponding
to an odd degree logarithmic connection whose residues and parabolic
weights are in the unstable zone (a) i.e. they satisfy \eqref{firstzone}.
Then choosing a rank one local system $\beta$
with $a_i = (f^+_i)^{-1}$, the middle convolution 
${\bf mc}_{\beta}(L)$ has local  monodromy transformations, again of finite
order, lying in the stable zone. 
\end{lemma}
\begin{proof}
The above calculations give \eqref{staba}, \eqref{stabb}, and \eqref{stabc}. 
\end{proof}

\begin{proposition}
The same holds for the other unstable zones: the middle convolution with a
suitably chosen $\beta$ goes into the stable zone. In the other direction,
if $L$ starts off in the stable zone then for a suitably chosen $\beta$
the middle convolution will lie in the unstable zone. 
\end{proposition}

There are $8$ unstable zones in all, types (a), (b) and $6$ zones of type (c).
The calculations are similar to the case (a) treated above. The images by
${\bf mc}$ divide the stable zone up into $8$ sub-regions, which presents a computational
difficulty for going back in the other direction. However, the fact that ${\bf mc}$ is
involutive up to operations of elementary transformations and tensoring with rank $1$
systems (which leave stable the distinction between stable and unstable zones), so
the fact that the unstable zones go to the stable zone implies that the stable zone
goes to the unstable zones. 

\begin{remark}
\label{badchoice}
If, in the example above, we had chosen $f_i= f^+_i$ for $i=1,2,3$ but $f_4=f^-_4$,
then the corresponding choice of $\beta$ would have left ${\bf mc}_{\beta}(L)$  remaining inside the unstable zone.
In general up to the operations of doing pairs of elementary transformations, there
are two distinct choices for $\beta$ and one of them will interchange the two zones. 
\end{remark}

As Dettweiler and Reiter \cite{DettweilerReiterPainleve} point out, 
the middle convolution operation is one of the additional symmetries considered by
Okamoto, although its normalization depends on the choice of $a_i$ and $b_i$.

The middle convolution is obtained by pullback and higher direct image. These operations
preserve the Hodge filtration moduli spaces $\Mm _{\rm Hod}$ when well-defined 
(for example if we assume Kostov genericity). They are compatible with the action of $\Gm$,
so they preserve the limiting operation and hence the foliation by subspaces defined by
looking at what the limit is. Hence, the middle convolution operation preserves
the Higgs limit foliation, and since it exchanges stable and unstable zones, it takes the
apparent singularity foliation to the parabolic structure foliation. 

This gives a partial explanation of the interchanging
phenomenon observed by Arinkin-Lysenko \cite{ArinkinLysenko2} and the first author 
in \cite{Loray} and described in the previous section, 
although it leaves open the question of why it acts trivially
on the quotient space $\Pp$ by the foliation. It might be possible to answer that
by looking more carefully at the direct image operation in the middle convolution,
as applied to parabolic Higgs bundles.  

Using the middle convolution one can reduce the proof of the foliation
conjecture for the stable zone, to the case of the unstable zone which was already
known by \cite{InabaIwasakiSaito}. This gives an alternate method to prove Corollary
\ref{foliationstable}. 

\newpage

\bibliographystyle{amsalpha}

\providecommand{\bysame}{\leavevmode\hbox to3em{\hrulefill}\thinspace}
\providecommand{\MR}{\relax\ifhmode\unskip\space\fi MR }
\providecommand{\MRhref}[2]{%
  \href{http://www.ams.org/mathscinet-getitem?mr=#1}{#2}
}
\providecommand{\href}[2]{#2}
\begin{thebibliography}{}

\end{thebibliography}


\begin{thebibliography}{A}

\bibitem{Aidan}
J. Aidan. Doctoral dissertation, in preparation. 

\bibitem{Arinkin}
D. Arinkin.
Orthogonality of natural sheaves on moduli stacks of $SL(2)$-bundles with 
connections on $\pp ^1$ minus $4$ points.
{\em Selecta Mathematica} {\bf 7} (2001), 213-239.

\bibitem{ArinkinLysenko1}
D. Arinkin, S. Lysenko. On the moduli of $SL(2)$-bundles with connections on
$\pp ^1-\{t_1, \ldots , t_4\}$. {\em Internat. Math. Res. Not.} {\bf 19} (1997), 983–--999.

\bibitem{ArinkinLysenko2}
D. Arinkin, S. Lysenko. Isomorphisms between moduli spaces of $SL(2)$-bundles
with connections on $\pp ^1-\{t_1, \ldots , t_4\}$. 
{\em Math. Res. Lett.} {\bf 4} (1997), 181--–190.

\bibitem{Atiyah}
M. Atiyah. Complex analytic connections in fibre bundles. 
{\em Trans. Amer. Math. Soc.} {\bf 85}
(1957), 181–--207.

\bibitem{Biswas}
I. Biswas. A criterion for the existence of a flat connection on a
parabolic vector bundle. {\em Adv. Geom.} {\bf 2} (2002), 231–--241.

\bibitem{Boalch}
P. Boalch. 
From Klein to Painlev\'e via Fourier, Laplace and Jimbo. 
{\em Proc. London Math. Soc}. {\bf  90} (2005), 167-208.

\bibitem{BoalchQuiver}
P. Boalch. 
Quivers and difference Painlev\'e equations.
{\em Groups and Symmetries: from Neolithic Scots to John McKay}
(J. Harnad, P. Winternitz, eds.)
{\sc CRM Proceedings and Lecture Notes} 
{\bf 47} (2009), 25-52.

\bibitem{BodenYokogawa}
H. Boden, K. Yokogawa. Moduli spaces of parabolic Higgs bundles and parabolic $K(D)$ pairs over smooth curves, I.
{\em Internat. J. Math}. {\bf 7} (1996), 573-598.

\bibitem{Borne}
N. Borne. Sur les repr\'esentations du groupe fondamental d'une vari\'et\'e 
priv\'ee d'un diviseur \`a croisements normaux simples.
{\em Indiana Univ. Math. J.} {\bf 58} (2009), 137-180.

\bibitem{CantatLoray}
S. Cantat, F. Loray. 
Holomorphic dynamics, Painlev\'e VI equation and character varieties.
Preprint \verb+http://arxiv.org/abs/0711.1579+ (2007).

\bibitem{CantatLoray2}
S. Cantat, F. Loray. 
Dynamique sur la vari\'et\'e des caract\`eres et irr\'eductibilit\'e au sens de Malgrange de l'\'equation de Painlev\'e VI.
{\em Ann. Inst. Fourier} {\bf 59} (2009), 2927-2978. 


\bibitem{CrawleyBoevey}
W. Crawley Boevey. 
Indecomposable parabolic bundles and the existence of matrices in prescribed 
conjugacy class closures with product equal to the identity. 
{\em Publ. Math. I.H.E.S.} {\bf 100} (2004), 171-207. 

\bibitem{Deligne}
P. Deligne. {\em Equations diff\'erentielles \`a points singuliers r\'eguliers}.  
{\sc Lecture Notes in Math.} {\bf 163} (1970). 

\bibitem{DettweilerReiterPainleve}
M. Dettweiler, S. Reiter.
Painlev\'e equations and the middle convolution.
{\em Adv. Geom.} {\bf 7} (2007), 317-330. 

\bibitem{DubrovinMazzocco}
B. Dubrovin, M. Mazzocco. Monodromy of certain Painlev´e-VI transcendents
and reflection groups. {\em Invent. Math.} {\bf 141} (2000), 55–--147.

\bibitem{EsnaultViehweg}
H. Esnault, E. Viehweg. Semistable bundles on curves and irreducible representations
of the fundamental group. {\em Algebraic Geometry, Hirzebruch 70 (Warsaw 1998)},
{\sc Contemp. Math.} {\bf 241} (1999), 129-138. 

\bibitem{GPGM}
O. Garcia-Prada, P. Gothen, V. Mu\~noz. 
{\em Betti numbers of the moduli space of rank $3$ parabolic Higgs bundles}.
{\sc Memoirs A.M.S.} {\bf 879} (2007). 

\bibitem{GoldmanToledo}
W. Goldman, D. Toledo. 
Affine cubic surfaces and relative $SL(2)$-character varieties of compact surfaces. 
Preprint \verb+arXiv:1006.3838+ (2010).

\bibitem{Hitchin}
N. Hitchin. 
The self-duality equations on a Riemann surface.
{\em Proc. London Math. Soc}. (3) {\bf 55} (1987),  59-126.

\bibitem{InabaIwasakiSaito}
M. Inaba, K. Iwasaki, M.-H. Saito.
Moduli of stable parabolic connections, Riemann-Hilbert correspondence and geometry of Painlev\'e 
equation of type VI, Part I.
{\em Publ. Res. Inst. Math. Sci}. {\bf 42} (2006), 987-1089.

\bibitem{InabaIwasakiSaito2}
M. Inaba, K. Iwasaki, M.-H. Saito.
Moduli of stable parabolic connections, Riemann-Hilbert correspondence and geometry of Painlev\'e 
equation of type VI, Part II.
Moduli spaces and arithmetic geometry. {\em Adv. Stud. Pure Math.} {\bf 45}, Math. Soc. Japan (2006), 387--432.

\bibitem{IyerSimpson}
J. Iyer, C. Simpson. The Chern character of a parabolic bundle, and a parabolic corollary of Reznikov's theorem.
{\em Progr. in Math.} {\bf 265}, Birk\"auser (2007), 437-483.


\bibitem{JimboMiwa} M. Jimbo and  T. Miwa, 
 Monodromy preserving deformation of linear ordinary differential 
equations with rational coefficients. {II}.,  {\em Physica D}, {\bf 2},  (1981), {407--448}.

\bibitem{Konno}
H. Konno. Construction of the moduli space of stable parabolic Higgs bundles on a Riemann surface. {\em J. Math.
Soc. Japan} {\bf 45} (1993), 253-276.

\bibitem{Loray}
F. Loray. Okamoto symmetry of Painlev\'e VI equation and isomonodromic deformation of Lam\'e connections. {\em Algebraic, analytic and geometric aspects of complex differential equations and their deformations. Painlev\'e hierarchies}, 129--136, 
RIMS K\^{o}ky\^{u}roku Bessatsu, B2, Res. Inst. Math. Sci. (RIMS), Kyoto, 2007. A detailled article is in preparation.



\bibitem{Machu1}
F.-X. Machu. Monodromy of a class of logarithmic connections on an elliptic curve.
{\em S.I.G.M.A.} {\bf 3} (2007), 082, 31pp. 

\bibitem{MaruyamaYokogawa}
M. Maruyama, K. Yokogawa. Moduli of parabolic stable sheaves.
{\em Math. Ann.} {\bf 293} (1992), 77–99.

\bibitem{Mochizuki}
T. Mochizuki.
Asymptotic behaviour of tame harmonic bundles and an application to pure twistor $D$-modules, parts I and II.
{\em Memoirs of the AMS} {\bf 869}-{\bf 870} (2007).

\bibitem{Mochizuki2}
T. Mochizuki. 
Kobayashi-Hitchin correspondence for tame harmonic bundles and an application.
{\em Ast\'erisque} {\bf 309} (2006). 

\bibitem{Nakajima}
H. Nakajima. Hyper-K\"ahler structures on moduli spaces of parabolic Higgs bundles on Riemann surfaces.
Moduli of vector bundles (Sanda, 
Kyoto, 1994), {\em Lect. Notes Pure Appl. Math.} {\bf 179} (1996), 199-208.

\bibitem{Nitsure}
N. Nitsure. Moduli of semistable logarithmic connections. {\em J. Amer. Math. Soc.} {\bf 6} (1993), 597-609. 



\bibitem{NoumiYamada}
M. Noumi, Y. Yamada. A new Lax pair for the sixth Painlevé equation associated with $\hat{\mathfrak{so}}(8)$. {\em Microlocal analysis and complex Fourier analysis}, 238--252, World Sci. Publ., River Edge, NJ, 2002. 
 \verb+http://arxiv.org/abs/math-ph/0203029+ (2002).
 
 
\bibitem{Oblezin}
S. Oblezin. Isomonodromic deformations 
of the $sl(2)$ Fuchsian systems on the Riemann sphere.
{\em Mosc. Math. J.} 5 (2005), no. 2, 415--441, 494--495. 
Preprint arXiv:math-ph/0309048 (2003). 



\bibitem{Okamoto1} 
K. Okamoto.
Sur les feuilletages associ\'es aux \'equations du second
ordre \`a points critiques fixes de P. Painlev\'e, Espaces des
conditions initiales.
{\em Japan. J. Math.} {\bf 5}, (1979), 1-79.

\bibitem{Okamoto2} 
K. Okamoto.
Isomonodromic deformation and Painlev\'e equations
and the Garnier system. 
{\em J. Fac. Sci. Univ. Tokyo, Sect. IA,
Math.} {\bf 33},  (1986), 575-618

\bibitem{Okamoto3}
K. Okamoto.
Studies on the Painlev\'e equations I.
Sixth Painlev\'e equation P{\,}VI, {\em Ann. Mat. Pura Appl.}
(4) {\bf 146} (1987), 337-381.

\bibitem{Sabbah}
C. Sabbah.  Polarizable twistor ${\mathcal D}$-modules.
{\em Ast\'erisque} {\bf 300}, (2005).

\bibitem{SaitoTakebeTerajima} 
M.-H. Saito, T. Takebe, H. Terajima. 
Deformation of Okamoto-Painlev\'e pairs and Painlev\'e equations.
{\em J. Algebraic Geom.} {\bf 11}  (2002),  no. 2, 311--362.

\bibitem{Sakai}  
H. Sakai.
Rational surfaces associated with affine
root systems and geometry of the Painlev\'e equations. 
{\em Comm. Math. Phys.}
{\bf 220} (2001), 165--229.

\bibitem{Schmitt}
A. Schmitt. Projective moduli for Hitchin pairs.
{\em Internat. J. Math.} {\bf 9} (1998), 107–118. 

\bibitem{Seshadri}
C. Seshadri, 
Moduli of $\pi $-vector bundles over an algebraic curve, Questions on Algebraic Varieties (C.I.M.E., III Ciclo, Varenna, 1969), 
{\em Edizioni Cremonese}, Rome (1970) 139-260. 

\bibitem{hbnc}
C. Simpson. Harmonic bundles on noncompact curves. 
{\em J. Amer. Math. Soc.} {\bf 3}  (1990), 713-770.

\bibitem{hbls}
C. Simpson. Higgs bundles and local systems.
{\em Publ. Math. I.H.E.S.} {\bf 75} (1992), 5-95.

\bibitem{kmca}
C. Simpson. Katz's middle convolution algorithm. {\em Pure Appl. Math. Q.}
{\bf 5} (2009), 781–852. 

\bibitem{idsm}
C. Simpson. Iterated destabilizing modifications for vector bundles with connection.
{\em Vector Bundles and Complex Geometry (Conference in Honor of Ramanan, Miraflores 2008)}
{\sc Contemporary Math.} {\bf 522} (210), 183-206.

\bibitem{Szabo}
S. Szabo. Deformations of Fuchsian equations and logarithmic connections.
Arxiv preprint math/0703230 (2007).

\bibitem{Weil}
A. Weil. Generalization de fonctions abeliennes. {\em  J. Math. Pures Appl.} 
{\bf 17} (1938), 47-87.

\bibitem{Yokogawa}
K. Yokogawa. {\em Compactification of moduli of parabolic sheaves and moduli of parabolic Higgs sheaves.}
J. Math. Kyoto Univ. {\bf 33} (1993), 451-504.


\end{thebibliography}

\end{document}